\documentclass[11pt]{article}

\usepackage{bm}
\usepackage{amsfonts}
\usepackage{amssymb,amsmath,amsthm, amsfonts}
\usepackage{mathrsfs}

\def\sqr#1#2{{\vcenter{\vbox{\hrule height.#2pt
              \hbox{\vrule width.#2pt height#1pt \kern#1pt \vrule width.#2pt}
          \hrule height.#2pt}}}}
\def\sqr#1#2{{\vcenter{\vbox{\hrule height.#2pt
              \hbox{\vrule width.#2pt height#1pt \kern#1pt \vrule width.#2pt}
              \hrule height.#2pt}}}}
\def\3n{\negthinspace \negthinspace \negthinspace }
\def\2n{\negthinspace \negthinspace }
\def\1n{\negthinspace }

\def\={\buildrel \triangle \over =}

\def\limsup{\mathop{\overline{\rm lim}}}

\def\esssup{\mathop{\rm esssup}}

\def\max{\mathop{\rm max}}
\def\min{\mathop{\rm min}}
\def\exp{\mathop{\rm exp}}
\def\sup{\mathop{\rm sup}}
\def\inf{\mathop{\rm inf}}

\def\({\Big (}
\def\){\Big )}
\def\[{\Big[}
\def\]{\Big]}
\def\be{\begin{equation}}

\def\ee{\end{equation}}

\def\square#1{\vbox{\hrule\hbox{\vrule height#1%
     \kern#1\vrule}\hrule}}
\def\rectangle#1#2{\vbox{\hrule\hbox{\vrule height#1%
     \kern#2\vrule}\hrule}}

\font\tenbb=msbm10 \font\sevenbb=msbm7 \font\fivebb=msbm5

\newfam\bbfam
\scriptscriptfont\bbfam=\fivebb \textfont\bbfam=\tenbb
\scriptfont\bbfam=\sevenbb

\theoremstyle{definition}
\newtheorem{lemma}{Lemma}[section]
\newtheorem{remark}[lemma]{Remark}
\newtheorem{example}[lemma]{Example}
\newtheorem{theorem}[lemma]{Theorem}

\newtheorem{definition}[lemma]{Definition}
\newtheorem{proposition}[lemma]{Proposition}

\newtheorem{mylemma}{Lemma}

\makeatletter
   
   \@addtoreset{equation}{section}
\makeatother

\begin{document}

\title{General mean-field BSDEs with integrable terminal values \thanks{Juan Li is supported by the NSF of P.R. China (NOs. W2511002, 12031009), the NSF of Shandong Province (NO. ZR2023ZD35), National Key R and D Program of China (NO. 2018YFA0703900), and NSFC-RS (No. 11661130148; NA150344).}}
\author{Weimin Jiang$^{1}$,\,\, Juan Li$^{1,2,}$\footnote{Corresponding authors.},\,\, Yan Shen$^{1,\dag}$\\
{$^1$\small School of Mathematics and Statistics, Shandong University, Weihai, Weihai 264209, P.~R.~China.}\\
{$^2$\small Research Center for Mathematics and Interdisciplinary Sciences, Shandong University,}\\
{\small Qingdao 266237, P.~R.~China.}\\
{\footnotesize{\it E-mails: weiminjiang@mail.sdu.edu.cn,\,\ juanli@sdu.edu.cn,\,\ 202117806@mail.sdu.edu.cn.}}
}

\date{October 11, 2025}
\maketitle

\textbf{Abstract}. This paper investigates $L^{1}$ solutions for mean-field backward stochastic differential equations (MFBSDEs) under different weak assumptions in both one-dimensional and
multi-dimensional settings, whose generator $f(\omega,t,y,z,\mu)$ depends not only on the solution process $(Y,Z)$ but also on the law of $(Y,Z)$.
In the one-dimensional case where $f$ depends on the law of $Y$, we show with the help of a test function method and a localization procedure that such type of equations with an integrable terminal
condition admits an $L^{1}$ solution, when the generator $f(\omega,t,y,z,\mu)$ has a one-sided linear growth in $(y,\mu)$, and an iterated-logarithmically sub-linear growth
in $z$. Furthermore, by leveraging the additional extended monotonicity in $y$ and an iterated-logarithmically uniform continuity in $z$ of the generator
$f(\omega,t,y,z,\mu)$ together with a strengthened nondecreasing condition in $\mu$, we derive a comparison theorem for $L^{1}$ solutions, which immediately leads
to the uniqueness of the $L^{1}$ solutions. Next, we establish the existence and the uniqueness of $L^{1}$ solutions for multi-dimensional mean-field
BSDEs with integrable parameters in which the generator $f(\omega,t,y,z,\mu)$ depends on $\mu=\mathbb{P}_{Y}$ and satisfies a one-sided Osgood condition as well as a general growth condition in $y$, a Lipschitz continuity
as well as a sublinear growth condition in $z$, and a Lipschitz condition in $\mu$. Finally, the solvability of $L^{1}$ solutions for general MFBSDEs is studied, where the generator $f(\omega,t,y,z,\mu)$ depends on both the solution process $(Y,Z)$ and its joint law $\mathbb{P}_{(Y,Z)}$.


\section{Introduction}
\hspace{2em}
Linear backward stochastic differential equations (BSDEs) were first introduced by Bismut \cite{B73} in 1973 as adjoint equations in the framework of Pontryagin's stochastic maximum
 principle for stochastic control problems. Afterwards, the more general nonlinear BSDEs were introduced by Pardoux and Peng \cite{PP90}, for which they proved the
 existence and the uniqueness of adapted solutions with Lipschitz generators and square integrable terminal values. Their equation takes the form:
\begin{equation}\label{eq 1.1}\tag{1.1}
	Y_{t}=\xi+\int_{t}^{T}f(s,Y_{s},Z_{s})ds-\int_{t}^{T}Z_{s}dW_s,\ t\in[0,T],
\end{equation}
where the terminal value $\xi:\Omega\rightarrow\mathbb{R}^{n}$ is an $\mathcal{F}_{T}$-measurable random variable, and the generator (or driver)
 $f: \Omega \times [0, T] \times \mathbb{R}^{n} \times \mathbb{R}^{n \times d} \rightarrow \mathbb{R}^{n}$ is $\mathbb{F}$-progressively measurable
 for all $(y, z) \in \mathbb{R}^{n} \times \mathbb{R}^{n \times d}$. Since then, the theory of BSDEs has attracted significant interest and found extensive
 applications in diverse fields, including nonlinear PDEs, stochastic control, stochastic differential games, and mathematical finance; see, e.g.,
 El Karoui, Peng and Quenez \cite{EP97}, El Karoui and Hamad\`{e}ne \cite{EH03}, Pardoux and Peng \cite{PP92}, Peng \cite{P92}, etc.

Undoubtedly, the existence and uniqueness of solutions remains a fundamental issue in the BSDE theory and its applications. Numerous researchers have
extended the pioneering work of Pardoux and Peng \cite{PP90}. Several researchers have focused on weakening the Lipschitz assumptions on the generator $f$,
such as Briand, Lelepeltier and San Martin \cite{BR07}, Briand and Confortola \cite{BC08}, Kobylanski \cite{K00}, Lepeltier and San Martin \cite{LS97},
Mao \cite{M95}, among others. Different researchers on their turn have focused on relaxing the $L^{2}$-integrability requirements associated with the terminal
value $\xi$, e.g., Briand et al. \cite{BD03}, Fan \cite{F15}, Fan \cite{F16}, and many others. Since the late 1990s, the investigation of the existence
and the uniqueness of solutions that are bounded, as well as those with certain exponential moment orders, has emerged as a significant area of interest
 within the theory of BSDEs. For a more complete state of art, we refer the reader to Briand and Hu \cite{BH06},
 Delbaen, Hu and Bao \cite{DH11}, Kobylanski \cite{K00}, Lepeltier and San Martin \cite{LS98}, Luo and Fan \cite{LF18}, where the
 generator $f$ is of quadratic or super-quadratic growth with respect to $z$. Notably, existing results indicate that one-dimensional BSDEs are often more easy to handle
 than their multi-dimensional counterparts, largely due to the availability of tools such as the comparison theorem, the monotone convergence theorem, and the
 Girsanov transformation.

On the other hand, in 1997, Peng \cite{P97} introduced the concept of $g$-martingales of a square integrable random variable through the solutions of BSDEs,
which can be interpreted as some type of nonlinear martingales. Since classical martingale theory operates within integrable spaces,
this naturally motivates to solve BSDEs with only integrable data. Such equations with only integrable data are generally more challenging
than those with $L^{p}$-integrable parameters ($p>1$). We refer the reader to Briand et al. \cite{BD03}, and Fan \cite{F16}, \cite{F18}, for a
survey on the well-posedness of $L^{1}$ solutions for BSDEs. Notably, Briand et al. \cite{BD03} were the first to provide general existence
and uniqueness results of $L^{1}$ solutions for multi-dimensional BSDEs with a generator $f$ satisfying the criteria of monotonicity and
general growth with respect to $y$, as well as Lipschitz and sub-linear growth conditions with respect to $z$. Recently, using a test function
approach and a localization technique, Fan, Hu and Tang \cite{FH23} obtained the existence of solutions as well as a comparison theorem for
one-dimensional BSDEs with only integrable terminal values, under the condition that the generator $f$ has a one-sided linear growth in $y$ and
a logarithmic sub-linear growth in $z$. Furthermore, Fan, Hu and Tang \cite{FH24} relaxed the assumption of logarithmic sub-linear growth in $z$,
as initially introduced in their previous work of Fan, Hu and Tang \cite{FH23}, to a novel criterion of iterated-logarithmically sub-linear growth with regard to $z$.

Mean-field models serve as a useful tool to describe the asymptotic behavior in large-scale systems. Mean-field stochastic differential equations (MFSDEs),
or McKean-Vlasov equations, which characterize stochastic systems whose evolution is determined by both the microscopic location and the macroscopic
distribution of the particles, were first discussed by Kac \cite{K56} in the middle of the 1950s in the frame of his study of the Boltzmann equation for the particle
density in diluted monatomic gases and in that of the stochastic toy model for the Vlasov kinetic equation of plasma. Over the past two decades, there has been
extensive research on mean-field games and their closely associated mean-field control problems, which were independently
initiated by Lasry and Lions \cite{LL07} and Huang, Malham\'{e} and Caines \cite{HM06}. Such problems explore the limit dynamics of large-scale systems,
where the agents interact with each other in certain symmetric manner. Inspired by them, Buckdahn, Djehiche,
Li and Peng \cite{BD09} studied a class of special mean-field problems using a purely probabilistic methodology. They introduced the so-called mean-field BSDEs. Buckdahn,
Li and Peng \cite{BL09} studied the existence and the uniqueness of solutions of mean-field BSDEs with assuming that
the generator $f$ satisfies a Lipschitz condition and the terminal value $\xi$ is square integrable. Since then, there has been an intensive research
on mean-field SDEs and BSDEs addressing the existence and the uniqueness as well as related PDEs; see, for instance, Buckdahn, Li, Peng and Rainer \cite{BL17},
 Chen, Xing and Zhang \cite{CX20}, Hao, Hu, Tang and Wen \cite{HH22}, Li \cite{L18}, Li, Liang and Zhang \cite{LL18}, Li and Xing \cite{LX22}, etc.

However, to the best of our knowledge, the well-posedness of mean-field BSDEs with integrable parameters remains relatively unexplored in both one-dimensional and
multi-dimensional settings. Motivated by existing work, we focus on the solvability of $L^{1}$ solutions for mean-field BSDEs in both one-dimensional and
multi-dimensional cases, where the generator $f$ depends on the solution $(Y,Z)$ and the distribution $\mathbb{P}_{Y}$ of its $Y$-component. More precisely, the mean-field BSDEs is as follows:
\begin{equation*}\label{eq 1.2}\tag{1.2}
	Y_{t}=\xi+\int_{t}^{T}f(s,Y_{s},Z_{s},\mathbb{P}_{Y_{s}})ds-\int_{t}^{T}Z_{s}dW_{s},\ t\in[0,T],
\end{equation*}
where the terminal value $\xi:\Omega\rightarrow\mathbb{R}^{n}$ is an integrable $\mathcal{F}_{T}$-measurable random variable, and the
generator $f:=f(\omega,t,y,z,\mu):\Omega\times[0,T]\times\mathbb{R}^{n}\times\mathbb{R}^{n\times d}\times\mathcal{P}_{1}(\mathbb{R}^{n})\rightarrow\mathbb{R}^{n}$ is $\mathbb{F}$-progressively
 measurable, for all fixed $(y,z,\mu)\in\mathbb{R}^{n}\times\mathbb{R}^{n\times d}\times\mathcal{P}_{1}(\mathbb{R}^{n})$. The main objective of our paper consists in
 exploring the existence and the uniqueness of $L^{1}$ solutions to the mean-field BSDE \eqref{eq 1.2} under integrability conditions on the coefficients. We
establish four main results: Firstly, we prove the existence of an $L^{1}$ solution for a scalar-valued mean-field BSDE with an integrable terminal value, and
a generator satisfying a one-sided linear growth condition in $(y,\mu)$, as well as an iterated-logarithmically sub-linear growth condition in $z$. The proof relies on a test function method
and a localization technique. Secondly, using similar techniques, we establish the corresponding comparison theorem for $L^{1}$ solutions, which immediately yields
the uniqueness of the $Y$-component of the solution as a by-product. Thirdly, we investigate the existence and the uniqueness of $L^{1}$ solutions for this type of equations
in the multi-dimensional case, where the generator satisfies a one-sided Osgood condition and a general growth in $y$, a Lipschitz continuity as well as a sub-linear
growth in $z$, and a Lipschitz condition in $\mu$. Finally, we investigate the existence and the uniqueness of $L^{1}$ solutions for a new type of multi-dimensional mean-field BSDEs with integrable parameters, where the generator satisfies a monotonicity condition and a general growth in $y$, a Lipschitz continuity as well as a sub-linear
growth in $z$, and a Lipschitz condition in $\mu$. The equation is of the form:
\begin{equation*}\label{eq 1.3}\tag{1.3}
	Y_{t}=\xi+\int_{t}^{T}f(s,Y_{s},Z_{s},\mathbb{P}_{(Y_{s},Z_{s})})ds-\int_{t}^{T}Z_{s}dW_{s},\ t\in[0,T].
\end{equation*}

The main contributions of the present paper are twofold. Our first contribution consists in a non-trivial generalization of the $L^{1}$ solution results to the mean-field setting (see the Theorems \ref{th 3.6} and \ref{th 3.7}). In particular, we establish the existence of $L^{1}$ solutions and
give a new comparison theorem for such solutions in the one-dimensional mean-field framework via the localization method and new technical arguments. Our second
contribution consists of a non-trivial extension of the $L^{1}$ solvability results for multi-dimensional BSDEs to the mean-field setting
 (see the Theorems \ref{th 4.9}, \ref{th 4.10} and \ref{th 4.11}). With the help of BSDE estimates and Bihari's inequality, we prove the existence and the uniqueness of $L^{1}$ solutions for
 multi-dimensional mean-field BSDEs with integrable parameters under general conditions. This proof is very subtle; it involves namely a Picard iterative scheme and
 a technique for partitioning the time interval $[0,T]$.

The paper is organized as follows: In Section \ref{sec 2} we introduce the notations and provide preparations which are necessary for the subsequent discussions.
In Section \ref{sec 3} we prove the existence of $L^{1}$ solutions and establish a comparison theorem for one-dimensional mean-field BSDEs with integrable coefficients
 by using a test function method and a localization procedure. Finally, in Section \ref{sec 4} we discuss the existence and the uniqueness of $L^{1}$ solutions for
 multi-dimensional mean-field BSDEs with an integrable terminal condition whose generator satisfies a one-sided Osgood condition or a monotonicity condition.

\section{Preliminaries}\label{sec 2}
Throughout this paper, we let $T\in(0,+\infty)$ be a finite time horizon and $n,d\ge1$ be two positive integers. Let $(\Omega,\mathcal{F},\mathbb{P})$ be a complete probability space on which is defined a $d$-dimensional standard Brownian motion $(W_{t})_{t\in[0,T]}$. For $a,b\in\mathbb{R}$, we set $a\vee b:=\max\{a,b\}$, $a\wedge b:=\min\{a,b\}$, $a^{+}:=a\vee 0$ and $a^{-}:=-(a\wedge 0)$. We denote by ${\bm{1}}_{A}(\cdot)$ the indicator function of a set $A$, and sgn$(x):={\bm{1}}_{\{x>0\}}-{\bm{1}}_{\{x\le0\}},\ x\in \mathbb{R}$. For $1\le i\le n$, we denote by $z_{i},\ y_{i}$ and $f^{i}$ respectively the $i$-th row of the matrix $z\in \mathbb{R}^{n\times d}$, the $i$-th component of the vector $y\in \mathbb{R}^{n}$ and the generator $f$. The Euclidean norms of a vector $y\in\mathbb{R}^{n}$ and a matrix $z\in\mathbb{R}^{n\times d}$ are defined by $|y|$ and $|z|$, respectively. Let $\langle x,y\rangle$ represent the inner product between vectors $x,y\in\mathbb{R}^{n}$. The notation $\delta_{0}$ denotes the Dirac measure concentrated at the origin.

For every positive integer $m\ge1$, we recursively define the following functions:
\begin{equation*}
\begin{aligned}
&\ln^{(1)}(x):=\ln x,\ x\ge e^{(1)},\ \text{and}\ \ln^{(m)}(x):=\ln^{(m-1)}(\ln^{(1)}(x))=\ln^{(1)}(\ln^{(m-1)}(x)),\ x\ge e^{(m)},
\end{aligned}
\end{equation*}
where\ $e^{(1)}:=e\ \text{and}\ \ e^{(m)}:=e^{e^{(m-1)}}$.

Furthermore, for every positive integer $m\ge1$ and $\lambda\ge0$, we define the following function:
\begin{equation*}
\mathcal{IL}_{m}^{\lambda}(x):=\prod\limits_{i=1}^{m-1}\big(\ln^{(i)}(e^{(m)}+x)\big)^{\frac{1}{2}}\big(\ln^{(m)}(e^{(m)}+x)\big)^{\lambda},\ x\ge0.
\end{equation*}

\noindent Moreover, we suppose that there exists a sub-$\sigma$-field $\mathcal{F}^0\subset\mathcal{F}$, containing all the $\mathbb{P}$-null subsets of $\mathcal{F}$, such that the following statements hold true:\\
\indent\quad(i)\ \ the  Brownian motion $W$ is independent of $\mathcal{F}^0$;

\indent\quad(ii)\ $\mathcal{F}^{0}$ is ``rich enough'', i.e., $\mathcal{P}_{1}(\mathbb{R}^{k})=\{\mathbb{P}_{\xi},\ \xi\in L^{1}_{\mathcal{F}^{0}}(\Omega;\mathbb{R}^{k})\},\ k\geq1$.\\
\noindent Here $\mathbb{P}_\xi := \mathbb{P}\circ [\xi]^{-1}$ denotes the law of the random variable $\xi$ under the probability $\mathbb{P}$. We denote by $\mathbb{F}=(\mathcal{F}_{t})_{t\in[0,T]}$ the filtration generated by $W$ and augmented by $\mathcal{F}^{0}$, that is, $\mathcal{F}_t:=\sigma\{W_s,\ 0\leq s\leq t\}\vee\mathcal{F}^{0},\ t\in[0,T]$.

\indent The following spaces will be frequently used in what follows: For all $t\in[0,T]$ and $p>0$,\\
\noindent $\bullet\ L^{p}_{\mathcal{F}_{t}}(\Omega;\mathbb{R}^{n}):=\big\{\xi:\Omega\rightarrow \mathbb{R}^{n}|\ \xi\ \text{is an}\ \mathcal{F}_{t}\text{-measurable r.v. with}\  \Vert\xi\Vert_{L^{p}(\Omega)}:=\big(\mathbb{E}\big[|\xi|^{p}\big]\big)^{\frac{1}{p}\wedge 1}<+\infty\big\}$.

\noindent $\bullet\ L^{\infty}_{\mathcal{F}_{t}}(\Omega;\mathbb{R}^{n}):=\big\{\xi:\Omega\rightarrow \mathbb{R}^{n}|\ \xi\ \text{is an}\ \mathcal{F}_{t}\text{-measurable r.v. with}\  \Vert\xi\Vert_{\infty}:=\esssup\limits_{\omega\in \Omega}|\xi(\omega)|<+\infty\}$.

\noindent $\bullet\ \mathcal{S}^{p}_{\mathbb{F}}(t,T;\mathbb{R}^{n}):=\big\{\varphi:\Omega\times [t,T]\rightarrow \mathbb{R}^{n}|\ \varphi\ \text{is an}\ \mathbb{F}\text{-adapted}\ \text{continuous process}\ \text{with}\ \Vert\varphi\Vert_{\mathcal{S}^{p}_{\mathbb{F}}(t,T)}:=\big(\mathbb{E}\big[\sup_{s\in[t,T]}|\varphi_{s}|^{p}\big]\big)^{\frac{1}{p}\wedge1}<+\infty\big\}$.

\noindent $\bullet\ \mathcal{S}^{\infty}_{\mathbb{F}}(t,T;\mathbb{R}^{n}):=\big\{\varphi:\Omega\times [t,T]\rightarrow \mathbb{R}^{n}|\ \varphi\ \text{is an}\ \mathbb{F}\text{-adapted}\ \text{continuous process}\ \text{with}\ \Vert\varphi\Vert_{\mathcal{S}^{\infty}_{\mathbb{F}}(t,T)}:=\mathop{\rm esssup}\limits_{(s,\omega)\in [t,T]\times \Omega}|\varphi_{s}(\omega)|<+\infty\}$.

\noindent $\bullet\ \mathcal{L}^{p}_{\mathbb{F}}(t,T;\mathbb{R}^{n}):=\big\{\psi:\Omega\times [t,T]\rightarrow \mathbb{R}^{n}|\psi\ \text{is an}\ \mathbb{F}\text{-progressively measurable process with}\  \Vert\psi\Vert_{\mathcal{L}^{p}_{\mathbb{F}}(t,T)}:=\big(\mathbb{E}\big[\big(\int_{t}^{T}|\psi_{s}|ds\big)^{p}\big]\big)^{\frac{1}{p}\wedge1}<+\infty\big\}$.

\noindent $\bullet\ \mathcal{H}^{p}_{\mathbb{F}}(t,T;\mathbb{R}^{n}):=\big\{\psi:\Omega\times [t,T]\rightarrow \mathbb{R}^{n}|\psi\ \text{is an}\ \mathbb{F}\text{-progressively measurable process with}\  \Vert\psi\Vert_{\mathcal{H}^{p}_{\mathbb{F}}(t,T)}:=\big(\mathbb{E}\big[\big(\int_{t}^{T}|\psi_{s}|^{2}ds\big)^{\frac{p}{2}}\big]\big)^{\frac{1}{p}\wedge1}<+\infty\big\}$.

It is well-known that for all $p\ge1$, the $\mathcal{S}^{p}_{\mathbb{F}}(t,T;\mathbb{R}^{n})$ and $\mathcal{H}^{p}_{\mathbb{F}}(t,T;\mathbb{R}^{n})$ are Banach spaces, equipped with the norms $\Vert\cdot\Vert_{\mathcal{S}^{p}_{\mathbb{F}}(t,T)}$ and $\Vert\cdot\Vert_{\mathcal{H}^{p}_{\mathbb{F}}(t,T)}$, respectively. And, for all $p\in(0,1)$, the $\mathcal{S}^{p}_{\mathbb{F}}(t,T;\mathbb{R}^{n})$ and $\mathcal{H}^{p}_{\mathbb{F}}(t,T;\mathbb{R}^{n})$ are complete metric spaces with the resulting distances $(Y,\widetilde{Y})\rightarrow \Vert Y-\widetilde{Y}\Vert_{\mathcal{S}^{p}_{\mathbb{F}}(t,T)}$ and $(Z,\widetilde{Z})\rightarrow \Vert Z-\widetilde{Z}\Vert_{\mathcal{H}^{p}_{\mathbb{F}}(t,T)}$, respectively. Let us recall that a real-valued continuous process $Y=(Y_{t})_{t\in[0,T]}$ is said to belong to the class (D), if the set $\{Y_{\tau}:\tau\in\mathscr{T}[0,T]\}$ is a family of uniformly integrable random variables, where $\mathscr{T}[0,T]$ denotes the collection of all stopping times $\tau$ with values in $[0,T]$. Moreover, $\mathbb{E}_{t}[\cdot]$ denotes the conditional expectation knowing the $\sigma$-field $\mathcal{F}_{t}$, for $t\in[0,T]$.

For any $p\ge1$, we denote by $\mathcal{P}_{p}(\mathbb{R}^k)$ the totality of probability measures $\mu$ on $(\mathbb{R}^k,\mathcal{B}(\mathbb{R}^k))$ with finite $p$-th order moment, namely $\Vert\mu\Vert_{p}:=\big(\int_{\mathbb{R}^{k}}|x|^{p}\mu(dx)\big)^{\frac{1}{p}}<+\infty$, where $\mathcal{B}(\mathbb{R}^k)$ is the Borel $\sigma$-field over $\mathbb{R}^{k}$. We endow $\mathcal{P}_{p}(\mathbb{R}^k)$ with the topology induced by the $p$-Wasserstein metric: For $\mu,\nu\in\mathcal{P}_{p}(\mathbb{R}^k),$
\begin{equation*} W_{p}(\mu,\nu):=\inf\Big\{\Big(\int_{\mathbb{R}^k\times\mathbb{R}^k}|x-y|^{p}\pi(dxdy)\Big)^{\frac{1}{p}},\ \pi\in\mathcal{P}_{p}(\mathbb{R}^{2k})\ \text{with marginals}\ \mu\ \text{and}\ \nu\Big\}.
\end{equation*}

For our discussion which follows, we introduce the following definitions concerning the solutions of mean-field BSDE \eqref{eq 1.2}.

\begin{definition}\label{de 2.1} A solution of mean-field BSDE \eqref{eq 1.2} is a pair of $\mathbb{F}$-progressively measurable processes $(Y,Z)$ with values in $\mathbb{R}^{n}\times\mathbb{R}^{n\times d}$ such that $\int_{0}^{T}|Z_{s}|^{2}ds<+\infty$, $\int_{0}^{T}|f(s,Y_{s},Z_{s},\mathbb{P}_{Y_{s}})|ds<+\infty$, $\mathbb{P}$-a.s., and equation \eqref{eq 1.2}\ holds true for all $t\in[0,T]$, $\mathbb{P}$-a.s.
\end{definition}


\begin{definition}\label{de 2.2} A pair of processes $(Y,Z)$ is said to be an $L^{1}$ solution of mean-field BSDE \eqref{eq 1.2}, if $Y$ belongs to the class (D), $(Y,Z)\in \mathcal{S}_{\mathbb{F}}^{p}(0,T;\mathbb{R}^{n})\times \mathcal{H}_{\mathbb{F}}^{p}(0,T;\mathbb{R}^{n\times d})$, for all $p\in(0,1)$, and equation \eqref{eq 1.2}\ holds true for all $t\in[0,T]$, $\mathbb{P}$-a.s.
\end{definition}

Hereinafter, we present two technical lemmas which will be frequently used in our subsequent analysis and discussions. Lemma \ref{le 2.3} provides an upper bound sequence for functions with linear growth; the reader can refer to \cite[Lemma 1]{F18} for more details.

\begin{lemma}\label{le 2.3}
Suppose that $u(\cdot):[0,+\infty)\rightarrow[0,+\infty)$ is a function which grows at most linearly, i.e., there is a constant $A>0$ such that\ $u(x)\le A(x+1),\ x\ge0.$ Then for all $m\ge1$, we have $u(x)\le (m+2A)x+u\Big(\frac{2A}{m+2A}\Big),\ x\ge0$.
\end{lemma}

The following lemma can be viewed as a backward version of Bihari's inequality; the reader can refer to \cite[Lemma 2]{F18} for more details.

\begin{lemma}\label{le 2.4} (Bihari's inequality) Suppose that $v(\cdot):[0,T]\rightarrow[0,+\infty)$ is a nonnegative function which satisfies that there exists a constant $v_{0}\ge0$ and a continuous nondecreasing function $\kappa:[0,+\infty)\rightarrow[0,+\infty)$ with $\kappa(0)=0,\ \kappa(u)>0,\ u>0$ and $\int_{0^{+}}\frac{du}{\kappa(u)}=+\infty$, such that
\begin{equation*}
v(t)\le v_{0}+\int_{t}^{T}\kappa(v(s))ds,\ t\in[0,T].
\end{equation*}

\noindent Then we have $v(t)\le \Pi^{-1}\big(\Pi(v_{0})+T-t\big),\ t\in[0,T],$ where $\Pi(x):=\int_{1}^{x}\frac{1}{\kappa(u)}du,\ x>0$, is a continuous strictly increasing function with values in $\mathbb{R}$, and $\Pi^{-1}$ represents the inverse function of $\Pi$. Furthermore, provided that $v_{0}=0$, then $v(t)=0$, for all $t\in[0,T]$.
\end{lemma}

\section{One-dimensional case}\label{sec 3}
\hspace{2em}In this section, we investigate the existence and the uniqueness of the $L^{1}$ solution for one-dimensional mean-field BSDE \eqref{eq 1.2} with only integrable parameters ($n=1$). For this end, let the terminal value $\xi:\Omega\rightarrow \mathbb{R}$ be $\mathcal{F}_{T}$-measurable and the mapping $f:=f(\omega,t,y,z,\mu):\Omega\times[0,T]\times\mathbb{R}\times\mathbb{R}^{d}\times\mathcal{P}_{1}(\mathbb{R})\rightarrow \mathbb{R}$ be $\mathbb{F}$-progressively measurable, for all fixed $(y,z,\mu)\in \mathbb{R}\times\mathbb{R}^{d}\times\mathcal{P}_{1}(\mathbb{R})$. Moreover, we make the following assumptions on the generator:

\noindent\textbf{(A1)}\ For $d\mathbb{P}\times dt$-almost all $(\omega,t)\in \Omega\times [0,T]$, $f(\omega,t,\cdot,\cdot,\cdot)$ is continuous, and has a modulus of continuity w.r.t. the law: For all $t\in[0,T],\ y\in\mathbb{R},\ z\in\mathbb{R}^{d},\ \text{and}\ \mu,\ \bar{\mu}\in\mathcal{P}_{1}(\mathbb{R})$,
\begin{equation*}
|f(t,y,z,\mu)-f(t,y,z,\bar{\mu})|\le\rho\big(W_{1}(\mu,\bar{\mu})\big),\ d\mathbb{P}\times dt\text{-}a.e.
\end{equation*}
Here $\rho:[0,+\infty)\rightarrow[0,+\infty)$ is supposed to be nondecreasing, with the property that $\rho(0+)=0$.

\noindent\textbf{(A2)$_{m}^{\lambda}$}\ There exist two constants $K\ge0,\ \gamma>0$ and a nonnegative process $\theta=(\theta_{t})_{t\in[0,T]}\in \mathcal{L}_{\mathbb{F}}^{1}(0,T;\mathbb{R})$, such that, for all $t\in[0,T],\ y\in\mathbb{R}$, $z\in\mathbb{R}^{d}$, and $\mu\in \mathcal{P}_{1}(\mathbb{R}),$
\begin{equation*}
\text{sgn}(y)f(t,y,z,\mu)\le \theta_{t}+K|y|+KW_{1}(\mu,\delta_{0})+\frac{\gamma|z|}{\mathcal{IL}_{m}^{\lambda}(|z|)},\ d\mathbb{P}\times dt\text{-}a.e.
\end{equation*}

\noindent\textbf{(A3)}\ There exist a constant $\gamma_{0}>0$ and a continuous nondecreasing function $\psi:[0,+\infty)\rightarrow [0,+\infty)$, together with a nonnegative process $\theta=(\theta_{t})_{t\in[0,T]}\in \mathcal{L}_{\mathbb{F}}^{1}(0,T;\mathbb{R})$, such that, for all $t\in[0,T],\ y\in\mathbb{R}$, $z\in\mathbb{R}^{d}$, and $\mu\in \mathcal{P}_{1}(\mathbb{R}),$
\begin{equation*}
|f(t,y,z,\mu)|\le \theta_{t}+\psi\big(|y|\vee W_{1}(\mu,\delta_{0})\big)+\gamma_{0}|z|^{2},\ d\mathbb{P}\times dt\text{-}a.e.
\end{equation*}

\noindent\textbf{(A4)}\ Monotonicity in $\mu$: For all $t\in[0,T],\ y\in\mathbb{R},\ z\in\mathbb{R}^{d}$, and $\eta,\ \bar{\eta}\in L^{1}_{\mathcal{F}}(\Omega;\mathbb{R})$,
\begin{equation*}
f(t,y,z,\mathbb{P}_{\eta})\le f(t,y,z,\mathbb{P}_{\bar{\eta}}),\ d\mathbb{P}\times dt\text{-}a.e.,\ \ \text{whenever}\ \ \eta\le\bar{\eta},\ \mathbb{P}\text{-}a.s.
\end{equation*}

\noindent\textbf{(A5)}\ There exists a continuous, nondecreasing and concave function $\kappa:[0,+\infty)\rightarrow[0,+\infty)$ with $\kappa(0)=0,\ \kappa(u)>0,\ u>0$ and $\int_{0^{+}}\frac{du}{\kappa(u)}=+\infty$, such that, for all $t\in[0,T],\ y,\ \bar{y}\in\mathbb{R}$, $z\in\mathbb{R}^{d}$, and $\mu\in \mathcal{P}_{1}(\mathbb{R}),$
\begin{equation*}
f(t,y,z,\mu)-f(t,\bar{y},z,\mu)\le \kappa(y-\bar{y}),\ \ \text{with}\ \ y>\bar{y},\ d\mathbb{P}\times dt\text{-}a.e.
\end{equation*}

\noindent\textbf{(A6)$_{m}^{\lambda}$}\ There exists a continuous nondecreasing function $\zeta:[0,+\infty)\rightarrow[0,+\infty)$ with linear growth and $\zeta(0)=0$, such that, for all $t\in[0,T],\ y\in\mathbb{R}$, $z,\ \bar{z}\in\mathbb{R}^{d}$, and $\mu\in \mathcal{P}_{1}(\mathbb{R}),$
\begin{equation*}
|f(t,y,z,\mu)-f(t,y,\bar{z},\mu)|\le \zeta\big(\frac{|z-\bar{z}|}{\mathcal{IL}_{m}^{\lambda}(|z-\bar{z}|)}\big),\ d\mathbb{P}\times dt\text{-}a.e.
\end{equation*}

\noindent\textbf{(A7)}\ There exists a constant $L>0$, such that, for all $t\in[0,T],\ y\in\mathbb{R}$, $z\in\mathbb{R}^{d}$, and $\eta,\ \bar{\eta}\in L^{1}_{\mathcal{F}}(\Omega;\mathbb{R}),$
\begin{equation*}
f(t,y,z,\mathbb{P}_{\eta})-f(t,y,z,\mathbb{P}_{\bar{\eta}})\le L\mathbb{E}[(\eta-\bar{\eta})^{+}],\ d\mathbb{P}\times dt\text{-}a.e.
\end{equation*}

\begin{remark}\label{re 3.1}
\end{remark}
\begin{enumerate}
\item[(i)] For every $m\ge1$ and $\lambda>\frac{1}{2}$, there is a constant $M>0$ depending only on $m,\lambda$ satisfying $\mathcal{IL}_{m+1}^{\lambda}(x)\le M\cdot\mathcal{IL}_{m}^{\lambda}(x),\ x\ge0.$ This implies that when $m$ increases, assumptions \textbf{(A2)$_{m}^{\lambda}$} and \textbf{(A6)$_{m}^{\lambda}$} become weaker.

\item[(ii)] Since for every $m\ge1$, $\ln^{(m)}(e^{(m)}+x)\ge \ln e=1,\ x\ge0$, assumptions \textbf{(A2)$_{m}^{\lambda}$} and \textbf{(A6)$_{m}^{\lambda}$} become weaker as $\lambda\in(\frac{1}{2},+\infty)$ decreases.

\item[(iii)]  Observe that for every $m\ge1$ and $h>e^{(m)}$, there exists a constant $M>0$ depending only on $m,h$ such that $1\le \frac{\ln^{(m)}(h+x)}{\ln^{(m)}(e^{(m)}+x)}\le M,\ x\ge0.$ Therefore, the function $\mathcal{IL}_{m}^{\lambda}(x)$ appearing in assumptions \textbf{(A2)$_{m}^{\lambda}$} and \textbf{(A6)$_{m}^{\lambda}$} can be equivalently replaced by the following function: For every $m\ge1$ and $h>e^{(m)}$,
\begin{equation*}
\mathcal{IL}_{m,h}^{\lambda}(x):=\prod\limits_{i=1}^{m-1}\big(\ln^{(i)}(h+x)\big)^{\frac{1}{2}}\big(\ln^{(m)}(h+x)\big)^{\lambda},\ x\ge0.
\end{equation*}

\item[(iv)] It is not difficult to check that for every $m\ge1,\ \lambda>\frac{1}{2}$ and $\alpha\in(0,1)$, there exists a constant $M>0$ depending only on $m,\lambda,\alpha$ such that $|x|^{\alpha}\le M+\frac{|x|}{\mathcal{IL}_{m}^{\lambda}(|x|)},\ x\in\mathbb{R}.$ Hence, the iterated-logarithmically sub-linear growth with respect to $z$ in \textbf{(A2)$_{m}^{\lambda}$} is weaker than the sub-linear growth in $z$, which reduces to the logarithmic sub-linear growth of assumption \textbf{(H2)} in \cite{FH23} when $m=1\ \text{and}\ \lambda>\frac{1}{2}$.

\item[(v)] Obviously, assumption \textbf{(A7)} is stronger than assumption \textbf{(A4)}. Conversely, if the generator $f$ fulfills both the Lipschitz condition and assumption \textbf{(A4)} with respect to $\mu$, then it automatically satisfies assumption \textbf{(A7)} (cf. \cite[Remark 2.1]{LL18}).

\item[(vi)] Analogous to \cite[Remark 4.1]{LX22}, both assumptions \textbf{(A1)} and \textbf{(A4)} hold if and only if there exists a modulus of continuity $\rho:[0,+\infty)\rightarrow[0,+\infty)$, such that, for all $t\in[0,T],\ y\in\mathbb{R},\ z\in\mathbb{R}^{d},\ \text{and}\ \mu,\ \bar{\mu}\in\mathcal{P}_{1}(\mathbb{R})$, $f(t,y,z,\mu)-f(t,y,z,\bar{\mu})\le\rho\big(W_{1,+}(\mu,\bar{\mu})\big),\ d\mathbb{P}\times dt\text{-}a.e.,$ where $W_{1,+}(\mu,\bar{\mu}):=\inf\big\{\int_{\mathbb{R}\times\mathbb{R}}(x-y)^{+}\pi(dxdy),\ \pi\in\mathcal{P}_{1}(\mathbb{R}^{2})\ \text{with marginals}\ \mu\ \text{and}\ \bar{\mu}\big\}$.
\end{enumerate}

\begin{example}\label{ex 3.2}
For all $(\omega,t,y,z,\mu)\in \Omega\times[0,T]\times\mathbb{R}\times \mathbb{R}^{d}\times\mathcal{P}_{1}(\mathbb{R})$, we can consider the generator
\begin{equation*}
\begin{aligned}
f(\omega,t,y,z,\mu)&:=|W_{t}(\omega)|-e^{y}\cos^{2}|z|+|y|+\frac{2|z|\sin|z|}{\big(\ln(e^{(2)}+|z|)\big)^{\frac{1}{2}}\big(\ln\ln(e^{(2)}+|z|)\big)^{\frac{5}{6}}}\\
&\hspace{2em}-\text{sgn}(y)\sin^{2}(y)|z|^{2}+\Big(\int_{\mathbb{R}}x^{+}\mu(dx)\Big)^{\frac{1}{3}}+\int_{\mathbb{R}}x\mu(dx).
\end{aligned}
\end{equation*}
It is straightforward to check that $f$ satisfies assumptions \textbf{(A1)}, \textbf{(A2)$_{m}^{\lambda}$}, \textbf{(A3)} and \textbf{(A4)} with $\rho(x)=x+x^{\frac{1}{3}},\ x\ge0,\ \theta_{t}(\omega)=|W_{t}(\omega)|+2,\ K=2,\ \gamma=2,\ m=2,\ \lambda=\frac{5}{6},\ \gamma_{0}=2\ \ \text{and}\ \ \psi(u)=e^{u}+3u,\ u\ge0.$
\end{example}

\begin{example}\label{ex 3.3}
For all $(\omega,t,y,z,\mu)\in \Omega\times[0,T]\times\mathbb{R}\times \mathbb{R}^{d}\times\mathcal{P}_{1}(\mathbb{R})$, we can consider the generator
\begin{equation*}
\begin{aligned}
\bar{f}(\omega,t,y,z,\mu)&:=e^{|W_{t}(\omega)|}+y^{6}\textbf{1}_{\{y\le0\}}+\phi(|y|)+\cos|z|+|z|^{\frac{5}{6}}+\frac{|z|}{\ln(e+|z|)}\\
&\hspace{2em}+\frac{|z|}{\big(\ln(e^{(2)}+|z|)\big)^{\frac{1}{2}}\big(\ln\ln(e^{(2)}+|z|)\big)^{\frac{3}{4}}}+\arctan\Big(\int_{\mathbb{R}}x\mu(dx)\Big)+2\int_{\mathbb{R}}x\mu(dx),
\end{aligned}
\end{equation*}
where $\phi(u):=(u|\ln u|\ln|\ln u|)\textbf{1}_{\{0< u\le\varepsilon\}}+[\phi^{\prime}(\varepsilon-)(u-\varepsilon)+\phi(\varepsilon)]\textbf{1}_{\{u>\varepsilon\}},\ u>0,$ with $\phi(0)=0$ and $\varepsilon>0$ small enough. It can be verified that $\bar{f}$ satisfies assumptions \textbf{(A1)}, \textbf{(A2)$_{m}^{\lambda}$}, \textbf{(A3)}, \textbf{(A4)}, \textbf{(A5)}, \textbf{(A6)$_{m}^{\lambda}$} and \textbf{(A7)} with $\rho(x)=3x,\ x\ge0,\ \kappa(u)=\phi(u),\ u\ge0,\ m=2,\ \lambda=\frac{3}{4}\ \ \text{and}\ \ L=3.$
\end{example}

Firstly, the following proposition offers a crucial a priori estimate for the $L^{1}$ solution to the one-dimensional mean-field BSDE \eqref{eq 1.2} involving integrable parameters, which will be used later and extends the non mean-field results of \cite[Proposition 3.5]{FH24}.

\begin{proposition}\label{prop 3.4}
Let be given $\xi \in L_{\mathcal{F}_{T}}^{1}(\Omega;\mathbb{R})$ and a generator $f$ which fulfills assumption \textbf{(A2)$_{m}^{\lambda}$} with some $m\ge2\ \text{and}\ \lambda>\frac{1}{2}$. Suppose that $(Y,Z)$ is a solution of mean-field BSDE \eqref{eq 1.2} such that $(|Y_{t}|+\int_{0}^{t}\theta_{s}ds)_{t\in[0,T]}$ belongs to the class (D). Then there exists a constant $\widetilde{C}>0$ depending only on $m,\lambda,\gamma,K,T,\Vert\xi\Vert_{L^{1}(\Omega)}$ and $\Vert\theta\Vert_{\mathcal{L}_{\mathbb{F}}^{1}(0,T)}$ such that
\begin{equation*}
|Y_{t}|\le |Y_{t}|+\int_{0}^{t}\theta_{s}ds\le \widetilde{C}\mathbb{E}_{t}\Big[|\xi|+\int_{0}^{T}\theta_{s}ds\Big]+\widetilde{C},\ t\in[0,T],\ \mathbb{P}\text{-}a.s.
\end{equation*}
\end{proposition}

\begin{proof}
Based on (iii) in Remark \ref{re 3.1}, without loss of generality we can make the assumption that the generator $f$ satisfies assumption \textbf{(A2)$_{m}^{\lambda}$}, wherein $\mathcal{IL}_{m}^{\lambda}$ is substituted by $\mathcal{IL}_{m,h}^{\lambda}$, for a sufficiently large constant $h>e^{(m)}$ which depends only on $m,\lambda,\gamma$.

Let us first show that there is a uniform upper bound estimate on the expected value of $Y$, which involves proving the existence of a constant $C>0$ depending only on $m,\lambda,\gamma,K,T$ such that
\begin{equation*}
\mathbb{E}[|Y_{t}|]\le C\mathbb{E}\Big[|\xi|+\int_{0}^{T}\theta_{s}ds\Big]+C,\ t\in[0,T].
\end{equation*}

\noindent For the sake of clarity, we put $\widetilde{Y}_{t}:=|Y_{t}|+\int_{0}^{t}\theta_{s}ds,\ \widetilde{Z}_{t}:=\text{sgn}(Y_{t})Z_{t},\ t\in[0,T].$ Then applying It\^{o}-Tanaka's formula to $\widetilde{Y}_{t}$ on $[t,T]$ gives that
\begin{equation*}
\begin{aligned}
\widetilde{Y}_{t}=\widetilde{Y}_{T}+\int_{t}^{T}\big(\text{sgn}(Y_{s})f(s,Y_{s},Z_{s},\mathbb{P}_{Y_{s}})-\theta_{s}\big)ds-\int_{t}^{T}\widetilde{Z}_{s}dW_{s}-\int_{t}^{T}dL_{s},\ t\in[0,T],
\end{aligned}
\end{equation*}
where the process $L=(L_{t})_{t\in[0,T]}$ denotes the local time of $Y$ at $0$.

Now, we introduce a test function $\Phi(\cdot,\cdot)$, similar to (3.23) in \cite{FH24} and defined as follows: For all $h\ge e^{(m)}\ \text{and}\ (t,x)\in[0,T]\times[0,+\infty)$,
\begin{equation*}\label{eq 3.1}\tag{3.1}
\begin{aligned}
\Phi(t,x):=(h+x)\Big[1-\big(\ln^{(m)}(h+x)\big)^{1-2\lambda}\Big]\cdot\exp\Big\{2\Big(K+\frac{2\gamma^{2}}{2\lambda-1}\Big)t\Big\}.
\end{aligned}
\end{equation*}

\noindent A direct calculation shows that for all $(t,x)\in[0,T]\times[0,+\infty)$,
\begin{equation*}\label{eq 3.2}\tag{3.2}
\begin{aligned}
0<\partial_{x}\Phi(t,x)&=\bigg[1-\frac{1}{\big(\ln^{(m)}(h+x)\big)^{2\lambda-1}}\Big(1-\frac{2\lambda-1}{\prod_{j=1}^{m}\ln^{(j)}(h+x)}\Big)\bigg]\\
&\hspace{2.6em}\times\exp\Big\{2\Big(K+\frac{2\gamma^{2}}{2\lambda-1}\Big)t\Big\}\le Q_{0},
\end{aligned}
\end{equation*}
where $Q_{0}:=\exp\big\{2\big(K+\frac{2\gamma^{2}}{2\lambda-1}\big)T\big\}.$ It then follows from \cite[Proposition 3.4]{FH24} that for a sufficiently large constant $h>e^{(m)}$ which depends only on $m,\lambda,\gamma$, it holds that for all $(t,x,z)\in[0,T]\times[0,+\infty)\times \mathbb{R}^{d}$,
\begin{equation}\label{eq 3.3}\tag{3.3}
\begin{aligned}
-K\partial_{x}\Phi(t,x)x-\partial_{x}\Phi(t,x)\frac{\gamma|z|}{\mathcal{IL}_{m,h}^{\lambda}(|z|)}+\frac{1}{2}\partial_{xx}\Phi(t,x)|z|^{2}+\partial_{t}\Phi(t,x)\ge0.
\end{aligned}
\end{equation}

\noindent By reapplying It\^{o}-Tanaka's formula to $\Phi(t,\widetilde{Y}_{t})$ and combining it with \textbf{(A2)$_{m}^{\lambda}$}, we get
\begin{equation*}\label{eq 3.4}\tag{3.4}
\begin{aligned}
d\Phi(t,\widetilde{Y}_{t})&=\partial_{x}\Phi(t,\widetilde{Y}_{t})\big(-\text{sgn}(Y_{t})f(t,Y_{t},Z_{t},\mathbb{P}_{Y_{t}})+\theta_{t}\big)dt+\partial_{x}\Phi(t,\widetilde{Y}_{t})\widetilde{Z}_{t}dW_{t}\\
&\hspace{2em}+\partial_{x}\Phi(t,\widetilde{Y}_{t})dL_{t}+\frac{1}{2}\partial_{xx}\Phi(t,\widetilde{Y}_{t})|\widetilde{Z}_{t}|^{2}dt+\partial_{t}\Phi(t,\widetilde{Y}_{t})dt\\
&\ge \bigg[-K\partial_{x}\Phi(t,\widetilde{Y}_{t})|Y_{t}|-\partial_{x}\Phi(t,\widetilde{Y}_{t})\frac{\gamma|Z_{t}|}{\mathcal{IL}_{m,h}^{\lambda}(|Z_{t}|)}+\frac{1}{2}\partial_{xx}\Phi(t,\widetilde{Y}_{t})|Z_{t}|^{2}+\partial_{t}\Phi(t,\widetilde{Y}_{t})\bigg]dt\\
&\hspace{2em}-K\partial_{x}\Phi(t,\widetilde{Y}_{t})\cdot\mathbb{E}[|Y_{t}|]dt+\partial_{x}\Phi(t,\widetilde{Y}_{t})\widetilde{Z}_{t}dW_{t},\ t\in[0,T].
\end{aligned}
\end{equation*}

\noindent Thus, noting that $|Y_{t}|\le \widetilde{Y}_{t},\ t\in[0,T]$, we can deduce from \eqref{eq 3.3} that
\begin{equation}\label{eq 3.5}\tag{3.5}
\begin{aligned}
d\Phi(t,\widetilde{Y}_{t})\ge -K\partial_{x}\Phi(t,\widetilde{Y}_{t})\cdot\mathbb{E}[|Y_{t}|]dt+\partial_{x}\Phi(t,\widetilde{Y}_{t})\widetilde{Z}_{t}dW_{t},\ t\in[0,T].
\end{aligned}
\end{equation}

\noindent For every $n\ge1$ and $t\in[0,T]$, we define the following stopping time:
\begin{equation*}
\begin{aligned}
\tau_{n}^{t}:=\inf\Big\{s\in[t,T]:\int_{t}^{s}|\partial_{x}\Phi(r,\widetilde{Y}_{r})|^{2}|\widetilde{Z}_{r}|^{2}dr\ge n\Big\}\wedge T
\end{aligned}
\end{equation*}
with the convention that $\inf \emptyset=+\infty$. Then, using the definition of $\tau_{n}^{t}$ and combining this with \eqref{eq 3.5}, we derive that, for every $n\ge1$ and $t\in[0,T]$,
\begin{equation}\label{eq 3.6}\tag{3.6}
\begin{aligned}
\Phi(t,\widetilde{Y}_{t})\le \mathbb{E}_{t}\Big[\Phi(\tau_{n}^{t},\widetilde{Y}_{\tau_{n}^{t}})+K\int_{t}^{\tau_{n}^{t}}\partial_{x}\Phi(s,\widetilde{Y}_{s})\cdot\mathbb{E}[|Y_{s}|]ds\Big].
\end{aligned}
\end{equation}

\noindent Moreover, in view of the definition of $\Phi(\cdot,\cdot)$, one has that, for a sufficiently large constant $h>e^{(m)}$,
\begin{equation}\label{eq 3.7}\tag{3.7}
\begin{aligned}
\frac{1}{2}(h+x)\le \Phi(t,x)\le Q_{0}(h+x),\ (t,x)\in [0,T]\times [0,+\infty).
\end{aligned}
\end{equation}

\noindent By plugging \eqref{eq 3.7} into \eqref{eq 3.6} and taking into account that $0<\partial_{x}\Phi(t,x)\le Q_{0},\ (t,x)\in[0,T]\times[0,+\infty)$, we obtain that, for every $n\ge1$ and $t\in[0,T]$,
\begin{equation}\label{eq 3.8}\tag{3.8}
\begin{aligned}
\frac{1}{2}(h+\widetilde{Y}_{t})\le \Phi(t,\widetilde{Y}_{t})&\le \mathbb{E}_{t}\Big[\Phi(\tau_{n}^{t},\widetilde{Y}_{\tau_{n}^{t}})+K\int_{t}^{\tau_{n}^{t}}\partial_{x}\Phi(s,\widetilde{Y}_{s})\cdot\mathbb{E}[|Y_{s}|]ds\Big]\\
&\le Q_{0}h+Q_{0}\mathbb{E}_{t}[\widetilde{Y}_{\tau_{n}^{t}}]+KQ_{0}\int_{t}^{T}\mathbb{E}[|Y_{s}|]ds.
\end{aligned}
\end{equation}

\noindent As $\widetilde{Y}$ is of the class (D), and $\tau_{n}^{t}\uparrow T$, as $n\rightarrow +\infty$, for all $t\in[0,T]$, by taking the limit as $n$ tends to infinity on both sides of the latter inequality, we get that
\begin{equation*}\label{eq 3.9}\tag{3.9}
\begin{aligned}
|Y_{t}|\le\widetilde{Y}_{t}\le  2Q_{0}h+2Q_{0}\mathbb{E}_{t}[\widetilde{Y}_{T}]+2KQ_{0}\int_{t}^{T}\mathbb{E}[|Y_{s}|]ds,\ t\in[0,T].
\end{aligned}
\end{equation*}

\noindent Furthermore, according to the definition of $\widetilde{Y}$, taking the expectation on both sides yields
\begin{equation*}
\begin{aligned}
\mathbb{E}[|Y_{t}|]\le\mathbb{E}[\widetilde{Y}_{t}]\le 2Q_{0}h+2Q_{0}\mathbb{E}\Big[|\xi|+\int_{0}^{T}\theta_{s}ds\Big]+2KQ_{0}\int_{t}^{T}\mathbb{E}[|Y_{s}|]ds,\ t\in[0,T].
\end{aligned}
\end{equation*}
The (backward) Gronwall's inequality allows to conclude that for all $t\in[0,T]$,
\begin{equation}\label{eq 3.10}\tag{3.10}
\begin{aligned}
\mathbb{E}[|Y_{t}|]\le \Big(2Q_{0}h+2Q_{0}\mathbb{E}\Big[|\xi|+\int_{0}^{T}\theta_{s}ds\Big]\Big)e^{2KQ_{0}T}\le C\mathbb{E}\Big[|\xi|+\int_{0}^{T}\theta_{s}ds\Big]+C,
\end{aligned}
\end{equation}
where $C:=2Q_{0}(h\vee 1)e^{2KQ_{0}T}$, which completes the upper bound estimate on $\mathbb{E}[|Y_{t}|],\ t\in[0,T]$.

Finally, substituting \eqref{eq 3.10} in \eqref{eq 3.9} and recalling the definition of $\widetilde{Y}_{t},\ t\in[0,T]$, we conclude that there is a constant $\widetilde{C}>0$ depending only on $m,\lambda,\gamma,K,T,\Vert\xi\Vert_{L^{1}(\Omega)}$ and $\Vert\theta\Vert_{\mathcal{L}_{\mathbb{F}}^{1}(0,T)}$ such that
\begin{equation*}\label{eq 3.11}\tag{3.11}
|Y_{t}|\le |Y_{t}|+\int_{0}^{t}\theta_{s}ds\le \widetilde{C}\mathbb{E}_{t}\Big[|\xi|+\int_{0}^{T}\theta_{s}ds\Big]+\widetilde{C},\ t\in[0,T],\ \mathbb{P}\text{-}a.s.
\end{equation*}

\end{proof}

Next, before presenting our main results, we give the stability result of bounded solutions for mean-field BSDEs with quadratic growth, which plays an important role in our subsequent proof of the existence of $L^{1}$ solutions for mean-field BSDEs \eqref{eq 1.2} with integrable parameters.

\begin{lemma}\label{le 3.5} Suppose that for all $n\ge1$, the terminal value $\xi_{n}\in L^{\infty}_{\mathcal{F}_{T}}(\Omega;\mathbb{R})$ and the generator $f_{n}:\Omega\times[0,T]\times\mathbb{R}\times\mathbb{R}^{d}\times \mathcal{P}_{1}(\mathbb{R})\rightarrow\mathbb{R}$ is $\mathbb{F}$-progressively measurable, for all $(y,z,\mu)\in \mathbb{R}\times\mathbb{R}^{d}\times \mathcal{P}_{1}(\mathbb{R})$. Furthermore, suppose that $(Y^{n},Z^{n})\in \mathcal{S}_{\mathbb{F}}^{\infty}(0,T;\mathbb{R})\times \mathcal{H}_{\mathbb{F}}^{2}(0,T;\mathbb{R}^{d})$ is a solution of the following BSDE:
\begin{equation*}
\begin{aligned}
Y_{t}^{n}=\xi_{n}+\int_{t}^{T}f_{n}(s,Y_{s}^{n},Z_{s}^{n},\mathbb{P}_{Y_{s}^{n}})ds-\int_{t}^{T}Z_{s}^{n}dW_{s},\ t\in[0,T],
\end{aligned}
\end{equation*}

\noindent $Y_{t}^{n}(\omega)\uparrow Y_{t}(\omega)$ (resp., $Y_{t}^{n}(\omega)\downarrow Y_{t}(\omega)$), as $n\rightarrow+\infty$, $d\mathbb{P}\times dt\text{-}a.e.$, and that there exist two constants $M>0,\ \gamma_{0}>0$ and a nonnegative process $\theta=(\theta_{t})_{t\in[0,T]}\in\mathcal{L}^{1}_{\mathbb{F}}(0,T;\mathbb{R})$, together with a continuous nondecreasing function $\psi:[0,+\infty)\rightarrow [0,+\infty)$, such that $\sup_{n\ge1}|Y_{t}^{n}(\omega)|\le M,\ d\mathbb{P}\times dt\text{-}a.e.$, and for all $(y,z,\mu)\in[-M,M]\times\mathbb{R}^{d}\times\mathcal{P}_{1}(\mathbb{R})$,
\begin{equation*}
\begin{aligned}
\sup\limits_{n\ge1}|f_{n}(\omega,t,y,z,\mu)|\le \theta_{t}(\omega)+\psi\big(W_{1}(\mu,\delta_{0})\big)+\gamma_{0} |z|^{2},\ d\mathbb{P}\times dt\text{-}a.e.
\end{aligned}
\end{equation*}
Then there exists a process $Z\in\mathcal{H}_{\mathbb{F}}^{2}(0,T;\mathbb{R}^{d})$ such that the sequence $(Z^{n})_{n\ge1}$ strongly converges to $Z$ in $\mathcal{H}_{\mathbb{F}}^{2}(0,T;\mathbb{R}^{d})$. Moreover, if the sequence of generators $\{f_{n}\}_{n\ge1}$ converges locally uniformly to a generator $f$, that is, $d\mathbb{P}\times dt\text{-}a.e.$, for every sequence $(y^{n},z^{n},\mu^{n})_{n\ge1}$ in $\mathbb{R}\times\mathbb{R}^{d}\times \mathcal{P}_{1}(\mathbb{R})$,
\begin{equation*}
\begin{aligned}
f_{n}(\omega,t,y^{n},z^{n},\mu^{n})\rightarrow f(\omega,t,y,z,\mu),\ \ \ \text{if}\ \ (y^{n},z^{n},\mu^{n})\rightarrow(y,z,\mu),\ \text{as}\ n\rightarrow+\infty,
\end{aligned}
\end{equation*}
then $(Y,Z)\in\mathcal{S}_{\mathbb{F}}^{\infty}(0,T;\mathbb{R})\times \mathcal{H}_{\mathbb{F}}^{2}(0,T;\mathbb{R}^{d})$ is a solution of mean-field BSDE $(\xi,f)$ with $\xi=Y_{T}$.
\end{lemma}

\begin{proof}
Lemma \ref{le 3.5} is a direct consequence of \cite[Proposition 3.1]{LF18}. Indeed, it suffices to apply Proposition 3.1 to $g_{n}(\omega,t,y,z):=f_{n}(\omega,t,y,z,\mathbb{P}_{Y_{t}^{n}}),\ n\ge1$. (see \cite[Proposition 3.1]{LF18}).
\end{proof}

With the help of Proposition \ref{prop 3.4} and Lemma \ref{le 3.5}, we will formulate the existence theorem of an $L^{1}$ solution for one-dimensional mean-field BSDE \eqref{eq 1.2}. For this we use the localization approach employed in \cite{BH06} and proceed in analogy to the proof of \cite[Theorem 2.1]{FH23}.

\begin{theorem}\label{th 3.6}
Let assumptions \textbf{(A1)}, \textbf{(A2)$_{m}^{\lambda}$}, \textbf{(A3)} and \textbf{(A4)} be satisfied with some $m\ge2\ \text{and}\ \lambda>\frac{1}{2}$. Then, for every given $\xi\in L_{\mathcal{F}_{T}}^{1}(\Omega;\mathbb{R})$, the mean-field BSDE \eqref{eq 1.2} admits a solution $(Y,Z)$ such that $Y$ is of the class (D) and $(Y,Z)\in\mathcal{S}_{\mathbb{F}}^{p}(0,T;\mathbb{R})\times \mathcal{H}_{\mathbb{F}}^{p}(0,T;\mathbb{R}^{d})$, for all $p\in(0,1)$. Moreover, there exists a constant $\widetilde{C}>0$ depending only on $m,\lambda,\gamma,K,T,\Vert\xi\Vert_{L^{1}(\Omega)}$ and $\Vert\theta\Vert_{\mathcal{L}_{\mathbb{F}}^{1}(0,T)}$ such that
\begin{equation*}\label{eq 3.12}\tag{3.12}
|Y_{t}|\le |Y_{t}|+\int_{0}^{t}\theta_{s}ds\le \widetilde{C}\mathbb{E}_{t}\Big[|\xi|+\int_{0}^{T}\theta_{s}ds\Big]+\widetilde{C},\ t\in[0,T],\ \mathbb{P}\text{-}a.s.
\end{equation*}
\end{theorem}
\begin{proof}  \noindent \textbf{Step 1.}\ We implement the localization procedure, as outlined in \cite[Theorem 2]{BH06}, to construct a solution. For all $n,k\ge1$ and $(\omega,t,y,z,\mu)\in\Omega\times[0,T]\times\mathbb{R}\times\mathbb{R}^{d}\times\mathcal{P}_{1}(\mathbb{R})$, we set
\begin{equation*}
\xi^{n,k}:=\xi^{+}\wedge n-\xi^{-}\wedge k\ \ \text{and}\ \ f^{n,k}(\omega,t,y,z,\mu):=f^{+}(\omega,t,y,z,\mu)\wedge n-f^{-}(\omega,t,y,z,\mu)\wedge k.
\end{equation*}

\noindent Obviously, for all $(\omega,t,y,z,\mu)\in\Omega\times[0,T]\times\mathbb{R}\times\mathbb{R}^{d}\times\mathcal{P}_{1}(\mathbb{R})$,
$$|\xi^{n,k}|\le|\xi|\wedge(n\vee k)\ \ \text{and}\ \ |f^{n,k}(\omega,t,y,z,\mu)|\le|f(\omega,t,y,z,\mu)|\wedge(n\vee k),$$
and the generator $f^{n,k}$ satisfies assumptions \textbf{(A1)}, \textbf{(A2)$_{m}^{\lambda}$} and \textbf{(A3)} with $\theta$ being replaced by $\theta\wedge(n\vee k),$ for all $n,k\ge1$.

Simultaneously, based on the definition of $f^{n,k}$ and the fact that $f$ satisfies assumption \textbf{(A4)}, it is not difficult to verify that, for all $y\in\mathbb{R}$, $z\in\mathbb{R}^{d}$, and $\eta,\ \bar{\eta}\in L^{1}_{\mathcal{F}}(\Omega;\mathbb{R}),$
\begin{equation*}
f^{n,k}(\omega,t,y,z,\mathbb{P}_{\eta})\le f^{n,k}(\omega,t,y,z,\mathbb{P}_{\bar{\eta}}),\ d\mathbb{P}\times dt\text{-}a.e.,\ \ \text{whenever}\ \ \eta\le\bar{\eta},\ \mathbb{P}\text{-}a.s.
\end{equation*}
This implies that the generator $f^{n,k}$ also satisfies assumption \textbf{(A4)}. Hence, as $\xi^{n,k}\ \text{and}\ f^{n,k}$ are bounded, by \cite[Theorem 4.1]{LX22} we obtain that the following mean-field BSDE \eqref{eq 3.13} has a maximal bounded solution $(Y^{n,k},Z^{n,k})\in \mathcal{S}_{\mathbb{F}}^{\infty}(0,T;\mathbb{R})\times \mathcal{H}_{\mathbb{F}}^{2}(0,T;\mathbb{R}^{d})$:
\begin{equation*}\label{eq 3.13}\tag{3.13}
\begin{aligned}
Y_{t}^{n,k}=\xi^{n,k}+\int_{t}^{T}f^{n,k}(s,Y_{s}^{n,k},Z_{s}^{n,k},\mathbb{P}_{Y_{s}^{n,k}})ds-\int_{t}^{T}Z_{s}^{n,k}dW_{s},\ t\in[0,T].
\end{aligned}
\end{equation*}

\noindent The comparison theorem (cf. \cite[Theorem 4.2, Remark 4.3]{LX22}) allows to conclude that $(Y_{t}^{n,k})_{t\in[0,T]}$ is nondecreasing in $n$ and non-increasing in $k$. Furthermore, it follows from Proposition \ref{prop 3.4} that there exists a constant $\widetilde{C}>0$ depending only on $m,\lambda,\gamma,K,T,\Vert\xi\Vert_{L^{1}(\Omega)}$ and $\Vert\theta\Vert_{\mathcal{L}_{\mathbb{F}}^{1}(0,T)}$ but not on $n,k$, such that
\begin{equation*}\label{eq 3.14}\tag{3.14}
\begin{aligned}
|Y_{t}^{n,k}|&\le |Y_{t}^{n,k}|+\int_{0}^{t}[\theta_{s}\wedge (n \vee k)]ds\le\widetilde{C}\mathbb{E}_{t}\Big[|\xi^{n,k}|+\int_{0}^{T}[\theta_{s}\wedge (n \vee k)]ds\Big]+\widetilde{C}\\
&\le \widetilde{C}\mathbb{E}_{t}\Big[|\xi|+\int_{0}^{T}\theta_{s}ds\Big]+\widetilde{C},\ t\in[0,T],\ \mathbb{P}\text{-}a.s.
\end{aligned}
\end{equation*}


\noindent For $m\ge 1$, we define the following stopping time:
\begin{equation*}
\sigma_{m}:=\inf\Big\{t\in[0,T]:\widetilde{C}\mathbb{E}_{t}\Big[|\xi|+\int_{0}^{T}\theta_{s}ds\Big]+\widetilde{C}\ge m\Big\}\wedge T.
\end{equation*}

\noindent Then, $(Y^{n,k}_{m}(t),Z^{n,k}_{m}(t))_{t\in[0,T]}:=(Y^{n,k}_{t\wedge \sigma_{m}},Z^{n,k}_{t}{\mathbf{1}}_{[0,\sigma_{m}]}(t))_{t\in[0,T]}\in\mathcal{S}_{\mathbb{F}}^{\infty}(0,T;\mathbb{R})\times \mathcal{H}_{\mathbb{F}}^{2}(0,T;\mathbb{R}^{d})$ solves the following BSDE:
\begin{equation*}
Y^{n,k}_{m}(t)=Y^{n,k}_{\sigma_{m}}+\int_{t}^{T}{\mathbf{1}}_{\{s\le\sigma_{m}\}}f^{n,k}(s,Y^{n,k}_{m}(s),Z^{n,k}_{m}(s),\mathbb{P}_{Y^{n,k}_{s}})ds-\int_{t}^{T}Z^{n,k}_{m}(s)dW_{s},\ t\in[0,T].
\end{equation*}

\noindent Obviously, it can be observed that for all $n,k\ge1$, $(Y^{n,k}_{m}(t))_{t\in[0,T]}$ still increases with respect to $n$ and decreases with respect to $k$. For $k\ge1\ \text{and}\ t\in[0,T]$, we define
\begin{equation*}\label{eq 3.15}\tag{3.15}
\begin{aligned}
Y^{k}_{t}:=\sup\limits_{n\ge1}Y_{t}^{n,k}=\lim\limits_{n\rightarrow+\infty}Y_{t}^{n,k},\ \ \text{and}\ \ Y^{k}_{m}(t):=Y^{k}_{t\wedge\sigma_{m}}=\lim\limits_{n\rightarrow+\infty}Y^{n,k}_{m}(t).
\end{aligned}
\end{equation*}

\noindent It follows from \eqref{eq 3.14} that, for $t\in[0,T]$ and $n,k\ge1$,
\begin{equation*}\label{eq 3.16}\tag{3.16}
\begin{aligned}
\mathbb{E}[|Y_{t}^{n,k}|]\le \widetilde{C}\mathbb{E}\Big[|\xi|+\int_{0}^{T}\theta_{s}ds\Big]+\widetilde{C}=:\overline{C},
\end{aligned}
\end{equation*}

\noindent and from Fatou's lemma, also $\mathbb{E}[|Y_{t}^{k}|]\le\overline{C},\ t\in[0,T],\ k\ge1$. Simultaneously, due to \eqref{eq 3.14}--\eqref{eq 3.15} and Lebesgue's dominated convergence theorem, we obtain
\begin{equation*}\label{eq 3.17}\tag{3.17}
\begin{aligned}
\lim\limits_{n\rightarrow+\infty}\mathbb{E}[|Y_{t}^{n,k}-Y_{t}^{k}|]=0,\ t\in[0,T],\ k\ge1.
\end{aligned}
\end{equation*}

\noindent Moreover, according to the definition of $\sigma_{m}$ and \eqref{eq 3.15},
\begin{equation*}\label{eq 3.18}\tag{3.18}
\begin{aligned}
|Y_{m}^{n,k}(t)|\le m,\ \ \text{and}\ \ |Y^{k}_{m}(t)|\le m,\ t\in[0,T],\ n,k\ge1,\ m\ge1.
\end{aligned}
\end{equation*}

\noindent Then, for all $(\omega,t,y,z)\in\Omega\times[0,T]\times\mathbb{R}\times\mathbb{R}^{d}$, we define the driving coefficients
\begin{equation*}
\begin{aligned}
g^{n,k}(\omega,t,y,z):={\mathbf{1}}_{\{t\le\sigma_{m}(\omega)\}}f^{n,k}(\omega,t,y,z,\mathbb{P}_{Y^{n,k}_{t}}),\ \text{and}\ g^{k}(\omega,t,y,z):={\mathbf{1}}_{\{t\le\sigma_{m}(\omega)\}}f^{k}(\omega,t,y,z,\mathbb{P}_{Y^{k}_{t}}),
\end{aligned}
\end{equation*}
where $f^{k}:=f\vee(-k)=f^{+}-(f^{-}\wedge k)$. Then, thanks to our assumptions on $f$ and \eqref{eq 3.17}, we have for all $(y,z)\in[-m,m]\times \mathbb{R}^{d}$,
\begin{equation*}
\begin{aligned}
\sup\limits_{n\ge1}|g^{n,k}(\omega,t,y,z)|\le \theta_{t}(\omega)+\psi(m\vee\overline{C})+\gamma_{0}|z|^{2},\ d\mathbb{P}\times dt\text{-}a.e.
\end{aligned}
\end{equation*}
and
\begin{equation*}
\begin{aligned}
&|g^{n,k}(\omega,t,y_{n},z_{n})-g^{k}(\omega,t,y,z)|\le |f^{n,k}(\omega,t,y_{n},z_{n},\mathbb{P}_{Y^{k}_{t}})-f^{k}(\omega,t,y,z,\mathbb{P}_{Y^{k}_{t}})|\\
&\hspace{8em}+\rho\big(\mathbb{E}[|Y_{t}^{n,k}-Y_{t}^{k}|]\big)\rightarrow0,\ \ \text{as}\ (y_{n},z_{n})\rightarrow(y,z),\ n\rightarrow+\infty.\\
\end{aligned}
\end{equation*}

\noindent Consequently, from Lemma \ref{le 3.5} we get that there exists $Z_{m}^{k}\in\mathcal{H}_{\mathbb{F}}^{2}(0,T;\mathbb{R}^{d})$ such that $\mathbb{E}\big[\int_{0}^{T}|Z_{m}^{n,k}(s)-Z_{m}^{k}(s)|^{2}ds\big]\rightarrow0,\ n\rightarrow+\infty,$ and $(Y_{m}^{k},Z_{m}^{k})\in\mathcal{S}_{\mathbb{F}}^{\infty}(0,T;\mathbb{R})\times \mathcal{H}_{\mathbb{F}}^{2}(0,T;\mathbb{R}^{d})$ is a solution of the BSDE
\begin{equation*}
\begin{aligned}
Y_{m}^{k}(t)=Y^{k}_{\sigma_{m}}+\int_{t}^{T}{\mathbf{1}}_{\{s\le\sigma_{m}\}}f^{k}(s,Y^{k}_{m}(s),Z^{k}_{m}(s),\mathbb{P}_{Y^{k}_{s}})ds-\int_{t}^{T}Z_{m}^{k}(s)dW_{s},\ t\in[0,T].
\end{aligned}
\end{equation*}

\noindent Thanks to \eqref{eq 3.18} we can now apply the same argument for $Y_{t}:=\inf_{k\ge1}Y_{t}^{k}=\lim_{k\rightarrow+\infty}Y_{t}^{k}$, and $Y_{m}(t):=Y_{t\wedge\sigma_{m}}=\lim_{k\rightarrow+\infty}Y_{m}^{k}(t),\ t\in[0,T]$ and for the associated generator $g(\omega,t,y,z):={\mathbf{1}}_{\{t\le\sigma_{m}(\omega)\}}f(\omega,t,y,z,\mathbb{P}_{Y_{t}}),\ (\omega,t,y,z)\in\Omega\times[0,T]\times\mathbb{R}\times\mathbb{R}^{d}$ (note that Fatou's lemma extends \eqref{eq 3.16} to $Y^{k}$ and $Y$), and conclude that there is some $Z_{m}\in\mathcal{H}_{\mathbb{F}}^{2}(0,T;\mathbb{R}^{d})$ such that $\mathbb{E}\big[\int_{0}^{T}|Z_{m}^{k}(s)-Z_{m}(s)|^{2}ds\big]\rightarrow0,\ k\rightarrow+\infty,$ and $(Y_{m},Z_{m})\in\mathcal{S}_{\mathbb{F}}^{\infty}(0,T;\mathbb{R})\times \mathcal{H}_{\mathbb{F}}^{2}(0,T;\mathbb{R}^{d})$ solves
\begin{equation*}\label{eq 3.19}\tag{3.19}
\begin{aligned}
Y_{m}(t)=Y_{\sigma_{m}}+\int_{t}^{T}{\mathbf{1}}_{\{s\le\sigma_{m}\}}f(s,Y_{m}(s),Z_{m}(s),\mathbb{P}_{Y_{s}})ds-\int_{t}^{T}Z_{m}(s)dW_{s},\ t\in[0,T].
\end{aligned}
\end{equation*}

\noindent Recall that
\begin{equation*}
\begin{aligned}
Z_{m}^{n,k}(=Z^{n,k}\mathbf{1}_{[0,\sigma_{m}]})\ \mathop{\longrightarrow}\limits_{n\rightarrow+\infty}^{\mathcal{H}_{\mathbb{F}}^{2}}Z_{m}^{k}\ \mathop{\longrightarrow}\limits_{k\rightarrow+\infty}^{\mathcal{H}_{\mathbb{F}}^{2}}\ Z_{m},\ m\ge1.
\end{aligned}
\end{equation*}

\noindent Hence, as $\sigma_{m}\le \sigma_{m+1}\le T,\ m\ge1$, $Z_{m+1}^{n,k}{\mathbf{1}}_{[0,\sigma_{m}]}=Z_{m}^{n,k}$ and $Z_{m+1}{\mathbf{1}}_{[0,\sigma_{m}]}=Z_{m},\ m\ge1$. Thus, the fact that $\mathbb{P}\big(\mathop{\cup}\limits_{m\ge1}\{\sigma_{m}=T\}\big)=1$, allows to define the $\mathbb{F}$-progressively measurable process
\begin{equation*}
Z_{t}:=Z_{1}(t)\mathbf{1}_{[0,\sigma_{1}]}(t)+\sum\limits_{m\ge2}Z_{m}(t)\mathbf{1}_{(\sigma_{m-1},\sigma_{m}]}(t),\ t\in[0,T],
\end{equation*}
and $\mathbb{E}\big[\int_{0}^{\sigma_{m}}|Z_{s}|^{2}ds\big]=\mathbb{E}\big[\int_{0}^{T}|Z_{m}(s)|^{2}ds\big]<+\infty,\ m\ge1.$ Therefore, $\int_{0}^{T}|Z_{s}|^{2}ds<+\infty,\ \mathbb{P}\text{-}a.s.$ Moreover, form \eqref{eq 3.19}, as $Y_{T}=\xi$,
\begin{equation*}
\begin{aligned}
Y_{t}=\xi+\int_{t}^{T}f(s,Y_{s},Z_{s},\mathbb{P}_{Y_{s}})ds-\int_{t}^{T}Z_{s}dW_s,\ t\in[0,T],
\end{aligned}
\end{equation*}
and from \eqref{eq 3.11} obtained by taking the limits in \eqref{eq 3.14} we show that $Y$ is of the class (D).

\noindent \textbf{Step 2.}\ It remains to prove that $(Y,Z)\in\mathcal{S}^{p}_{\mathbb{F}}(0,T;\mathbb{R})\times \mathcal{H}^{p}_{\mathbb{F}}(0,T;\mathbb{R}^{d})$, for all $p\in(0,1)$. In view of \eqref{eq 3.11}, we deduce from \cite[Lemma 6.1]{BD03} that, for all $p\in(0,1)$,
\begin{equation*}
\begin{aligned}
\mathbb{E}\big[\sup\limits_{t\in[0,T]}|Y_{t}|^{p}\big]\le\mathbb{E}\Big[\sup\limits_{t\in[0,T]}\Big(|Y_{t}|+\int_{0}^{t}\theta_{s}ds\Big)^{p}\Big]\le (2\widetilde{C})^{p}\bigg[\frac{1}{1-p}\Big(\mathbb{E}\Big[|\xi|+\int_{0}^{T}\theta_{s}ds\Big]\Big)^{p}+1\bigg]<+\infty.
\end{aligned}
\end{equation*}

\noindent This implies that $Y\in\mathcal{S}^{p}_{\mathbb{F}}(0,T;\mathbb{R})$, for all $p\in(0,1)$.

It remains to show that $Z\in\mathcal{H}_{\mathbb{F}}^{p}(0,T;\mathbb{R}^{d})$, for all $p\in(0,1)$. For all $n\ge1$, we introduce the stopping time: $\tau_{n}:=\inf\big\{t\in[0,T]:\int_{0}^{t}|Z_{s}|^{2}ds\ge n\big\}\wedge T.$ Applying It\^{o}'s formula to $|Y_{t}|^{2}$ on $[0,\tau_{n}]$ yields
\begin{equation*}\label{eq 3.20}\tag{3.20}
|Y_{0}|^{2}+\int_{0}^{\tau_{n}}|Z_{s}|^{2}ds=|Y_{\tau_{n}}|^{2}+2\int_{0}^{\tau_{n}}Y_{s}f(s,Y_{s},Z_{s},\mathbb{P}_{Y_{s}})ds-2\int_{0}^{\tau_{n}}Y_{s}Z_{s}dW_{s}.
\end{equation*}

\noindent Then, due to assumption \textbf{(A2)$_{m}^{\lambda}$} and Young's inequality, we see that, $d\mathbb{P}\times dt\text{-}a.e.$,
\begin{equation*}
\begin{aligned}
2Y_{s}f(s,Y_{s},Z_{s},\mathbb{P}_{Y_{s}})
\le 2\theta_{s}|Y_{s}|+(2K+2\gamma^{2})|Y_{s}|^{2}+2K|Y_{s}|W_{1}(\mathbb{P}_{Y_{s}},\delta_{0})+\frac{1}{2}|Z_{s}|^{2}.
\end{aligned}
\end{equation*}

\noindent Hence, setting $Y_{T}^{\ast}:=\sup\limits_{s\in[0,T]}|Y_{s}|$, \eqref{eq 3.20} yields
\begin{equation*}
\begin{aligned}
\int_{0}^{\tau_{n}}|Z_{s}|^{2}ds&\le(4KT+4\gamma^{2}T+2)|Y_{T}^{\ast}|^{2}+4Y_{T}^{\ast}\int_{0}^{T}\theta_{s}ds\\
&\hspace{2em}+4KY_{T}^{\ast}\int_{0}^{T}\mathbb{E}[|Y_{s}|]ds+4\sup\limits_{t\in[0,\tau_{n}]}\Big|\int_{0}^{t}Y_{s}Z_{s}dW_{s}\Big|.
\end{aligned}
\end{equation*}

\noindent Furthermore, for all $p\in(0,1)$, by taking the $\frac{p}{2}$-th power of both sides of the above inequality and then the expectation, we see that there exists $c_{p}>0$ depending only on $p,\gamma,K,T$ such that
\begin{equation*}\label{eq 3.21}\tag{3.21}
\begin{aligned}
\mathbb{E}\Big[\Big(\int_{0}^{\tau_{n}}|Z_{s}|^{2}ds\Big)^{\frac{p}{2}}\Big]&\le c_{p}\bigg(\mathbb{E}[|Y_{T}^{\ast}|^{p}]+\mathbb{E}\Big[|Y_{T}^{\ast}|^{\frac{p}{2}}\Big(\int_{0}^{T}\theta_{s}ds\Big)^{\frac{p}{2}}\Big]\\
&\hspace{2em}+\mathbb{E}\Big[|Y_{T}^{\ast}|^{\frac{p}{2}}\Big(\int_{0}^{T}\mathbb{E}[|Y_{s}|]ds\Big)^{\frac{p}{2}}\Big]+\mathbb{E}\Big[\sup\limits_{t\in[0,\tau_{n}]}\Big|\int_{0}^{t}Y_{s}Z_{s}dW_{s}\Big|^{\frac{p}{2}}\Big]\bigg).
\end{aligned}
\end{equation*}

\noindent Moreover, it follows from Burkholder-Davis-Gundy's inequality and Young's inequality that
\begin{equation*}
\begin{aligned}
c_{p}\mathbb{E}\Big[\sup\limits_{t\in[0,\tau_{n}]}\Big|\int_{0}^{t}Y_{s}Z_{s}dW_{s}\Big|^{\frac{p}{2}}\Big]\le\frac{9}{2}c_{p}^{2}\mathbb{E}[|Y_{T}^{\ast}|^{p}]+\frac{1}{2}\mathbb{E}\Big[\Big(\int_{0}^{\tau_{n}}|Z_{s}|^{2}ds\Big)^{\frac{p}{2}}\Big].
\end{aligned}
\end{equation*}

\noindent Plugging the latter inequality into \eqref{eq 3.21} and noting that $\sup_{t\in[0,T]}\mathbb{E}[|Y_{t}|]<+\infty$, then for all $p\in(0,1)$, we get from H\"{o}lder's inequality that there exists a suitable constant $d_{p}>0$ such that
\begin{equation*}\label{eq 3.22}\tag{3.22}
\begin{aligned}
\mathbb{E}\Big[\Big(\int_{0}^{\tau_{n}}|Z_{s}|^{2}ds\Big)^{\frac{p}{2}}\Big]\le d_{p}\Big(\mathbb{E}[|Y_{T}^{\ast}|^{p}]+\big(\mathbb{E}[|Y_{T}^{\ast}|^{p}]\big)^{\frac{1}{2}}+\big(\mathbb{E}[|Y_{T}^{\ast}|^{p}]\big)^{\frac{1}{2}}\cdot\Big(\mathbb{E}\Big[\int_{0}^{T}\theta_{s}ds\Big]\Big)^{\frac{p}{2}}\Big).
\end{aligned}
\end{equation*}

\noindent Since $Y\in\mathcal{S}^{p}_{\mathbb{F}}(0,T;\mathbb{R})$, for all $p\in(0,1)$ and $\theta\in\mathcal{L}^{1}_{\mathbb{F}}(0,T;\mathbb{R})$, by letting $n$ tend to infinity in \eqref{eq 3.22} and making use of Fatou's lemma, we conclude that $Z\in\mathcal{H}_{\mathbb{F}}^{p}(0,T;\mathbb{R}^{d})$, for all $p\in(0,1)$.
Consequently, the proof of the existence theorem is now complete.
\end{proof}

The uniqueness of the $L^{1}$ solution to the one-dimensional mean-field BSDE \eqref{eq 1.2} is a direct consequence of the following Theorem \ref{th 3.7}.

\begin{theorem}\label{th 3.7}
(Comparison theorem) Suppose that $\xi_{1},\ \xi_{2}\in L^{1}_{\mathcal{F}_{T}}(\Omega;\mathbb{R})$ and $f^{1}$ (resp., $f^{2}$) satisfies assumptions \textbf{(A5)}, \textbf{(A6)$_{m}^{\lambda}$} and \textbf{(A7)} with some $m\ge2\ \text{and}\ \lambda>\frac{1}{2}$. We denote by $(Y^{1},Z^{1})$ and $(Y^{2},Z^{2})$ the solutions of the mean-field BSDE \eqref{eq 1.2} with parameters $(\xi_{1},f^{1})$ and $(\xi_{2},f^{2})$, resp., such that $(Y^{1}-Y^{2})^{+}$ is of the class (D). If $\xi_{1}\le\xi_{2},\ \mathbb{P}\text{-}a.s.$, and $d\mathbb{P}\times dt\text{-}a.e.$,
\begin{equation*}
\begin{aligned}
&{\bf{1}}_{\{Y_{t}^{1}> Y_{t}^{2}\}}\big(f^{1}(t,Y_{t}^{2},Z_{t}^{2},\mathbb{P}_{Y_{t}^{2}})-f^{2}(t,Y_{t}^{2},Z_{t}^{2},\mathbb{P}_{Y_{t}^{2}})\big)\le0\\
\big(\text{resp.,\ }&{\bf{1}}_{\{Y_{t}^{1}> Y_{t}^{2}\}}\big(f^{1}(t,Y_{t}^{1},Z_{t}^{1},\mathbb{P}_{Y_{t}^{1}})-f^{2}(t,Y_{t}^{1},Z_{t}^{1},\mathbb{P}_{Y_{t}^{1}})\big)\le0\big),
\end{aligned}
\end{equation*}
then it holds that $Y_{t}^{1}\le Y_{t}^{2}$, $t\in[0,T],\ \mathbb{P}\text{-}a.s.$
\end{theorem}

\begin{proof}
Without loss of generality, we assume that the generator $f^{1}$ satisfies assumptions \textbf{(A5)}, \textbf{(A6)$_{m}^{\lambda}$} and \textbf{(A7)} with some $m\ge2\ \text{and}\ \lambda>\frac{1}{2}$, and $d\mathbb{P}\times dt\text{-}a.e.$,
\begin{equation*}\label{eq 3.23}\tag{3.23}
\begin{aligned}
{\bf{1}}_{\{Y_{t}^{1}> Y_{t}^{2}\}}\big(f^{1}(t,Y_{t}^{2},Z_{t}^{2},\mathbb{P}_{Y_{t}^{2}})-f^{2}(t,Y_{t}^{2},Z_{t}^{2},\mathbb{P}_{Y_{t}^{2}})\big)\le0.\\
\end{aligned}
\end{equation*}

\noindent We remark that, due to \textbf{(A5)} and \textbf{(A6)$_{m}^{\lambda}$}, there is a constant $J>0$ such that, for all $x\ge0$,
\begin{equation*}\label{eq 3.24}\tag{3.24}
\begin{aligned}
\kappa(x)\le Jx+J\ \ \text{and}\ \ \zeta(x)\le Jx+J.
\end{aligned}
\end{equation*}

\noindent Moreover, according to (iii) in Remark \ref{re 3.1}, we can also suppose that the generator $f^{1}$ satisfies assumption \textbf{(A6)$_{m}^{\lambda}$}, where $\mathcal{IL}_{m}^{\lambda}$ is substituted by $\mathcal{IL}_{m,h}^{\lambda}$, for a sufficiently large constant $h>e^{(m)}$, which depends only on $m,\lambda,J$.

\noindent \textbf{Step 1.} First, we show that $(Y^{1}-Y^{2})^{+}$ is a bounded process.

\noindent We put $\widehat{Y}:=Y^{1}-Y^{2}$ and $\widehat{Z}:=Z^{1}-Z^{2}$. Then, $(\widehat{Y},\widehat{Z})$ solves the following BSDE
\begin{equation*}
\begin{aligned}
\widehat{Y}_{t}=\xi_{1}-\xi_{2}+\int_{t}^{T}\big(f^{1}(s,Y_{s}^{1},Z_{s}^{1},\mathbb{P}_{Y_{s}^{1}})-f^{2}(s,Y_{s}^{2},Z_{s}^{2},\mathbb{P}_{Y_{s}^{2}})\big)ds-\int_{t}^{T}\widehat{Z}_{s}dW_{s},\ t\in[0,T].
\end{aligned}
\end{equation*}

\noindent Putting $\overline{Y}_{t}:=\widehat{Y}_{t}^{+}+2Jt\ \text{and}\ \overline{Z}_{t}:=\textbf{1}_{\{\widehat{Y}_{t}>0\}}\widehat{Z}_{t},\ t\in[0,T],$ and using It\^{o}-Tanaka's formula applied to the term $\overline{Y}_{t}$, we get
\begin{equation*}
\begin{aligned}
\overline{Y}_{t}&=\overline{Y}_{T}+\int_{t}^{T}\Big[\textbf{1}_{\{\widehat{Y}_{s}>0\}}\big(f^{1}(s,Y_{s}^{1},Z_{s}^{1},\mathbb{P}_{Y_{s}^{1}})-f^{2}(s,Y_{s}^{2},Z_{s}^{2},\mathbb{P}_{Y_{s}^{2}})\big)-2J\Big]ds\\
&\hspace{2em}-\int_{t}^{T}\overline{Z}_{s}dW_{s}-\frac{1}{2}\int_{t}^{T}d\widehat{L}_{s},\ t\in[0,T],
\end{aligned}
\end{equation*}
where the process $\widehat{L}=(\widehat{L}_{t})_{t\in[0,T]}$ denotes the local time of $\widehat{Y}$ at $0$.

Thanks to assumptions \textbf{(A5)}, \textbf{(A6)$_{m}^{\lambda}$} and \textbf{(A7)} with $m\ge2\ \text{and}\ \lambda>\frac{1}{2}$, together with \eqref{eq 3.23}--\eqref{eq 3.24}, we obtain that, $d\mathbb{P}\times dt\text{-}a.e.$,
\begin{equation*}\label{eq 3.25}\tag{3.25}
\begin{aligned}
&\hspace{1.2em}\textbf{1}_{\{\widehat{Y}_{t}>0\}}\big(f^{1}(t,Y_{t}^{1},Z_{t}^{1},\mathbb{P}_{Y_{t}^{1}})-f^{2}(t,Y_{t}^{2},Z_{t}^{2},\mathbb{P}_{Y_{t}^{2}})\big)-2J\\
&\le \kappa(\widehat{Y}_{t}^{+})+\textbf{1}_{\{\widehat{Y}_{t}>0\}}\zeta\big(\frac{|\widehat{Z}_{t}|}{\mathcal{IL}_{m,h}^{\lambda}(|\widehat{Z}_{t}|)}\big)+L\mathbb{E}[\widehat{Y}_{t}^{+}]-2J\\
&\le J\widehat{Y}_{t}^{+}+\textbf{1}_{\{\widehat{Y}_{t}>0\}}\frac{J|\widehat{Z}_{t}|}{\mathcal{IL}_{m,h}^{\lambda}(|\widehat{Z}_{t}|)}+L\mathbb{E}[\widehat{Y}_{t}^{+}]\le J\overline{Y}_{t}+\frac{J|\overline{Z}_{t}|}{\mathcal{IL}_{m,h}^{\lambda}(|\overline{Z}_{t}|)}+L\mathbb{E}[\overline{Y}_{t}].
\end{aligned}
\end{equation*}

\noindent Let $\Psi(\cdot,\cdot)$ be the test function defined by \eqref{eq 3.1} in the proof of Proposition \ref{prop 3.4}, except that $K$ and $\gamma$ are replaced by $J$. A straightforward calculation similar to \eqref{eq 3.2} shows that, for all $(t,x)\in[0,T]\times[0,+\infty)$,
\begin{equation*}\label{eq 3.26}\tag{3.26}
\begin{aligned}
0<\partial_{x}\Psi(t,x)&=\bigg[1-\frac{1}{\big(\ln^{(m)}(h+x)\big)^{2\lambda-1}}\Big(1-\frac{2\lambda-1}{\prod_{j=1}^{m}\ln^{(j)}(h+x)}\Big)\bigg]\\
&\hspace{2.6em}\times\exp\Big\{2\Big(J+\frac{2J^{2}}{2\lambda-1}\Big)t\Big\}\le Q_{1},
\end{aligned}
\end{equation*}
where $Q_{1}:=\exp\big\{2\big(J+\frac{2J^{2}}{2\lambda-1}\big)T\big\}.$ Then, applying It\^{o}-Tanaka's formula to $\Psi(t,\overline{Y}_{t})$ combined with \eqref{eq 3.25}--\eqref{eq 3.26} implies that
\begin{equation*}
\begin{aligned}
d\Psi(t,\overline{Y}_{t})&=-\partial_{x}\Psi(t,\overline{Y}_{t})\Big[\textbf{1}_{\{\widehat{Y}_{t}>0\}}\big(f^{1}(t,Y_{t}^{1},Z_{t}^{1},\mathbb{P}_{Y_{t}^{1}})-f^{2}(t,Y_{t}^{2},Z_{t}^{2},\mathbb{P}_{Y_{t}^{2}})\big)-2J\Big]dt\\
&\hspace{2em}+\frac{1}{2}\partial_{x}\Psi(t,\overline{Y}_{t})d\widehat{L}_{t}+\frac{1}{2}\partial_{xx}\Psi(t,\overline{Y}_{t})|\overline{Z}_{t}|^{2}dt+\partial_{t}\Psi(t,\overline{Y}_{t})dt+\partial_{x}\Psi(t,\overline{Y}_{t})\overline{Z}_{t}dW_{t}\\
&\ge \bigg[-J\partial_{x}\Psi(t,\overline{Y}_{t})\overline{Y}_{t}-\partial_{x}\Psi(t,\overline{Y}_{t})\frac{J|\overline{Z}_{t}|}{\mathcal{IL}_{m,h}^{\lambda}(|\overline{Z}_{t}|)}+\frac{1}{2}\partial_{xx}\Psi(t,\overline{Y}_{t})|\overline{Z}_{t}|^{2}+\partial_{t}\Psi(t,\overline{Y}_{t})\bigg]dt\\
&\hspace{2em}-L\partial_{x}\Psi(t,\overline{Y}_{t})\cdot\mathbb{E}[\overline{Y}_{t}]dt+\partial_{x}\Psi(t,\overline{Y}_{t})\overline{Z}_{t}dW_{t},\ t\in[0,T].\\
\end{aligned}
\end{equation*}

\noindent Analogous to \eqref{eq 3.5}, for $h>e^{(m)}$ sufficiently large, we deduce that
\begin{equation*}\label{eq 3.27}\tag{3.27}
\begin{aligned}
d\Psi(t,\overline{Y}_{t})\ge-J\partial_{x}\Psi(t,\overline{Y}_{t})\cdot\mathbb{E}[\overline{Y}_{t}]dt+\partial_{x}\Psi(t,\overline{Y}_{t})\overline{Z}_{t}dW_{t},\ t\in[0,T].
\end{aligned}
\end{equation*}

\noindent Next, for every $n\ge1$ and $t\in[0,T]$, we define the following stopping time:
\begin{equation*}
\begin{aligned}
\overline{\tau}_{n}^{t}:=\inf\Big\{s\in[t,T]:\int_{t}^{s}|\partial_{x}\Psi(r,\overline{Y}_{r})|^{2}|\overline{Z}_{r}|^{2}dr\ge n\Big\}\wedge T.
\end{aligned}
\end{equation*}
Then, it follows from \eqref{eq 3.27} and the definition of $\overline{\tau}^{t}_{n}$ that
\begin{equation}\label{eq 3.28}\tag{3.28}
\begin{aligned}
\Psi(t,\overline{Y}_{t})\le \mathbb{E}_{t}\Big[\Psi(\overline{\tau}_{n}^{t},\overline{Y}_{\overline{\tau}_{n}^{t}})+J\int_{t}^{\overline{\tau}_{n}^{t}}\partial_{x}\Psi(s,\overline{Y}_{s})\cdot\mathbb{E}[\overline{Y}_{s}]ds\Big],\ t\in[0,T].
\end{aligned}
\end{equation}

\noindent Similar to \eqref{eq 3.7}, from the definition of $\Psi(\cdot,\cdot)$, we have, for a sufficiently large constant $h>e^{(m)}$,
\begin{equation}\label{eq 3.29}\tag{3.29}
\begin{aligned}
\frac{1}{2}(h+x)\le \Psi(t,x)\le Q_{1}(h+x),\ (t,x)\in [0,T]\times [0,+\infty).
\end{aligned}
\end{equation}

\noindent By inserting \eqref{eq 3.29} into \eqref{eq 3.28} and combining this with \eqref{eq 3.26}, we obtain that
\begin{equation}\label{eq 3.30}\tag{3.30}
\begin{aligned}
\frac{1}{2}(h+\overline{Y}_{t})\le \Psi(t,\overline{Y}_{t})&\le \mathbb{E}_{t}\Big[\Psi(\overline{\tau}_{n}^{t},\overline{Y}_{\overline{\tau}_{n}^{t}})+J\int_{t}^{\overline{\tau}_{n}^{t}}\partial_{x}\Psi(s,\overline{Y}_{s})\cdot\mathbb{E}[\overline{Y}_{s}]ds\Big]\\
&\le Q_{1}h+Q_{1}\mathbb{E}_{t}[\overline{Y}_{\overline{\tau}_{n}^{t}}]+JQ_{1}\int_{t}^{T}\mathbb{E}[\overline{Y}_{s}]ds,\ t\in[0,T].
\end{aligned}
\end{equation}

\noindent As $\overline{Y}$ is of the class (D), and $\overline{\tau}_{n}^{t}\uparrow T$, as $n\rightarrow +\infty$, $t\in[0,T]$, passing to the limit in \eqref{eq 3.30} and recalling the definition of $\overline{Y}$ together with the fact that $\xi_{1}\le\xi_{2},\ \mathbb{P}\text{-}a.s.$, we have
\begin{equation*}\label{eq 3.31}\tag{3.31}
\begin{aligned}
\overline{Y}_{t}\le 2Q_{1}h+4Q_{1}JT+2JQ_{1}\int_{t}^{T}\mathbb{E}[\overline{Y}_{s}]ds,\ t\in[0,T].
\end{aligned}
\end{equation*}

\noindent Furthermore, taking the expectation and using (backward) Gronwall's inequality yields
\begin{equation*}\label{eq 3.32}\tag{3.32}
\mathbb{E}[\overline{Y}_{t}]\le C^{\ast}:=\big(2Q_{1}h+4Q_{1}JT\big)e^{2JQ_{1}T},\ t\in[0,T].
\end{equation*}

\noindent Finally, we substitute \eqref{eq 3.32} in \eqref{eq 3.31} to conclude that
\begin{equation*}\label{eq 3.33}\tag{3.33}
(Y_{t}^{1}-Y_{t}^{2})^{+}\le (Y_{t}^{1}-Y_{t}^{2})^{+}+2Jt\le 2Q_{1}(h+2JT+JC^{\ast}T),\ t\in[0,T].
\end{equation*}

\noindent Hence, $(Y^{1}-Y^{2})^{+}$ is bounded.

\noindent \textbf{Step 2.} For every $n\ge1$, we define now the following stopping time:
\begin{equation*}
\begin{aligned}
\tau_{n}:=\inf\Big\{t\in[0,T]:\int_{0}^{t}\big(|Z_{s}^{1}|^{2}+|Z_{s}^{2}|^{2}\big)ds\ge n\Big\}\wedge T.
\end{aligned}
\end{equation*}

\noindent By applying It\^{o}-Tanaka's formula, we obtain that, for all $t\in[0,T]$,
\begin{equation*}\label{eq 3.34}\tag{3.34}
\begin{aligned}
\widehat{Y}_{t\wedge\tau_{n}}^{+}\le\widehat{Y}_{\tau_{n}}^{+}+\int_{t\wedge\tau_{n}}^{\tau_{n}}\textbf{1}_{\{\widehat{Y}_{s}>0\}}\big(f^{1}(s,Y_{s}^{1},Z_{s}^{1},\mathbb{P}_{Y_{s}^{1}})-f^{2}(s,Y_{s}^{2},Z_{s}^{2},\mathbb{P}_{Y_{s}^{2}})\big)ds-\int_{t\wedge\tau_{n}}^{\tau_{n}}\textbf{1}_{\{\widehat{Y}_{s}>0\}}\widehat{Z}_{s}dW_{s}.
\end{aligned}
\end{equation*}

\noindent From \eqref{eq 3.25}, we have, $d\mathbb{P}\times dt\text{-}a.e.$,
\begin{equation*}\label{eq 3.35}\tag{3.35}
\begin{aligned}
\textbf{1}_{\{\widehat{Y}_{t}>0\}}\big(f^{1}(t,Y_{t}^{1},Z_{t}^{1},\mathbb{P}_{Y_{t}^{1}})-f^{2}(t,Y_{t}^{2},Z_{t}^{2},\mathbb{P}_{Y_{t}^{2}})\big)
\le\kappa(\widehat{Y}_{t}^{+})+\textbf{1}_{\{\widehat{Y}_{t}>0\}}\zeta(|\widehat{Z}_{t}|)+L\mathbb{E}[\widehat{Y}_{t}^{+}].\\
\end{aligned}
\end{equation*}

\noindent Moreover, from Lemma \ref{le 2.3} and \eqref{eq 3.24}, it follows that, for all $m\ge1$,
\begin{equation*}\label{eq 3.36}\tag{3.36}
\begin{aligned}
\zeta(x)\le (m+2J)x+\zeta\Big(\frac{2J}{m+2J}\Big),\ x\ge0.
\end{aligned}
\end{equation*}

\noindent Then, combining the results from \eqref{eq 3.34}--\eqref{eq 3.36}, we obtain that, for all $m\ge1$ and $t\in[0,T]$,
\begin{equation*}\label{eq 3.37}\tag{3.37}
\begin{aligned}
\widehat{Y}_{t\wedge\tau_{n}}^{+}&\le a_{m}+\widehat{Y}_{\tau_{n}}^{+}+\int_{t\wedge\tau_{n}}^{\tau_{n}}\big(\kappa(\widehat{Y}_{s}^{+})+\textbf{1}_{\{\widehat{Y}_{s}>0\}}(m+2J)|\widehat{Z}_{s}|+L\mathbb{E}[\widehat{Y}_{s}^{+}]\big)ds-\int_{t\wedge\tau_{n}}^{\tau_{n}}\textbf{1}_{\{\widehat{Y}_{s}>0\}}\widehat{Z}_{s}dW_{s}\\
&=a_{m}+\widehat{Y}_{\tau_{n}}^{+}+\int_{t\wedge\tau_{n}}^{\tau_{n}}\big(\kappa(\widehat{Y}_{s}^{+})+L\mathbb{E}[\widehat{Y}_{s}^{+}]\big)ds-\int_{t\wedge\tau_{n}}^{\tau_{n}}\textbf{1}_{\{\widehat{Y}_{s}>0\}}\widehat{Z}_{s}\cdot \big(dW_{s}-b_{s}^{m}ds\big),
\end{aligned}
\end{equation*}
where $a_{m}:=\zeta\big(\frac{2J}{m+2J}\big)T\rightarrow 0,\ \text{as}\ m\rightarrow +\infty,$ and $b_{s}^{m}:=\frac{(m+2J)\widehat{Z}_{s}^{\mathsf{T}}}{|\widehat{Z}_{s}|},\ \text{if}\ \widehat{Z}_{s}\neq0;$ $b_{s}^{m}:=0,\ \text{if}\ \widehat{Z}_{s}=0.$ Next, for all $m\ge1$, we define
\begin{equation*}
\begin{aligned}
W^{m}_{t}:=W_{t}-\int_{0}^{t}b_{s}^{m}ds,\ t\in[0,T],\ \ \text{and}\ \ d\mathbb{P}^{m}:=\exp\Big\{\int_{0}^{T}(b_{s}^{m})^{\mathsf{T}}dW_{s}-\frac{1}{2}\int_{0}^{T}|b_{s}^{m}|^{2}ds\Big\}d\mathbb{P}.
\end{aligned}
\end{equation*}

\noindent Due to Girsanov's theorem, $\mathbb{P}^{m}$ is equivalent to $\mathbb{P}$, and the process $W^{m}$ is a Brownian motion under $\mathbb{P}^{m}$. For simplicity, the expectation and conditional expectation with respect to $\mathcal{F}_{t}$ under $\mathbb{P}^{m}$ are denoted by $\mathbb{E}^{m}[\cdot]$ and $\mathbb{E}_{t}^{m}[\cdot]$, respectively.

By taking the conditional expectation of \eqref{eq 3.37} under $\mathbb{P}^{m}$, we deduce that, for all $t\in[0,T]$,
\begin{equation*}
\begin{aligned}
\widehat{Y}_{t\wedge\tau_{n}}^{+}\le a_{m}+\mathbb{E}^{m}_{t}[\widehat{Y}_{\tau_{n}}^{+}]+\mathbb{E}^{m}_{t}\Big[\int_{t\wedge\tau_{n}}^{\tau_{n}}\big(\kappa(\widehat{Y}_{s}^{+})+L\mathbb{E}[\widehat{Y}_{s}^{+}]\big)ds\Big],\ \mathbb{P}\text{-}a.s.
\end{aligned}
\end{equation*}

\noindent Hence, as $\tau_{n}\uparrow T$, as $n\rightarrow+\infty$, $\widehat{Y}^{+}$ is bounded, and $\widehat{Y}_{T}^{+}=(\xi_{1}-\xi_{2})^{+}=0,$ taking the limit in the latter inequality and applying Lebesgue's dominated convergence theorem yields, for all $t\in[0,T]$,
\begin{equation*}\label{eq 3.38}\tag{3.38}
\begin{aligned}
\widehat{Y}_{t}^{+}\le a_{m}+\mathbb{E}^{m}_{t}\Big[\int_{t}^{T}\big(\kappa(\widehat{Y}_{s}^{+})+L\mathbb{E}[\widehat{Y}_{s}^{+}]\big)ds\Big],\ \mathbb{P}\text{-}a.s.
\end{aligned}
\end{equation*}

\noindent Consequently, as $\kappa(\cdot)$ is nondecreasing,
\begin{equation*}\label{eq 3.39}\tag{3.39}
\begin{aligned}
\Vert \widehat{Y}^{+}_{t}\Vert_{\infty}\le a_{m}+\int_{t}^{T}\big(\kappa(\Vert\widehat{Y}_{s}^{+}\Vert_{\infty})+L\Vert \widehat{Y}^{+}_{s}\Vert_{\infty}\big)ds=a_{m}+\int_{t}^{T}\widetilde{\kappa}(\Vert \widehat{Y}^{+}_{s}\Vert_{\infty})ds,\ t\in[0,T],
\end{aligned}
\end{equation*}
where $\widetilde{\kappa}(x):=\kappa(x)+Lx,\ x\ge0$. It's easy to verify that $\widetilde{\kappa}(\cdot)$ is still a nondecreasing and concave function with $\widetilde{\kappa}(0)=0,\ \widetilde{\kappa}(u)>0,\ u>0$ and $\int_{0^{+}}\frac{du}{\widetilde{\kappa}(u)}=+\infty$ (we refer to Appendix in \cite{F13}).

Then it follows from Bihari's inequality (see Lemma \ref{le 2.4}) that
\begin{equation*}\label{eq 3.40}\tag{3.40}
\begin{aligned}
\Vert \widehat{Y}^{+}_{t}\Vert_{\infty}\le \widetilde{\Pi}^{-1}\big(\widetilde{\Pi}(a_{m})+T-t\big),\ t\in[0,T],
\end{aligned}
\end{equation*}
where $\widetilde{\Pi}(x):=\int_{1}^{x}\frac{1}{\widetilde{\kappa}(r)}dr,\ x>0.$ Finally, letting $m\rightarrow +\infty$ in \eqref{eq 3.40} yields $\Vert \widehat{Y}^{+}_{t}\Vert_{\infty}=0,\ t\in[0,T]$, i.e., $Y_{t}^{1}\le Y_{t}^{2},\ t\in[0,T],\ \mathbb{P}\text{-}a.s.$
\end{proof}

By combining the existence stated in Theorem \ref{th 3.6} with the comparison theorem (Theorem \ref{th 3.7}), we get the existence and the uniqueness of the $L^{1}$ solution for one-dimensional mean-field BSDEs \eqref{eq 1.2} with integrable parameters.

\begin{theorem}\label{th 3.8}
Let assumptions \textbf{(A1)}, \textbf{(A2)$_{m}^{\lambda}$}, \textbf{(A3)} and \textbf{(A4)} hold true with some $m\ge2\ \text{and}\ \lambda>\frac{1}{2}$. Then, for every given $\xi\in L_{\mathcal{F}_{T}}^{1}(\Omega;\mathbb{R})$, mean-field BSDE \eqref{eq 1.2} possesses a solution $(Y,Z)$ such that $Y$ belongs to the class (D) and $(Y,Z)\in\mathcal{S}^{p}_{\mathbb{F}}(0,T;\mathbb{R})\times \mathcal{H}^{p}_{\mathbb{F}}(0,T;\mathbb{R}^{d})$, for all $p\in(0,1)$. Moreover, there exists a constant $\widetilde{C}>0$ depending only on $m,\lambda,\gamma,K,T,\Vert\xi\Vert_{L^{1}(\Omega)}$ and $\Vert\theta\Vert_{\mathcal{L}_{\mathbb{F}}^{1}(0,T)}$ such that
\begin{equation*}
|Y_{t}|\le |Y_{t}|+\int_{0}^{t}\theta_{s}ds\le \widetilde{C}\mathbb{E}_{t}\Big[|\xi|+\int_{0}^{T}\theta_{s}ds\Big]+\widetilde{C},\ t\in[0,T],\ \mathbb{P}\text{-}a.s.
\end{equation*}

\noindent Furthermore, provided that also assumptions \textbf{(A5)}, \textbf{(A6)$_{m}^{\lambda}$} and \textbf{(A7)} are satisfied for the generator $f$, then the mean-field BSDE \eqref{eq 1.2} has a unique $L^{1}$ solution $(Y,Z)$.
\end{theorem}

\begin{proof} We only need to prove the uniqueness of the solution $(Y,Z)$.\\
\indent If $(Y^{1},Z^{1})$ and $(Y^{2},Z^{2})$ are both an $L^{1}$ solution of the mean-field BSDE \eqref{eq 1.2}, then from Theorem \ref{th 3.7} we get $Y_{t}^{1}=Y_{t}^{2},\ \mathbb{P}\text{-}a.s.$, for all $t\in[0,T$]. We define the following stopping times:
\begin{equation*}
\tau_{n}^{i}:=\inf\Big\{t\in[0,T]:\int_{0}^{t}|Z_{s}^{i}|^{2}ds\ge n\Big\}\wedge T,\ i=1,2;\ \ \tau_{n}:=\tau_{n}^{1}\wedge\tau_{n}^{2},\ n\geq 1.
\end{equation*}

\noindent Then, since both $(Y^{1},Z^{1})$ and $(Y^{1},Z^{2})$ are the $L^{1}$ solutions of \eqref{eq 1.2} we have
\begin{equation*}
\int_{0}^{t}\big(f(s,Y_{s}^{1},Z_{s}^{1},\mathbb{P}_{Y_{s}^{1}})-f(s,Y_{s}^{1},Z_{s}^{2},\mathbb{P}_{Y_{s}^{1}})\big)ds=\int_{0}^{t}(Z_{s}^{1}-Z_{s}^{2})dW_{s},\ 0\le t\le \tau_{n},
\end{equation*}
which means $Z_{s}^{1}=Z_{s}^{2},\ d\mathbb{P}\times ds\text{-}a.e.$ on $\Omega\times[0,\tau_{n}]$, and as $\tau_{n}\uparrow T$ ($n\rightarrow+\infty$), this implies $Z_{s}^{1}=Z_{s}^{2},\ d\mathbb{P}\times ds\text{-}a.e.$ on $\Omega\times[0,T]$.
\end{proof}

\section{Multi-dimensional case}\label{sec 4}
This section is devoted to the study of the existence and the uniqueness of $L^{1}$ solutions for multi-dimensional mean-field BSDEs \eqref{eq 1.2} with integrable data, where the generator satisfies a one-sided Osgood condition in the variable $y$. For this end, we assume that the terminal value $\xi:\Omega\rightarrow \mathbb{R}^{n}$ is $\mathcal{F}_{T}$-measurable and the generator $f:=f(\omega,t,y,z,\mu):\Omega\times[0,T]\times\mathbb{R}^{n}\times\mathbb{R}^{n\times d}\times\mathcal{P}_{1}(\mathbb{R}^{n})\rightarrow \mathbb{R}^{n}$ is $\mathbb{F}$-progressively measurable, for every $(y,z,\mu)\in \mathbb{R}^{n}\times\mathbb{R}^{n\times d}\times\mathcal{P}_{1}(\mathbb{R}^{n})$. Additionally, we make use of the following assumptions on the generator $f$:

\noindent\textbf{(H1)}\ There exists a nondecreasing and concave function $\eta:[0,+\infty)\rightarrow[0,+\infty)$ with $\eta(0)=0$, $\eta(u)>0,\ u>0$  and $\int_{0^{+}}\frac{du}{\eta(u)}=+\infty$, such that, for all $t\in[0,T],\ y,\ \bar{y}\in\mathbb{R}^{n}$, $z\in\mathbb{R}^{n\times d}$ and $\mu\in \mathcal{P}_{1}(\mathbb{R}^{n}),$
\begin{equation*}
\begin{aligned}
\Big\langle\frac{y-\bar{y}}{|y-\bar{y}|}\textbf{1}_{\{|y-\bar{y}|\neq0\}},f(t,y,z,\mu)-f(t,\bar{y},z,\mu)\Big\rangle\le \eta\big(|y-\bar{y}|\big),\ d\mathbb{P}\times dt\text{-}a.e.
\end{aligned}
\end{equation*}

\noindent\textbf{(H2)}\ There exists $K>0$, such that, for all $t\in[0,T],\ y\in\mathbb{R}^{n}$, $z\in\mathbb{R}^{n\times d}$ and $\mu,\ \bar{\mu}\in \mathcal{P}_{1}(\mathbb{R}^{n}),$
\begin{equation*}
\begin{aligned}
|f(t,y,z,\mu)-f(t,y,z,\bar{\mu})|\le KW_{1}(\mu,\bar{\mu}),\ d\mathbb{P}\times dt\text{-}a.e.
\end{aligned}
\end{equation*}

\noindent\textbf{(H3)}\ There exist constants $L>0,\ M>0,\ \alpha\in(0,1)$ and a nonnegative process $(\vartheta_{t})_{t\in[0,T]}\in \mathcal{L}_{\mathbb{F}}^{1}(0,T;\mathbb{R})$, such that, $d\mathbb{P}\times dt\text{-}a.e.$, for all $t\in[0,T],\ y\in\mathbb{R}^{n}$, $z,\ \bar{z}\in\mathbb{R}^{n\times d}$ and $\mu\in \mathcal{P}_{1}(\mathbb{R}^{n}),$
\begin{equation*}
\begin{aligned}
|f(t,y,z,\mu)-f(t,y,\bar{z},\mu)|\le L|z-\bar{z}|,\\
\end{aligned}
\end{equation*}
\begin{equation*}
\begin{aligned}
|f(t,y,z,\mu)-f(t,y,0,\mu)|\le M(\vartheta_{t}+|y|+W_{1}(\mu,\delta_{0})+|z|)^{\alpha}.
\end{aligned}
\end{equation*}

\noindent\textbf{(H4)}\ For all $r\ge0$, it holds that $\mathbb{E}\big[\int_{0}^{T}\bar{\phi}_{r}(t)dt\big]<+\infty\ \text{with}\ \bar{\phi}_{r}(t):=\sup_{|y|\le r}|f(t,y,0,\delta_{0})|.$ Moreover, $d\mathbb{P}\times dt\text{-}a.e.$, for all $z\in\mathbb{R}^{n\times d}$ and $\mu\in \mathcal{P}_{1}(\mathbb{R}^{n})$, $y\rightarrow f(\omega,t,y,z,\mu)$ is continuous.

\begin{remark}\label{re 4.1}
According to \cite[Proposition 1]{F15}, the one-sided Osgood condition in \textbf{(H1)} is stronger than the $p$-order weak monotonicity condition \textbf{(H1$p$)} ($p>1$): $d\mathbb{P}\times dt\text{-}a.e.$,
\begin{equation*}
\begin{aligned}
|y-\bar{y}|^{p-1}\Big\langle\frac{y-\bar{y}}{|y-\bar{y}|}\textbf{1}_{\{|y-\bar{y}|\neq0\}},f(t,y,z,\mu)-f(t,\bar{y},z,\mu)\Big\rangle\le \eta_{p}\big(|y-\bar{y}|^{p}\big),
\end{aligned}
\end{equation*}
\noindent for all $t\in[0,T],\ y,\ \bar{y}\in\mathbb{R}^{n}$, $z\in\mathbb{R}^{n\times d}$ and $\mu\in \mathcal{P}_{1}(\mathbb{R}^{n}),$ where $\eta_{p}:[0,+\infty)\rightarrow[0,+\infty)$ is a nondecreasing and concave function with $\eta_{p}(0)=0$, $\eta_{p}(u)>0,\ u>0$  and $\int_{0^{+}}\frac{du}{\eta_{p}(u)}=+\infty$.
\end{remark}

\begin{remark}\label{re 4.2}
The function $\eta(\cdot)$ in \textbf{(H1)} satisfies an at most linear growth condition owing to its nondecreasing concavity with $\eta(0)=0$. Here and thereafter,  $R$ is used to represent the corresponding linear growth constant, namely
\begin{equation*}\label{eq 4.1}\tag{4.1}
\eta(x)\le R(x+1),\ x\ge0.
\end{equation*}
Analogously, we denote by $R_{p}$ the linear growth constant associated with $\eta_{p}(\cdot)$, namely
\begin{equation*}
\eta_{p}(x)\le R_{p}(x+1),\ x\ge0.
\end{equation*}
\end{remark}

\begin{example}\label{ex 4.3}
For all $(\omega,t,y,z,\mu)\in \Omega\times[0,T]\times\mathbb{R}^{n}\times \mathbb{R}^{n\times d}\times\mathcal{P}_{1}(\mathbb{R}^{n})$, we can consider the generator:
\begin{equation*}
\begin{aligned}
F^{i}(\omega,t,y,z,\mu)&:=|W_{t}(\omega)|-y_{i}^{3}+\frac{2|z|}{1+|z|}+\int_{\mathbb{R}^{n}}3x_{i}\mu(dx),\ 1\le i\le n.
\end{aligned}
\end{equation*}
It is not difficult to verify that $F$ satisfies assumptions \textbf{(H1)}--\textbf{(H4)} with $\eta(x)=x,\ x\ge0,\ K=3,\ L=2\sqrt{n},\ M=2\sqrt{n},\ \vartheta_{t}\equiv1,\ \alpha=\frac{1}{3},\ \bar{\phi}_{r}(t)=\sup_{|y|\le r}\big[\sum_{i=1}^{n}(|W_{t}|-y_{i}^{3})^{2}\big]^{\frac{1}{2}}$.
\end{example}

\begin{example}\label{ex 4.4}
For all $(\omega,t,y,z,\mu)\in \Omega\times[0,T]\times\mathbb{R}^{n}\times \mathbb{R}^{n\times d}\times\mathcal{P}_{1}(\mathbb{R}^{n})$, we can consider the generator:
\begin{equation*}
\begin{aligned}
\bar{F}^{i}(\omega,t,y,z,\mu)&:=e^{|W_{t}(\omega)|}-e^{y_{i}}+\psi(|y|)+2\arctan(|z|)+(e^{-y_{i}}\wedge1)|\sin|z||+3W_{1}(\mu,\delta_{0}),\ 1\le i\le n,
\end{aligned}
\end{equation*}
where $\psi(x):=(x|\ln x|)\textbf{1}_{\{0< x\le\varepsilon\}}+[\psi^{\prime}(\varepsilon-)(x-\varepsilon)+\psi(\varepsilon)]\textbf{1}_{\{x>\varepsilon\}},\ x>0,$ with $\psi(0)=0$. It can be checked that $\bar{F}$ satisfies assumptions \textbf{(H1)}--\textbf{(H4)} with $\eta(x)=\sqrt{n}\psi(x),\ x\ge0,\ K=3\sqrt{n},\ L=3\sqrt{n},\ M=(\pi+1)\sqrt{n},\ \vartheta_{t}\equiv1,\ \alpha=\frac{1}{4},\ \bar{\phi}_{r}(t)=\sup_{|y|\le r}\big[\sum_{i=1}^{n}\{e^{|W_{t}|}-e^{y_{i}}+\psi(|y|)\}^{2}\big]^{\frac{1}{2}}$.
\end{example}

For reader's convenience we give some important estimates for solutions of multi-dimensional BSDEs in the appendix; one may refer to the Propositions 1--4 in \cite{F18} and Proposition 3.2 in \cite{BD03} for more details.

Now we will present the main results of this section concerning the existence and the uniqueness of $L^{1}$ solutions for multi-dimensional mean-field BSDEs \eqref{eq 1.2} with integrable data. To be more precise, Theorem \ref{th 4.9} addresses the uniqueness of the solution, while Theorem \ref{th 4.10} focuses on the existence of such equations.

\begin{theorem}\label{th 4.9}
Let assumptions \textbf{(H1)}--\textbf{(H3)} hold true. Then for every given $\xi\in L_{\mathcal{F}_{T}}^{1}(\Omega;\mathbb{R}^{n})$, the mean-field BSDE \eqref{eq 1.2} admits at most one solution $(Y,Z)$ such that $Y$ belongs to the class (D) and $Z$ belongs to $\bigcup_{\beta>\alpha}\mathcal{H}^{\beta}_{\mathbb{F}}(0,T;\mathbb{R}^{n\times d})$, which directly implies the uniqueness of the $L^{1}$ solution.
\end{theorem}

\begin{proof}
Let $(Y,Z)$ and $(\widetilde{Y},\widetilde{Z})$ be two adapted solutions to mean-field BSDE \eqref{eq 1.2} such that $Y$ and $\widetilde{Y}$ are both of the class (D), and both $Z$ and $\widetilde{Z}$ belong to $\mathcal{H}^{\beta}_{\mathbb{F}}(0,T;\mathbb{R}^{n\times d})$ for some $\beta\in(\alpha,1)$. For simplicity of notation, we put $\widehat{Y}:=Y-\widetilde{Y}$ and $\widehat{Z}:=Z-\widetilde{Z}.$

\noindent \textbf{Step 1.}\ Let us first show that $\widehat{Y}\in \mathcal{S}_{\mathbb{F}}^{p}(0,T;\mathbb{R}^{n})$ with $p:=\beta/\alpha>1$. For this purpose, we introduce the following sequence of stopping times: For all $n\ge1$,
\begin{equation*}
\begin{aligned}
\tau_{n}:=\inf\Big\{t\in[0,T]:\int_{0}^{t}\big(|Z_{s}|^{2}+|\widetilde{Z}_{s}|^{2}\big)ds\ge n\Big\}\wedge T.
\end{aligned}
\end{equation*}

\noindent From the It\^{o}-Tanaka formula (see also Briand et al. \cite[Corollary 2.3]{BD03}) we deduce that, for all $t\in[0,T]$,
\begin{equation*}\label{eq 4.2}\tag{4.2}
\begin{aligned}
|\widehat{Y}_{t\wedge \tau_{n}}|\le|\widehat{Y}_{\tau_{n}}|+\int_{t\wedge \tau_{n}}^{\tau_{n}}\big\langle\varrho (\widehat{Y}_{s}),f(s,Y_{s},Z_{s},\mathbb{P}_{Y_{s}})-f(s,\widetilde{Y}_{s},\widetilde{Z}_{s},\mathbb{P}_{\widetilde{Y}_{s}})\big\rangle ds
-\int_{t\wedge \tau_{n}}^{\tau_{n}}\langle\varrho (\widehat{Y}_{s}),\widehat{Z}_{s} dW_{s}\rangle,
\end{aligned}
\end{equation*}
where $\varrho(x):=|x|^{-1}x,\ \text{if}\ x\neq0;\ \varrho(x):=0,\ \text{if}\ x=0$. In view of assumptions \textbf{(H1)}--\textbf{(H3)}, we get that, $d\mathbb{P}\times ds\text{-}a.e.$,
\begin{equation*}\label{eq 4.3}\tag{4.3}
\begin{aligned}
&\hspace{1.3em}\big\langle\varrho (\widehat{Y}_{s}),f(s,Y_{s},Z_{s},\mathbb{P}_{Y_{s}})-f(s,\widetilde{Y}_{s},\widetilde{Z}_{s},\mathbb{P}_{\widetilde{Y}_{s}})\big\rangle\\
&\le\big\langle\varrho (\widehat{Y}_{s}),f(s,Y_{s},Z_{s},\mathbb{P}_{Y_{s}})-f(s,\widetilde{Y}_{s},Z_{s},\mathbb{P}_{Y_{s}})\big\rangle+|f(s,\widetilde{Y}_{s},Z_{s},\mathbb{P}_{Y_{s}})-f(s,\widetilde{Y}_{s},Z_{s},\mathbb{P}_{\widetilde{Y}_{s}})|\\
&\hspace{2em}+|f(s,\widetilde{Y}_{s},Z_{s},\mathbb{P}_{\widetilde{Y}_{s}})-f(s,\widetilde{Y}_{s},\widetilde{Z}_{s},\mathbb{P}_{\widetilde{Y}_{s}})|\\
&\le \eta(|\widehat{Y}_{s}|)+K\mathbb{E}[|\widehat{Y}_{s}|]+2M(\vartheta_{s}+|\widetilde{Y}_{s}|+\mathbb{E}[|\widetilde{Y}_{s}|]+|Z_{s}|+|\widetilde{Z}_{s}|)^{\alpha}.
\end{aligned}
\end{equation*}

\noindent Inserting \eqref{eq 4.3} in \eqref{eq 4.2} and then taking the conditional expectation gives, for all $n\ge1$ and $t\in[0,T]$,
\begin{equation*}\label{eq 4.4}\tag{4.4}
\begin{aligned}
|\widehat{Y}_{t\wedge \tau_{n}}|\le\mathbb{E}_{t}\Big[|\widehat{Y}_{\tau_{n}}|+\int_{t\wedge \tau_{n}}^{\tau_{n}}\big(\eta(|\widehat{Y}_{s}|)+K\mathbb{E}[|\widehat{Y}_{s}|]\big)ds\Big]+Q(t),
\end{aligned}
\end{equation*}
where $Q(t):=2M\mathbb{E}_{t}\big[\int_{0}^{T}\big(\vartheta_{s}+|\widetilde{Y}_{s}|+\mathbb{E}[|\widetilde{Y}_{s}|]+|Z_{s}|+|\widetilde{Z}_{s}|\big)^{\alpha}ds\big].$ As $\widehat{Y}$ belongs to the class (D) and $\tau_{n}\uparrow T$, as $n$ goes to infinity, we can take the limit $n\rightarrow +\infty$ in \eqref{eq 4.4}. Then, thanks to Lebesgue's dominated convergence theorem and \eqref{eq 4.1} in Remark \ref{re 4.2}, we obtain, for all $t\in[0,T]$,
\begin{equation*}
\begin{aligned}
|\widehat{Y}_{t}|\le \mathbb{E}_{t}\Big[\int_{t}^{T}\big(\eta(|\widehat{Y}_{s}|)+K\mathbb{E}[|\widehat{Y}_{s}|]\big)ds\Big]+Q(t)\le \int_{t}^{T}\big(R\mathbb{E}_{t}[|\widehat{Y}_{s}|]+K\mathbb{E}[|\widehat{Y}_{s}|]\big)ds+RT+Q(t).
\end{aligned}
\end{equation*}

\noindent Consequently, $\mathbb{E}_{t}[|\widehat{Y}_{r}|]\le \int_{r}^{T}\big(R\mathbb{E}_{t}[|\widehat{Y}_{s}|]+K\mathbb{E}[|\widehat{Y}_{s}|]\big)ds+RT+Q(t),\ r\in[t,T]$. Then Gronwall's inequality yields that
\begin{equation*}
\begin{aligned}
\mathbb{E}_{t}[|\widehat{Y}_{r}|]\le \Big(\int_{r}^{T}K\mathbb{E}[|\widehat{Y}_{s}|]ds+RT+Q(t)\Big)e^{RT},\ r\in[t,T].
\end{aligned}
\end{equation*}

\noindent In particular, by setting $r=t$ and noting that $\widehat{Y}$ is of the class (D) which guarantees that $\gamma:=\sup_{s\in[0,T]}\mathbb{E}[|\widehat{Y}_{s}|]<+\infty$, we conclude that
\begin{equation*}\label{eq 4.5}\tag{4.5}
\begin{aligned}
|\widehat{Y}_{t}|\le \big(\gamma KT+RT+Q(t)\big)e^{RT},\ t\in[0,T].
\end{aligned}
\end{equation*}

\noindent Recall that $\widehat{Y}$ is of the class (D), $Z$ and $\widetilde{Z}$ belong to $\mathcal{H}^{\beta}_{\mathbb{F}}(0,T;\mathbb{R}^{n\times d})$ with $\beta>\alpha$, and $\vartheta\in\mathcal{L}_{\mathbb{F}}^{1}(0,T;\mathbb{R})$. Thus, it follows from Doob's maximal inequality that
\begin{equation*}\label{eq 4.6}\tag{4.6}
\begin{aligned}
\mathbb{E}\big[\sup\limits_{t\in[0,T]}|Q(t)|^{\beta/\alpha}\big]\le \Big(\frac{\beta}{\beta-\alpha}\Big)^{\beta/\alpha}\mathbb{E}\big[|Q(T)|^{\beta/\alpha}\big]<+\infty.
\end{aligned}
\end{equation*}

\noindent Hence, combining this with \eqref{eq 4.5} yields that $\widehat{Y}\in \mathcal{S}_{\mathbb{F}}^{\beta/\alpha}(0,T;\mathbb{R}^{n})$.

\noindent \textbf{Step 2.}\ First, notice that $(\widehat{Y},\widehat{Z})$ solves the following BSDE:
\begin{equation*}\label{eq 4.7}\tag{4.7}
\begin{aligned}
\widehat{Y}_{t}=\int_{t}^{T}\Delta f(s,\widehat{Y}_{s},\widehat{Z}_{s})ds-\int_{t}^{T}\widehat{Z}_{s}dW_{s},\ t\in[0,T],
\end{aligned}
\end{equation*}
where $\Delta f(t,y,z):=f(t,y+\widetilde{Y}_{t},z+\widetilde{Z}_{t},\mathbb{P}_{Y_{t}})-f(t,\widetilde{Y}_{t},\widetilde{Z}_{t},\mathbb{P}_{\widetilde{Y}_{t}})$, for all $(t,y,z)\in [0,T]\times\mathbb{R}^{n}\times\mathbb{R}^{n\times d}$.
It follows from assumptions \textbf{(H1)}--\textbf{(H3)} on the generator $f$ along with Remarks \ref{re 4.1}--\ref{re 4.2} that there exists a constant $R_{2}>0$ such that, $d\mathbb{P}\times dt\text{-}a.e.$, for all $t\in[0,T],\ y\in \mathbb{R}^{n}\ \text{and}\ z\in\mathbb{R}^{n\times d}$,
\begin{equation*}\label{eq 4.8}\tag{4.8}
\begin{aligned}
\big\langle y,\Delta f(t,y,z)\big\rangle &\le \eta_{2}(|y|^{2})+L|y||z|+K|y|W_{1}(\mathbb{P}_{Y_{t}},\mathbb{P}_{\widetilde{Y}_{t}})\\
&\le R_{2}|y|^{2}+L|y||z|+K|y|W_{1}(\mathbb{P}_{Y_{t}},\mathbb{P}_{\widetilde{Y}_{t}})+R_{2}.
\end{aligned}
\end{equation*}
Thus, as $\widehat{Y}\in \mathcal{S}_{\mathbb{F}}^{\beta/\alpha}(0,T;\mathbb{R}^{n})$, Lemma \ref{le 4.5} with $\lambda_{1}=R_{2},\ \lambda_{2}=L,\ g_{t}=KW_{1}(\mathbb{P}_{Y_{t}},\mathbb{P}_{\widetilde{Y}_{t}}),\ \theta_{t}\equiv R_{2}$, implies $\widehat{Z}\in \mathcal{H}_{\mathbb{F}}^{p_{0}}(0,T;\mathbb{R}^{n\times d})$ with $p_{0}=\beta/\alpha>1$. In particular, $(\widehat{Y},\widehat{Z})\in \mathcal{S}^{p}_{\mathbb{F}}(0,T;\mathbb{R}^{n})\times \mathcal{H}^{p}_{\mathbb{F}}(0,T;\mathbb{R}^{n\times d})$, for all $p\in(0,1)$, is a solution to BSDE \eqref{eq 4.7}.

Moreover, from assumptions \textbf{(H1)}--\textbf{(H3)} and Remark \ref{re 4.1}, we also derive that there exists a nondecreasing and concave function $\eta_{p_{0}}(\cdot):[0,+\infty)\rightarrow[0,+\infty)$ with $\eta_{p_{0}}(0)=0$, $\eta_{p_{0}}(u)>0,\ u>0$  and $\int_{0^{+}}\frac{du}{\eta_{p_{0}}(u)}=+\infty$, such that, $d\mathbb{P}\times dt\text{-}a.e.$, for all $t\in[0,T],\ y\in\mathbb{R}^{n}$ and $z\in\mathbb{R}^{n\times d}$,
\begin{equation*}
\begin{aligned}
|y|^{p_{0}-1}\big\langle \varrho(y),\Delta f(t,y,z)\big\rangle 
\le \eta_{p_{0}}(|y|^{p_{0}})+L|y|^{p_{0}-1}|z|+K|y|^{p_{0}-1}W_{1}(\mathbb{P}_{Y_{t}},\mathbb{P}_{\widetilde{Y}_{t}}).
\end{aligned}
\end{equation*}

\noindent Applying Lemma \ref{le 4.6} with $\varphi(\cdot)=\eta_{p_{0}}(\cdot),\ \lambda=L,\ g_{t}=KW_{1}(\mathbb{P}_{Y_{t}},\mathbb{P}_{\widetilde{Y}_{t}})$, and using H\"{o}lder's inequality, there exists a constant $C>0$ depending only on $p_{0},L,K,T$ such that, for all $t\in[0,T]$,
\begin{equation*}
\begin{aligned}
\mathbb{E}\big[\sup\limits_{s\in[t,T]}|\widehat{Y}_{s}|^{p_{0}}\big]\le C\Big(\int_{t}^{T}\eta_{p_{0}}\big(\mathbb{E}[|\widehat{Y}_{s}|^{p_{0}}]\big)ds+\int_{t}^{T}\mathbb{E}[|\widehat{Y}_{s}|^{p_{0}}]ds\Big)
\le C\int_{t}^{T}\widetilde{\eta}_{p_{0}}\Big(\mathbb{E}\big[\sup\limits_{u\in[s,T]}|\widehat{Y}_{u}|^{p_{0}}\big]\Big)ds,
\end{aligned}
\end{equation*}
where $\widetilde{\eta}_{p_{0}}(x):=\eta_{p_{0}}(x)+x,\ x\ge0$. It's easy to verify that $\widetilde{\eta}_{p_{0}}(\cdot)$ is still a nondecreasing and concave function with $\widetilde{\eta}_{p_{0}}(0)=0,\ \widetilde{\eta}_{p_{0}}(u)>0,\ u>0$ and $\int_{0^{+}}\frac{du}{\widetilde{\eta}_{p_{0}}(u)}=+\infty$ (we refer to Appendix in \cite{F13}).

Thus it follows from Bihari's inequality (see Lemma \ref{le 2.4}) that
\begin{equation*}\label{eq 4.9}\tag{4.9}
\mathbb{E}\big[\sup\limits_{t\in[0,T]}|Y_{t}-\widetilde{Y}_{t}|^{p_{0}}\big]=\mathbb{E}\big[\sup\limits_{t\in[0,T]}|\widehat{Y}_{t}|^{p_{0}}\big]=0.
\end{equation*}

\noindent Moreover, thanks to \eqref{eq 4.8}, Lemma \ref{le 2.3} and Remark \ref{re 4.2}, the generator $\Delta f$ of BSDE \eqref{eq 4.7} satisfies that, $d\mathbb{P}\times dt\text{-}a.e.$, for all $t\in[0,T],\ y\in\mathbb{R}^{n}$, $z\in\mathbb{R}^{n\times d}$, and $m\ge1$,
\begin{equation*}
\begin{aligned}
\big\langle y,\Delta f(t,y,z)\big\rangle \le (m+2R_{2})|y|^{2}+\eta_{2}\Big(\frac{2R_{2}}{m+2R_{2}}\Big)+L|y||z|+K|y|W_{1}(\mathbb{P}_{Y_{t}},\mathbb{P}_{\widetilde{Y}_{t}}).
\end{aligned}
\end{equation*}

\noindent Then applying Lemma \ref{le 4.5} with $\lambda_{1}=m+2R_{2},\ \lambda_{2}=L,\ g_{t}=KW_{1}(\mathbb{P}_{Y_{t}},\mathbb{P}_{\widetilde{Y}_{t}}),\ \theta_{t}\equiv\eta_{2}(\frac{2R_{2}}{m+2R_{2}})$, this shows that there exist two positive constants $C>0$ depending only on $p_{0},m,R_{2},L,K,T$ and $\widetilde{C}_{p_{0}}>0$ depending only on $p_{0}$, such that, for all $m\ge1$,

\begin{equation*}\label{eq 4.10}\tag{4.10}
\mathbb{E}\Big[\Big(\int_{0}^{T}|\widehat{Z}_{s}|^{2}ds\Big)^{\frac{p_{0}}{2}}\Big]\le C\mathbb{E}\big[\sup\limits_{t\in[0,T]}|\widehat{Y}_{t}|^{p_{0}}\big]+\widetilde{C}_{p_{0}}\Big(\eta_{2}\Big(\frac{2R_{2}}{m+2R_{2}}\Big)\cdot T\Big)^{\frac{p_{0}}{2}}.
\end{equation*}

\noindent Therefore, by (4.9) and the fact that $\eta_{2}(\cdot)$ is a continuous function with $\eta_{2}(0)=0$, letting $m\rightarrow +\infty$ in the last inequality we obtain that
\begin{equation*}\label{eq 4.11}\tag{4.11}
\mathbb{E}\Big[\Big(\int_{0}^{T}|Z_{s}-\widetilde{Z}_{s}|^{2}ds\Big)^{\frac{p_{0}}{2}}\Big]=\mathbb{E}\Big[\Big(\int_{0}^{T}|\widehat{Z}_{s}|^{2}ds\Big)^{\frac{p_{0}}{2}}\Big]=0.
\end{equation*}
With \eqref{eq 4.9} and \eqref{eq 4.11}, the proof of the uniqueness is complete.
\end{proof}

We come now to our existence result.

\begin{theorem}\label{th 4.10}
Let assumptions \textbf{(H1)}--\textbf{(H4)} hold true. In the case where $\alpha$ takes values in $[\frac{1}{2},1)$ in \textbf{(H4)}, we further suppose that $\eta(\cdot)$ in \textbf{(H1)} satisfies that there exists $p^{\ast}>1$ such that
\begin{equation*}\label{eq 4.12}\tag{4.12}
\int_{0+}\frac{u^{p^{\ast}-1}}{\eta^{p^{\ast}}(u)}du=+\infty.
\end{equation*}
\end{theorem}

\noindent Then, for every given $\xi\in L_{\mathcal{F}_{T}}^{1}(\Omega;\mathbb{R}^{n})$, the mean-field BSDE \eqref{eq 1.2} possesses an $L^{1}$ solution $(Y,Z)$, i.e., $Y$ is of the class (D) and $(Y,Z)\in \mathcal{S}_{\mathbb{F}}^{p}(0,T;\mathbb{R}^{n})\times \mathcal{H}_{\mathbb{F}}^{p}(0,T;\mathbb{R}^{n\times d})$, for all $p\in(0,1)$.

\begin{proof}
\textbf{Step 1.}\ Given a pair of processes $(U,V)\in \mathcal{S}_{\mathbb{F}}^{p}(0,T;\mathbb{R}^{n})\times\mathcal{H}_{\mathbb{F}}^{p}(0,T;\mathbb{R}^{n\times d})$, for all $p\in(0,1)$, where $U$ is of the class (D), we consider the BSDE
\begin{equation*}\label{eq 4.13}\tag{4.13}
Y_{t}=\xi+\int_{t}^{T}f(s,Y_{s},V_{s},\mathbb{P}_{U_{s}})ds-\int_{t}^{T}Z_{s}dW_s,\ t\in[0,T].
\end{equation*}

\noindent Obviously, by assumption \textbf{(H1)}, we have $d\mathbb{P}\times dt\text{-}a.e.$, for all $t\in[0,T],\ \text{and}\ y,\ \bar{y}\in\mathbb{R}^{n}$,
\begin{equation*}
\begin{aligned}
\big\langle\varrho(y-\bar{y}),f(t,y,V_{t},\mathbb{P}_{U_{t}})-f(t,\bar{y},V_{t},\mathbb{P}_{U_{t}})\big\rangle\le \eta\big(|y-\bar{y}|\big).
\end{aligned}
\end{equation*}

\noindent In addition, by using assumptions \textbf{(H2)}--\textbf{(H4)} and H\"{o}lder's inequality, we get, for all $r\ge0$,
\begin{equation*}
\begin{aligned}
&\hspace{1.35em}\mathbb{E}\Big[\int_{0}^{T}\sup\limits_{|y|\le r}|f(t,y,V_{t},\mathbb{P}_{U_{t}})|dt\Big]\\
&\le\mathbb{E}\Big[\int_{0}^{T}\sup\limits_{|y|\le r}|f(t,y,0,\delta_{0})|dt\Big]+KT\sup\limits_{t\in[0,T]}\mathbb{E}[|U_{t}|]+M\Big(\mathbb{E}\Big[\int_{0}^{T}\vartheta_{t}dt\Big]\Big)^{\alpha}T^{1-\alpha}\\
&\hspace{2em}+MTr^{\alpha}+M\mathbb{E}\Big[\Big(\int_{0}^{T}|V_{s}|^{2}ds\Big)^{\frac{\alpha}{2}}\Big]T^{1-\frac{\alpha}{2}}<+\infty.
\end{aligned}
\end{equation*}

\noindent Hence, it follows from \cite[Theorems 1--2]{F18} that BSDE \eqref{eq 4.13} admits a unique $L^{1}$ solution $(Y,Z)$, that is, $Y$ is of the class (D) and $(Y,Z)\in \mathcal{S}_{\mathbb{F}}^{p}(0,T;\mathbb{R}^{n})\times \mathcal{H}_{\mathbb{F}}^{p}(0,T;\mathbb{R}^{n\times d})$, for all $p\in(0,1)$.

\noindent\textbf{Step 2.}\ Step 1 enables us to formulate the following Picard iteration procedure: We put as usual $(Y^{0},Z^{0})=(0,0)$ and define recursively the sequence of processes $(Y^{n},Z^{n}),\ n\ge1$, as the unique $L^{1}$ solution to the following BSDE:
\begin{equation*}\label{eq 4.14}\tag{4.14}
Y_{t}^{n}=\xi+\int_{t}^{T}f(s,Y_{s}^{n},Z_{s}^{n-1},\mathbb{P}_{Y_{s}^{n-1}})ds-\int_{t}^{T}Z_{s}^{n}dW_s,\ t\in[0,T].
\end{equation*}


\noindent In what follows, for all $n,k\ge1$, we set $\widehat{Y}^{n,k}:=Y^{n+k}-Y^{n}$ and $\widehat{Z}^{n,k}:=Z^{n+k}-Z^{n}$. By a similar argument as the proof of \eqref{eq 4.5} in Theorem \ref{th 4.9}, using assumptions \textbf{(H1)}--\textbf{(H3)}, we prove that, for $\gamma^{n,k}:=\sup_{s\in[0,T]}\mathbb{E}[|\widehat{Y}_{s}^{n,k}|]$ and for all $n,k\ge1$,
\begin{equation*}
\begin{aligned}
|\widehat{Y}_{t}^{n,k}|\le \big(\gamma^{n-1,k} KT+RT+Q^{n,k}(t)\big)e^{RT},\ t\in[0,T],
\end{aligned}
\end{equation*}
where $Q^{n,k}(t):=2M\mathbb{E}_{t}\big[\int_{0}^{T}\big(\vartheta_{s}+|Y_{s}^{n}|+\mathbb{E}[|Y_{s}^{n-1}|]+|Z_{s}^{n-1}|+|Z^{n+k-1}_{s}|\big)^{\alpha}ds\big]\in \mathcal{S}_{\mathbb{F}}^{r}(0,T;\mathbb{R})$, given that $\alpha r<1$ with $r>1$. Then we obtain that, for all $n,k\ge1$, $(\widehat{Y}_{t}^{n,k})_{t\in[0,T]}\in \mathcal{S}_{\mathbb{F}}^{r}(0,T;\mathbb{R}^{n})$, under the condition $\alpha r<1$ for $r>1$. In the following, we discuss two cases separately: (i)\ $\alpha\in(0,1/2)$;\ (ii)\ $\alpha\in[1/2,1)$.

\noindent\textbf{Step 3.} We first study the case $\alpha\in(0,1/2)$. In this case, by choosing $r=2$, we have
\begin{equation*}\label{eq 4.15}\tag{4.15}
\begin{aligned}
(\widehat{Y}_{t}^{n,k})_{t\in[0,T]}\in \mathcal{S}_{\mathbb{F}}^{2}(0,T;\mathbb{R}^{n}),\ n,k\ge1.
\end{aligned}
\end{equation*}

\noindent Observe that $(\widehat{Y}^{n,k},\widehat{Z}^{n,k})$ is a solution to the BSDE
\begin{equation*}\label{eq 4.16}\tag{4.16}
\begin{aligned}
\widehat{Y}_{t}^{n,k}=\int_{t}^{T}\widehat{f}^{n,k}(s,\widehat{Y}_{s}^{n,k})ds-\int_{t}^{T}\widehat{Z}_{s}^{n,k}dW_{s},\ t\in[0,T],
\end{aligned}
\end{equation*}
where $\widehat{f}^{n,k}(s,y):=f(s,y+Y_{s}^{n},Z_{s}^{n+k-1},\mathbb{P}_{Y_{s}^{n+k-1}})-f(s,Y_{s}^{n},Z_{s}^{n-1},\mathbb{P}_{Y_{s}^{n-1}})$, for all $(s,y)\in[0,T]\times\mathbb{R}^{n}$. From assumptions \textbf{(H1)}--\textbf{(H3)}, and the Remarks \ref{re 4.1}--\ref{re 4.2}, we get that there exists a nondecreasing and concave function $\eta_{2}(\cdot):[0,+\infty)\rightarrow[0,+\infty)$ with $\eta_{2}(0)=0$, $\eta_{2}(u)>0,\ u>0$  and $\int_{0^{+}}\frac{du}{\eta_{2}(u)}=+\infty$, such that, $d\mathbb{P}\times ds\text{-}a.e.$, for all $s\in[0,T]\ \text{and}\ y\in\mathbb{R}^{n}$,
\begin{equation*}\label{eq 4.17}\tag{4.17}
\begin{aligned}
&\langle y,\widehat{f}^{n,k}(s,y)\rangle\le\eta_{2}(|y|^{2})+2M|y|\big(\vartheta_{s}+|Y_{s}^{n}|+\mathbb{E}[|Y_{s}^{n-1}|]+|Z_{s}^{n-1}|+|Z_{s}^{n+k-1}|\big)^{\alpha}+K|y|\mathbb{E}[|\widehat{Y}_{s}^{n-1,k}|]\\
&\le R_{2}|y|^{2}+2M|y|\big(\vartheta_{s}+|Y_{s}^{n}|+\mathbb{E}[|Y_{s}^{n-1}|]+|Z_{s}^{n-1}|+|Z_{s}^{n+k-1}|\big)^{\alpha}+K|y|\mathbb{E}[|\widehat{Y}_{s}^{n-1,k}|]+R_{2},
\end{aligned}
\end{equation*}
and
\begin{equation*}\label{eq 4.18}\tag{4.18}
\begin{aligned}
\langle y,\widehat{f}^{n,k}(s,y)\rangle\le \eta_{2}(|y|^{2})+L|y||\widehat{Z}_{s}^{n-1,k}|+K|y|\mathbb{E}[|\widehat{Y}_{s}^{n-1,k}|].
\end{aligned}
\end{equation*}
By virtue of \eqref{eq 4.15} and \eqref{eq 4.17}, it follows from Lemma \ref{le 4.5} with $p=2,\ \lambda_{1}=R_{2},\ \lambda_{2}=0,\ g_{t}=2M\big(\vartheta_{t}+|Y_{t}^{n}|+\mathbb{E}[|Y_{t}^{n-1}|]+|Z_{t}^{n-1}|+|Z_{t}^{n+k-1}|\big)^{\alpha}+K\mathbb{E}[|\widehat{Y}_{t}^{n-1,k}|],\ \theta_{t}\equiv R_{2}$, that $\widehat{Z}^{n,k}\in \mathcal{H}_{\mathbb{F}}^{2}(0,T;\mathbb{R}^{n\times d})$. Hence, $(\widehat{Y}^{n,k},\widehat{Z}^{n,k})\in \mathcal{S}_{\mathbb{F}}^{2}(0,T;\mathbb{R}^{n})\times \mathcal{H}_{\mathbb{F}}^{2}(0,T;\mathbb{R}^{n\times d})$ is a solution of BSDE \eqref{eq 4.16}, for all $n,k\ge1$.

Moreover, taking into account \eqref{eq 4.18} and recalling Lemma \ref{le 4.7} with $\varphi(\cdot)=\eta_{2}(\cdot),\ \lambda=0,\ g_{t}=L|\widehat{Z}_{t}^{n-1,k}|+K\mathbb{E}[|\widehat{Y}_{t}^{n-1,k}|]$, we obtain the existence of a positive constant $C>0$, such that, for all $t\in[0,T]$,
\begin{equation*}\label{eq 4.19}\tag{4.19}
\begin{aligned}
&\mathbb{E}\big[\sup\limits_{s\in[t,T]}|\widehat{Y}_{s}^{n,k}|^{2}\big]+\mathbb{E}\Big[\int_{t}^{T}|\widehat{Z}_{s}^{n,k}|^{2}ds\Big]\le e^{C(T-t)}\bigg(\int_{t}^{T}\eta_{2}\Big(\mathbb{E}\big[\sup\limits_{u\in[s,T]}|\widehat{Y}_{u}^{n,k}|^{2}\big]\Big)ds\\
&\hspace{6em}+2L^{2}(T-t)\mathbb{E}\Big[\int_{t}^{T}|\widehat{Z}_{s}^{n-1,k}|^{2}ds\Big]+2K^{2}(T-t)^{2}\mathbb{E}\big[\sup\limits_{s\in[t,T]}|\widehat{Y}_{s}^{n-1,k}|^{2}\big]\bigg).
\end{aligned}
\end{equation*}

\noindent Next, we put $\varepsilon:=\min\big\{\frac{\ln 2}{C},\frac{1}{32L^{2}},\frac{1}{4\sqrt{2} K},\frac{\ln 2}{2R_{2}}\big\}\ \text{and}\ T_{j}:=(T-j\varepsilon)\vee 0,\ j\ge1.$ Thus, for all $t\in[T_{1},T]$, we have
\begin{equation*}\label{eq 4.20}\tag{4.20}
\begin{aligned}
e^{C(T-t)}\le2,\ \ 2e^{C(T-t)}L^{2}(T-t)\le\frac{1}{8},\ \ 2e^{C(T-t)}K^{2}(T-t)^{2}\le\frac{1}{8},\ \ e^{2R_{2}(T-t)}\le2.
\end{aligned}
\end{equation*}

\noindent Combining \eqref{eq 4.19}--\eqref{eq 4.20} gives, for all $t\in[T_{1},T]$,
\begin{equation*}\label{eq 4.21}\tag{4.21}
\begin{aligned}
&\mathbb{E}\big[\sup\limits_{s\in[t,T]}|\widehat{Y}_{s}^{n,k}|^{2}\big]+\mathbb{E}\Big[\int_{t}^{T}|\widehat{Z}_{s}^{n,k}|^{2}ds\Big]\le 2\int_{t}^{T}\eta_{2}\Big(\mathbb{E}\big[\sup\limits_{u\in[s,T]}|\widehat{Y}_{u}^{n,k}|^{2}\big]\Big)ds\\
&\hspace{7em}+\frac{1}{8}\mathbb{E}\Big[\int_{t}^{T}|\widehat{Z}_{s}^{n-1,k}|^{2}ds\Big]+\frac{1}{8}\mathbb{E}\big[\sup\limits_{s\in[t,T]}|\widehat{Y}_{s}^{n-1,k}|^{2}\big].
\end{aligned}
\end{equation*}

\noindent Thanks to Remark \ref{re 4.2} that $\eta_{2}(x)\le R_{2}x+R_{2}$, $x\ge0$, it follows from \eqref{eq 4.20}--\eqref{eq 4.21} and Gronwall's inequality that, for all $n\ge2,\ k\ge1$ and $t\in[T_{1},T]$,
\begin{equation*}\label{eq 4.22}\tag{4.22}
\begin{aligned}
&\hspace{2em}\mathbb{E}\big[\sup\limits_{s\in[t,T]}|\widehat{Y}_{s}^{n,k}|^{2}\big]+\mathbb{E}\Big[\int_{t}^{T}|\widehat{Z}_{s}^{n,k}|^{2}ds\Big]\\
&\le e^{2R_{2}(T-t)}\Big(2R_{2}T+\frac{1}{8}\mathbb{E}\Big[\int_{t}^{T}|\widehat{Z}_{s}^{n-1,k}|^{2}ds\Big]+\frac{1}{8}\mathbb{E}\big[\sup\limits_{s\in[t,T]}|\widehat{Y}_{s}^{n-1,k}|^{2}\big]\Big)\\
&\le 4R_{2}T+\frac{1}{4}\mathbb{E}\Big[\int_{t}^{T}|\widehat{Z}_{s}^{n-1,k}|^{2}ds\Big]+\frac{1}{4}\mathbb{E}\big[\sup\limits_{s\in[t,T]}|\widehat{Y}_{s}^{n-1,k}|^{2}\big].\\
\end{aligned}
\end{equation*}

\noindent In particular, by taking $n=2$ and $k=m-2$ in \eqref{eq 4.22}, we obtain that, for all $m\ge3$ and $t\in[T_{1},T]$,
\begin{equation*}
\begin{aligned}
&\hspace{2em}\mathbb{E}\big[\sup\limits_{s\in[t,T]}|Y_{s}^{m}-Y_{s}^{2}|^{2}\big]+\mathbb{E}\Big[\int_{t}^{T}|Z_{s}^{m}-Z_{s}^{2}|^{2}ds\Big]\\
&\le 4R_{2}T+\frac{1}{2}\mathbb{E}\Big[\int_{t}^{T}|Z_{s}^{2}-Z_{s}^{1}|^{2}ds\Big]+\frac{1}{2}\mathbb{E}\big[\sup\limits_{s\in[t,T]}|Y_{s}^{2}-Y_{s}^{1}|^{2}\big]\\
&\hspace{3em}+\frac{1}{2}\mathbb{E}\Big[\int_{t}^{T}|Z_{s}^{m-1}-Z_{s}^{2}|^{2}ds\Big]+\frac{1}{2}\mathbb{E}\big[\sup\limits_{s\in[t,T]}|Y_{s}^{m-1}-Y_{s}^{2}|^{2}\big],
\end{aligned}
\end{equation*}
which by induction leads to that, for all $t\in[T_{1},T]$,
\begin{equation*}\label{eq 4.23}\tag{4.23}
\begin{aligned}
&\hspace{2em}\sup\limits_{m\ge1}\Big(\mathbb{E}\big[\sup\limits_{s\in[t,T]}|Y_{s}^{m}-Y_{s}^{2}|^{2}\big]+\mathbb{E}\Big[\int_{t}^{T}|Z_{s}^{m}-Z_{s}^{2}|^{2}ds\Big]\Big)\\
&\le 8R_{2}T+\mathbb{E}\Big[\int_{0}^{T}|Z_{s}^{2}-Z_{s}^{1}|^{2}ds\Big]+\mathbb{E}\big[\sup\limits_{s\in[0,T]}|Y_{s}^{2}-Y_{s}^{1}|^{2}\big]<+\infty.
\end{aligned}
\end{equation*}

\noindent Moreover, note that for all $n\ge2,\ k\ge1$ and $t\in[T_{1},T]$,
\begin{equation*}\label{eq 4.24}\tag{4.24}
\begin{aligned}
&\hspace{2em}\frac{1}{4}\mathbb{E}\Big[\int_{t}^{T}|\widehat{Z}_{s}^{n-1,k}|^{2}ds\Big]+\frac{1}{4}\mathbb{E}\big[\sup\limits_{s\in[t,T]}|\widehat{Y}_{s}^{n-1,k}|^{2}\big]\\
&\le \frac{1}{2}\mathbb{E}\Big[\int_{t}^{T}\big(|Z_{s}^{n+k-1}-Z_{s}^{2}|^{2}+|Z_{s}^{n-1}-Z_{s}^{2}|^{2}\big)ds\Big]\\
&\hspace{2.6em}+\frac{1}{2}\mathbb{E}\big[\sup\limits_{s\in[t,T]}|Y_{s}^{n+k-1}-Y_{s}^{2}|^{2}\big]+\frac{1}{2}\mathbb{E}\big[\sup\limits_{s\in[t,T]}|Y_{s}^{n-1}-Y_{s}^{2}|^{2}\big]\\
&\le \sup\limits_{m\ge1}\Big(\mathbb{E}\big[\sup\limits_{s\in[t,T]}|Y_{s}^{m}-Y_{s}^{2}|^{2}\big]+\mathbb{E}\Big[\int_{t}^{T}|Z_{s}^{m}-Z_{s}^{2}|^{2}ds\Big]\Big).
\end{aligned}
\end{equation*}

\noindent Therefore, by combining \eqref{eq 4.22}--\eqref{eq 4.24}, we get, for all $t\in[T_{1},T]$,
\begin{equation*}\label{eq 4.25}\tag{4.25}
\begin{aligned}
&\hspace{2em}\sup\limits_{n\ge2}\sup\limits_{k\ge1}\Big(\mathbb{E}\big[\sup\limits_{s\in[t,T]}|\widehat{Y}_{s}^{n,k}|^{2}\big]+\mathbb{E}\Big[\int_{t}^{T}|\widehat{Z}_{s}^{n,k}|^{2}ds\Big]\Big)\\
&\le 12R_{2}T+\mathbb{E}\Big[\int_{0}^{T}|Z_{s}^{2}-Z_{s}^{1}|^{2}ds\Big]+\mathbb{E}\big[\sup\limits_{s\in[0,T]}|Y_{s}^{2}-Y_{s}^{1}|^{2}\big]<+\infty.
\end{aligned}
\end{equation*}

\noindent Let us put $a_{n,k}(t):=\mathbb{E}\big[\sup_{s\in[t,T]}|\widehat{Y}_{s}^{n,k}|^{2}\big]+\mathbb{E}\big[\int_{t}^{T}|\widehat{Z}_{s}^{n,k}|^{2}ds\big],\ t\in[T_{1},T].$ Then, from \eqref{eq 4.21} and the properties of $\eta_{2}(\cdot)$,
\begin{equation*}
\begin{aligned}
a_{n,k}(t)\le 2\int_{t}^{T}\eta_{2}\big(a_{n,k}(s)\big)ds+\frac{1}{8}a_{n-1,k}(t),\ t\in[T_{1},T],\ n\ge2,\ k\ge1,
\end{aligned}
\end{equation*}

\noindent and so also,
\begin{equation*}
\begin{aligned}
\sup\limits_{k\ge1}a_{n,k}(t)\le 2\int_{t}^{T}\eta_{2}\big(\sup\limits_{k\ge1}a_{n,k}(s)\big)ds+\frac{1}{8}\sup\limits_{k\ge1}a_{n-1,k}(t),\ t\in[T_{1},T],\ n\ge2.
\end{aligned}
\end{equation*}

\noindent Consequently, from Fatou's lemma, for all $t\in[T_{1},T]$,
\begin{equation*}
\begin{aligned}
\limsup\limits_{n\rightarrow \infty}\sup\limits_{k\ge1}a_{n,k}(t)\le 2\int_{t}^{T}\eta_{2}\big(\limsup\limits_{n\rightarrow \infty}\sup\limits_{k\ge1}a_{n,k}(s)\big)ds+\frac{1}{8}\limsup\limits_{n\rightarrow \infty}\sup\limits_{k\ge1}a_{n,k}(t).
\end{aligned}
\end{equation*}

\noindent Thus, thanks to \eqref{eq 4.25} we have
\begin{equation*}
\begin{aligned}
\limsup\limits_{n\rightarrow \infty}\sup\limits_{k\ge1}a_{n,k}(t)\le \frac{16}{7}\int_{t}^{T}\eta_{2}\big(\limsup\limits_{n\rightarrow \infty}\sup\limits_{k\ge1}a_{n,k}(s)\big)ds,\ t\in[T_{1},T],
\end{aligned}
\end{equation*}
which together with Bihari's inequality allows to conclude that
\begin{equation*}
\begin{aligned}
\lim\limits_{n\rightarrow \infty}\sup\limits_{k\ge1}a_{n,k}(T_{1})=\lim\limits_{n\rightarrow \infty}\sup\limits_{k\ge1}\Big(\mathbb{E}\big[\sup_{s\in[T_{1},T]}|\widehat{Y}_{s}^{n,k}|^{2}\big]+\mathbb{E}\Big[\int_{T_{1}}^{T}|\widehat{Z}_{s}^{n,k}|^{2}ds\Big]\Big)=0.
\end{aligned}
\end{equation*}

\noindent Hence, there exists a pair of processes $(Y^{\ast},Z^{\ast})\in\mathcal{S}_{\mathbb{F}}^{2}(T_{1},T;\mathbb{R}^{n})\times \mathcal{H}_{\mathbb{F}}^{2}(T_{1},T;\mathbb{R}^{n\times d})$ satisfying
\begin{equation*}\label{eq 4.26}\tag{4.26}
\begin{aligned}
\lim\limits_{n\rightarrow+\infty}\mathbb{E}\Big[\sup\limits_{s\in[T_{1},T]}|(Y_{s}^{n}-Y_{s}^{1})-Y^{\ast}_{s}|^{2}+\int_{T_{1}}^{T}|(Z_{s}^{n}-Z_{s}^{1})-Z^{\ast}_{s}|^{2}ds\Big]=0.
\end{aligned}
\end{equation*}

\noindent Recall that $(Y^{1},Z^{1})\in\mathcal{S}_{\mathbb{F}}^{p}(0,T;\mathbb{R}^{n})\times \mathcal{H}_{\mathbb{F}}^{p}(0,T;\mathbb{R}^{n\times d})$, for all $p\in(0,1)$, and $Y^{1}$ is of class (D). This allows to take the limit in BSDE \eqref{eq 4.14}, and to conclude that $(Y_{t},Z_{t})_{t\in[T_{1},T]}:=(Y^{\ast}_{t}+Y_{t}^{1},Z^{\ast}_{t}+Z_{t}^{1})_{t\in[T_{1},T]}$ constitutes an $L^{1}$ solution of mean-field BSDE \eqref{eq 1.2} on the time interval $[T_{1},T]$.

Finally, observing that $\varepsilon$ depends solely on $L,K,R_{2}$, we can identify the minimal integer $N$ satisfying $T_{N}=0$. Thus, by iteration of the preceding argument, we can construct an $L^{1}$ solution to mean-field BSDE \eqref{eq 1.2} on the interval $[T_{i+1},T_{i}]$ with terminal condition $Y_{T_{i}}$, $1\le i\le N-1$, which effectively extends the solution $(Y,Z)$ to the whole interval $[0,T]$ within a finite number of steps.

\noindent\textbf{Step 4.} We now consider the case $\alpha\in[\frac{1}{2},1)$. Let us fix $r\in(1,p^{\ast}\wedge\frac{1}{\alpha})$. Then, for all $n,k\ge1$,
\begin{equation*}\label{eq 4.27}\tag{4.27}
\begin{aligned}
(\widehat{Y}_{t}^{n,k})_{t\in[0,T]}\in \mathcal{S}_{\mathbb{F}}^{r}(0,T;\mathbb{R}^{n}).
\end{aligned}
\end{equation*}

\noindent Notice that we have assumed the identity \eqref{eq 4.12} to be valid in this situation. Therefore, following \cite[Proposition 1]{F15}, the generator $f$ of our mean-field BSDE \eqref{eq 1.2} satisfies a $p^{\ast}$-order one-sided Mao's condition in $y$, namely, there exists a nondecreasing and concave function $\varsigma:[0,+\infty)\rightarrow[0,+\infty)$ with $\varsigma(0)=0$, $\varsigma(u)>0,\ u>0$ and $\int_{0+}\frac{du}{\varsigma(u)}=+\infty$ (the linear growth constant is denoted by $R_{\varsigma}$), such that, for all $t\in[0,T],\ y,\ \bar{y}\in\mathbb{R}^{n}$, $z\in\mathbb{R}^{n\times d}$ and $\mu\in\mathcal{P}_{1}(\mathbb{R}^{n})$,
\begin{equation*}\label{eq 4.28}\tag{4.28}
\begin{aligned}
\big\langle\varrho(y-\bar{y}),f(t,y,z,\mu)-f(t,\bar{y},z,\mu)\big\rangle\le \varsigma^{\frac{1}{p^{\ast}}}(|y-\bar{y}|^{p^{\ast}}),\ d\mathbb{P}\times dt\text{-}a.e.
\end{aligned}
\end{equation*}

\noindent Additionally, owing to the fact that $r<p^{\ast}$, it follows again from \cite[Proposition 1]{F15} that inequality \eqref{eq 4.28} with $p^{\ast}$ substituted by $r$ also holds for the generator $f$. On the other hand, note that \eqref{eq 4.16}--\eqref{eq 4.17} hold true simultaneously. Therefore, by adopting the same argument as in case (i) and considering \eqref{eq 4.27}, we use Lemma \ref{le 4.5} with $p=r$ to obtain that $\widehat{Z}^{n,k}\in\mathcal{H}_{\mathbb{F}}^{r}(0,T;\mathbb{R}^{n\times d})$. Hence, $(\widehat{Y}^{n,k},\widehat{Z}^{n,k})\in \mathcal{S}_{\mathbb{F}}^{r}(0,T;\mathbb{R}^{n})\times \mathcal{H}_{\mathbb{F}}^{r}(0,T;\mathbb{R}^{n\times d})$ is a solution of BSDE \eqref{eq 4.16}, for all $n,k\ge1$.

Moreover, by combining the preceding arguments with \textbf{(H2)}--\textbf{(H3)}, it is straightforward to verify that, $d\mathbb{P}\times dt\text{-}a.e.$, for all $t\in[0,T]\ \text{and}\ y\in\mathbb{R}^{n}$,
\begin{equation*}
\begin{aligned}
\big\langle\varrho(y),\widehat{f}^{n,k}(t,y)\big\rangle\le \varsigma^{\frac{1}{r}}(|y|^{r})+L|\widehat{Z}_{t}^{n-1,k}|+K\mathbb{E}[|\widehat{Y}_{t}^{n-1,k}|].
\end{aligned}
\end{equation*}

\noindent Then, it follows from Lemma \ref{le 4.8} with $\varpi(\cdot)=\varsigma(\cdot),\ \lambda=0,\ g_{t}=L|\widehat{Z}_{t}^{n-1,k}|+K\mathbb{E}[|\widehat{Y}_{t}^{n-1,k}|]$, and H\"{o}lder's inequality that there exists a constant $C_{r}>0$ depending only on $r$ such that, for all $n\ge2,\ k\ge1$ and $t\in[0,T]$,
\begin{equation*}\label{eq 4.29}\tag{4.29}
\begin{aligned}
&\mathbb{E}\big[\sup\limits_{s\in[t,T]}|\widehat{Y}_{s}^{n,k}|^{r}\big]+\mathbb{E}\Big[\Big(\int_{t}^{T}|\widehat{Z}_{s}^{n,k}|^{2}ds\Big)^{\frac{r}{2}}\Big]\le e^{C_{r}(T-t)}\bigg(\int_{t}^{T}\varsigma \Big(\mathbb{E}\big[\sup\limits_{u\in[s,T]}|\widehat{Y}_{u}^{n,k}|^{r}\big]\Big)ds\\
&\hspace{6em}+2^{r}L^{r}(T-t)^{\frac{r}{2}}\mathbb{E}\Big[\Big(\int_{t}^{T}|\widehat{Z}_{s}^{n-1,k}|^{2}ds\Big)^{\frac{r}{2}}\Big]+2^{r}K^{r}(T-t)^{r}\mathbb{E}\big[\sup\limits_{s\in[t,T]}|\widehat{Y}_{s}^{n-1,k}|^{r}\big]\bigg).
\end{aligned}
\end{equation*}

\noindent Then, using the notations
\begin{equation*}
\begin{aligned}
\delta:=\min\Big\{\frac{\ln 2}{C_{r}},\Big(\frac{1}{16\cdot(2L)^{r}}\Big)^{\frac{2}{r}},\Big(\frac{1}{16\cdot(2K)^{r}}\Big)^{\frac{1}{r}},\frac{\ln 2}{2R_{\varsigma}}\Big\}\ \ \text{and}\ \ \widetilde{T}_{j}:=(T-j\delta)\vee 0,\ j\ge1,
\end{aligned}
\end{equation*}

\noindent we have, for all $t\in[\widetilde{T}_{1},T]$,
\begin{equation*}\label{eq 4.30}\tag{4.30}
\begin{aligned}
e^{C_{r}(T-t)}\le2,\ \ 2^{r}e^{C_{r}(T-t)}L^{r}(T-t)^{\frac{r}{2}}\le\frac{1}{8},\ \ 2^{r}e^{C_{r}(T-t)}K^{r}(T-t)^{r}\le\frac{1}{8},\ \ e^{2R_{\varsigma}(T-t)}\le2.
\end{aligned}
\end{equation*}

\noindent Substituting \eqref{eq 4.30} in \eqref{eq 4.29}, we obtain that, for all $n\ge2,\ k\ge1$ and $t\in[\widetilde{T}_{1},T]$,
\begin{equation*}\label{eq 4.31}\tag{4.31}
\begin{aligned}
&\mathbb{E}\big[\sup\limits_{s\in[t,T]}|\widehat{Y}_{s}^{n,k}|^{r}\big]+\mathbb{E}\Big[\Big(\int_{t}^{T}|\widehat{Z}_{s}^{n,k}|^{2}ds\Big)^{\frac{r}{2}}\Big]\le2\int_{t}^{T}\varsigma \Big(\mathbb{E}\big[\sup\limits_{u\in[s,T]}|\widehat{Y}_{u}^{n,k}|^{r}\big]\Big)ds\\
&\hspace{6em}+\frac{1}{8}\mathbb{E}\Big[\Big(\int_{t}^{T}|\widehat{Z}_{s}^{n-1,k}|^{2}ds\Big)^{\frac{r}{2}}\Big]+\frac{1}{8}\mathbb{E}\big[\sup\limits_{s\in[t,T]}|\widehat{Y}_{s}^{n-1,k}|^{r}\big].
\end{aligned}
\end{equation*}

\noindent Hence, by taking into account the linear growth of the function $\varsigma(\cdot)$, and applying Gronwall's inequality, we get, for all $n\ge2,\ k\ge1$ and $t\in[\widetilde{T}_{1},T]$,
\begin{equation*}\label{eq 4.32}\tag{4.32}
\begin{aligned}
&\hspace{2em}\mathbb{E}\big[\sup\limits_{s\in[t,T]}|\widehat{Y}_{s}^{n,k}|^{r}\big]+\mathbb{E}\Big[\Big(\int_{t}^{T}|\widehat{Z}_{s}^{n,k}|^{2}ds\Big)^{\frac{r}{2}}\Big]\\
&\le 4R_{\varsigma}T+\frac{1}{4}\mathbb{E}\Big[\Big(\int_{t}^{T}|\widehat{Z}_{s}^{n-1,k}|^{2}ds\Big)^{\frac{r}{2}}\Big]+\frac{1}{4}\mathbb{E}\big[\sup\limits_{s\in[t,T]}|\widehat{Y}_{s}^{n-1,k}|^{r}\big].
\end{aligned}
\end{equation*}

\noindent Note that for all $a,b\in L^{2}([t,T])$,
\begin{equation*}\label{eq 4.33}\tag{4.33}
\begin{aligned}
\Big(\int_{t}^{T}|a_{s}+b_{s}|^{2}ds\Big)^{\frac{r}{2}}\le 2\Big(\int_{t}^{T}|a_{s}|^{2}ds\Big)^{\frac{r}{2}}+2\Big(\int_{t}^{T}|b_{s}|^{2}ds\Big)^{\frac{r}{2}}.
\end{aligned}
\end{equation*}

\noindent Hence, from \eqref{eq 4.32} and the argument developed for case (i), we get, for all $t\in[\widetilde{T}_{1},T]$,
\begin{equation*}
\begin{aligned}
&\hspace{2em}\sup\limits_{n\ge2}\sup\limits_{k\ge1}\Big(\mathbb{E}\big[\sup\limits_{s\in[t,T]}|\widehat{Y}_{s}^{n,k}|^{r}\big]+\mathbb{E}\Big[\Big(\int_{t}^{T}|\widehat{Z}_{s}^{n,k}|^{2}ds\Big)^{\frac{r}{2}}\Big]\Big)\\
&\le 12R_{\varsigma}T+\mathbb{E}\Big[\Big(\int_{0}^{T}|Z_{s}^{2}-Z_{s}^{1}|^{2}ds\Big)^{\frac{r}{2}}\Big]+\mathbb{E}\big[\sup\limits_{s\in[0,T]}|Y_{s}^{2}-Y_{s}^{1}|^{r}\big]<+\infty.
\end{aligned}
\end{equation*}
\noindent This latter estimate allows to deduce from \eqref{eq 4.31} with the argument used in case (i) that there exists $(Y^{\ast},Z^{\ast})\in\mathcal{S}_{\mathbb{F}}^{r}(\widetilde{T}_{1},T;\mathbb{R}^{n})\times \mathcal{H}_{\mathbb{F}}^{r}(\widetilde{T}_{1},T;\mathbb{R}^{n\times d})$ such that
\begin{equation*}\label{eq 4.34}\tag{4.34}
\begin{aligned}
\lim\limits_{n\rightarrow+\infty}\mathbb{E}\Big[\sup\limits_{s\in[\widetilde{T}_{1},T]}|(Y_{s}^{n}-Y_{s}^{1})-Y^{\ast}_{s}|^{r}+\Big(\int_{\widetilde{T}_{1}}^{T}|(Z_{s}^{n}-Z_{s}^{1})-Z^{\ast}_{s}|^{2}ds\Big)^{\frac{r}{2}}\Big]=0.
\end{aligned}
\end{equation*}

\noindent Recall that $(Y^{1},Z^{1})\in\mathcal{S}_{\mathbb{F}}^{p}(0,T;\mathbb{R}^{n})\times \mathcal{H}_{\mathbb{F}}^{p}(0,T;\mathbb{R}^{n\times d})$, for all $p\in(0,1)$, and $Y^{1}$ is of the class (D). Hence, by passing to the limit in BSDE \eqref{eq 4.14}, we show that $(Y_{t},Z_{t})_{t\in[\widetilde{T}_{1},T]}:=(Y^{\ast}_{t}+Y_{t}^{1},Z^{\ast}_{t}+Z_{t}^{1})_{t\in[\widetilde{T}_{1},T]}$ is an $L^{1}$ solution of the mean-field BSDE \eqref{eq 1.2} on the time interval $[\widetilde{T}_{1},T]$.

As in case (i), a finite iteration extends the solution $(Y,Z)$ to the whole interval $[0,T]$. The proof is complete now.
\end{proof}

  Finally, we come back to the existence and the uniqueness results, but now for more general mean-field BSDEs with integrable parameters. The generator now depends on the solution $(Y,Z)$ but also on its joint distribution. To be more precise, the equation takes the following form:
\begin{equation*}\label{eq 4.35}\tag{4.35}
	Y_{t}=\xi+\int_{t}^{T}f(s,Y_{s},Z_{s},\mathbb{P}_{(Y_{s},Z_{s})})ds-\int_{t}^{T}Z_{s}dW_{s},\ t\in[0,T],
\end{equation*}
where the terminal value $\xi:\Omega\rightarrow\mathbb{R}^{n}$ is an integrable $\mathcal{F}_{T}$-measurable random variable, and the
generator $f:=f(\omega,t,y,z,\mu):\Omega\times[0,T]\times\mathbb{R}^{n}\times\mathbb{R}^{n\times d}\times\mathcal{P}_{1}(\mathbb{R}^{n+n\times d})\rightarrow\mathbb{R}^{n}$ is $\mathbb{F}$-progressively
 measurable, for all fixed $(y,z,\mu)\in\mathbb{R}^{n}\times\mathbb{R}^{n\times d}\times\mathcal{P}_{1}(\mathbb{R}^{n+n\times d})$. Moreover, the following assumptions are supposed to be satisfied:

\noindent\textbf{(B1)}\ There exists $A>0$, such that, for all $t\in[0,T],\ y,\ \bar{y}\in\mathbb{R}^{n}$, $z\in\mathbb{R}^{n\times d}$ and $\mu\in \mathcal{P}_{1}(\mathbb{R}^{n+n\times d}),$
\begin{equation*}
\begin{aligned}
\big\langle y-\bar{y},f(t,y,z,\mu)-f(t,\bar{y},z,\mu)\big\rangle\le A|y-\bar{y}|^{2},\ d\mathbb{P}\times dt\text{-}a.e.
\end{aligned}
\end{equation*}

\noindent\textbf{(B2)}\ There exists $K>0$, such that, for all $t\in[0,T],\ y\in\mathbb{R}^{n}$, $z\in\mathbb{R}^{n\times d}$ and $\mu,\ \bar{\mu}\in \mathcal{P}_{1}(\mathbb{R}^{n+n\times d}),$
\begin{equation*}
\begin{aligned}
|f(t,y,z,\mu)-f(t,y,z,\bar{\mu})|\le KW_{1}(\mu,\bar{\mu}),\ d\mathbb{P}\times dt\text{-}a.e.
\end{aligned}
\end{equation*}

\noindent\textbf{(B3)}\ There exist constants $L>0,\ M>0,\ \alpha\in(0,1)$ and a nonnegative process $(\vartheta_{t})_{t\in[0,T]}\in \mathcal{L}_{\mathbb{F}}^{1}(0,T;\mathbb{R})$, such that, $d\mathbb{P}\times dt\text{-}a.e.$, for all $t\in[0,T],\ y\in\mathbb{R}^{n}$, $z,\ \bar{z}\in\mathbb{R}^{n\times d}$ and $\mu\in \mathcal{P}_{1}(\mathbb{R}^{n+n\times d}),$
\begin{equation*}
\begin{aligned}
|f(t,y,z,\mu)-f(t,y,\bar{z},\mu)|\le L|z-\bar{z}|,\\
\end{aligned}
\end{equation*}
\begin{equation*}
\begin{aligned}
|f(t,y,z,\mu)-f(t,y,0,\mu)|\le M(\vartheta_{t}+|y|+W_{1}(\mu,\delta_{0})+|z|)^{\alpha}.
\end{aligned}
\end{equation*}

\noindent\textbf{(B4)}\ For all $r\ge0$, it holds that $\mathbb{E}\big[\int_{0}^{T}\bar{\phi}_{r}(t)dt\big]<+\infty\ \text{with}\ \bar{\phi}_{r}(t):=\sup_{|y|\le r}|f(t,y,0,\delta_{0})|.$ Moreover, $d\mathbb{P}\times dt\text{-}a.e.$, for all $z\in\mathbb{R}^{n\times d}$ and $\mu\in \mathcal{P}_{1}(\mathbb{R}^{n+n\times d})$, $y\rightarrow f(\omega,t,y,z,\mu)$ is continuous.

\noindent\textbf{(B5)}
\begin{equation*}
\mathbb{E}\bigg[\sup\limits_{t\in[0,T]}\bigg(\mathbb{E}\Big[|\xi|+\int_{0}^{T}|f(s,0,0,\delta_{0})|ds\Big|\mathcal{F}_{t}\Big]\bigg)\bigg]<+\infty.
\end{equation*}

\begin{remark}
Observe that, if we set $\eta(u)\equiv Au,\ u\ge0$ in assumption \textbf{(H1)}, then assumption \textbf{(H1)} reduces to the monotonicity condition of assumption \textbf{(B1)}. Therefore, assumption \textbf{(B1)} is stronger than assumption \textbf{(H1)}.
\end{remark}

\begin{example}\label{ex new}
For all $(\omega,t,y,z,\mu)\in \Omega\times[0,T]\times\mathbb{R}^{n}\times \mathbb{R}^{n\times d}\times\mathcal{P}_{1}(\mathbb{R}^{n+n\times d})$, we consider the generator:
\begin{equation*}
\begin{aligned}
\widehat{F}^{i}(\omega,t,y,z,\mu)&:=|W_{t}(\omega)|+2y_{i}-2y_{i}^{5}+\frac{|z|}{1+|z|}+(|z|^{3}\wedge|z|^{\frac{3}{4}})+3W_{1}(\mu,\delta_{0}),
\end{aligned}
\end{equation*}
with terminal value $\xi_{i}:=2|W_{T}|,\ 1\le i\le n$. It can be verified that $\widehat{F}$ satisfies assumptions \textbf{(B1)}--\textbf{(B5)} with $A=2,\ K=3\sqrt{n},\ L=4\sqrt{n},\ M=2\sqrt{n},\ \vartheta_{t}\equiv1,\ \alpha=\frac{3}{4},\ \bar{\phi}_{r}(t)=\sup_{|y|\le r}\big[\sum_{i=1}^{n}\{|W_{t}|+2y_{i}-2y_{i}^{5}\}^{2}\big]^{\frac{1}{2}},\ \widehat{F}^{i}(t,0,0,\delta_{0})=|W_{t}|,\ 1\le i\le n$.
\end{example}

\begin{theorem}\label{th 4.11}
Let assumptions \textbf{(B1)}--\textbf{(B5)} hold true. \noindent Then the mean-field BSDE \eqref{eq 4.35} admits a unique solution $(Y,Z)\in \mathcal{S}_{\mathbb{F}}^{1}(0,T;\mathbb{R}^{n})\times \mathcal{H}_{\mathbb{F}}^{1}(0,T;\mathbb{R}^{n\times d})$, which is, in particular, an $L^{1}$ solution.
\end{theorem}

\begin{proof}
\noindent We begin with the proof of \underline{the uniqueness}, which we divide into 3 steps.

\noindent\textbf{Step 1.} Let $(Y,Z),\ (\widetilde{Y},\widetilde{Z})\in \mathcal{S}_{\mathbb{F}}^{1}(0,T;\mathbb{R}^{n})\times \mathcal{H}_{\mathbb{F}}^{1}(0,T;\mathbb{R}^{n\times d})$ be two solutions to mean-field BSDE \eqref{eq 4.35}. For the sake of simplicity, we put $\widehat{Y}:=Y-\widetilde{Y}$ and $\widehat{Z}:=Z-\widetilde{Z}.$ Next, we show that $\widehat{Y}\in\mathcal{S}_{\mathbb{F}}^{p}(0,T;\mathbb{R}^{n})$, for all $1<p\le\frac{1}{\alpha}$, where $\alpha$ is introduced in \textbf{(B3)}. To this end, we introduce the following sequence of stopping times: For all $n\ge1$,
\begin{equation*}
\begin{aligned}
\tau_{n}:=\inf\Big\{t\in[0,T]:\int_{0}^{t}\big(|Z_{s}|^{2}+|\widetilde{Z}_{s}|^{2}\big)ds\ge n\Big\}\wedge T.
\end{aligned}
\end{equation*}

\noindent Then, thanks to the It\^{o}-Tanaka formula (see also Briand et al. \cite[Corollary 2.3]{BD03}), we have, for all $t\in[0,T]$,
\begin{equation*}\label{eq 4.36}\tag{4.36}
\begin{aligned}
|\widehat{Y}_{t\wedge \tau_{n}}|\le|\widehat{Y}_{\tau_{n}}|+\int_{t\wedge \tau_{n}}^{\tau_{n}}\big\langle\varrho (\widehat{Y}_{s}),f(s,Y_{s},Z_{s},\mathbb{P}_{(Y_{s},Z_{s})})-f(s,\widetilde{Y}_{s},\widetilde{Z}_{s},\mathbb{P}_{(\widetilde{Y}_{s},\widetilde{Z}_{s})})\big\rangle ds
-\int_{t\wedge \tau_{n}}^{\tau_{n}}\langle\varrho (\widehat{Y}_{s}),\widehat{Z}_{s} dW_{s}\rangle,
\end{aligned}
\end{equation*}
where $\varrho(x):=|x|^{-1}x,\ \text{if}\ x\neq0;\ \varrho(x):=0,\ \text{if}\ x=0$. Moreover, the assumptions \textbf{(B1)}--\textbf{(B3)} give
\begin{equation*}\label{eq 4.37}\tag{4.37}
\begin{aligned}
&\hspace{2em}\big\langle\varrho (\widehat{Y}_{s}),f(s,Y_{s},Z_{s},\mathbb{P}_{(Y_{s},Z_{s})})-f(s,\widetilde{Y}_{s},\widetilde{Z}_{s},\mathbb{P}_{(\widetilde{Y}_{s},\widetilde{Z}_{s})})\big\rangle\\
&\le A|\widehat{Y}_{s}|+K\mathbb{E}[|\widehat{Y}_{s}|]+K\mathbb{E}[|\widehat{Z}_{s}|]+2M(\vartheta_{s}+|\widetilde{Y}_{s}|+\mathbb{E}[|\widetilde{Y}_{s}|]+\mathbb{E}[|\widetilde{Z}_{s}|]+|Z_{s}|+|\widetilde{Z}_{s}|)^{\alpha}.
\end{aligned}
\end{equation*}

\noindent By substituting \eqref{eq 4.37} in \eqref{eq 4.36} and then taking the conditional expectation, we obtain that, for all $n\ge1$ and $t\in[0,T]$,
\begin{equation*}
\begin{aligned}
|\widehat{Y}_{t\wedge \tau_{n}}|\le\mathbb{E}_{t}\Big[|\widehat{Y}_{\tau_{n}}|+\int_{t\wedge \tau_{n}}^{\tau_{n}}\big(A|\widehat{Y}_{s}|+K\mathbb{E}[|\widehat{Y}_{s}|]+K\mathbb{E}[|\widehat{Z}_{s}|]\big)ds\Big]+\widehat{Q}(t),
\end{aligned}
\end{equation*}
where $\widehat{Q}(t):=2M\mathbb{E}_{t}\big[\int_{0}^{T}\big(\vartheta_{s}+|\widetilde{Y}_{s}|+\mathbb{E}[|\widetilde{Y}_{s}|]+\mathbb{E}[|\widetilde{Z}_{s}|]+|Z_{s}|+|\widetilde{Z}_{s}|\big)^{\alpha}ds\big].$ Then, since $(Y,Z),\ (\widetilde{Y},\widetilde{Z})\in \mathcal{S}_{\mathbb{F}}^{1}(0,T;\mathbb{R}^{n})\times \mathcal{H}_{\mathbb{F}}^{1}(0,T;\mathbb{R}^{n\times d})$, an argument similar to the proof of Step 1 in Theorem \ref{th 4.9} yields $\widehat{Y}\in\mathcal{S}_{\mathbb{F}}^{p}(0,T;\mathbb{R}^{n})$, for all $1<p\le\frac{1}{\alpha}$. In particular, we fix some $p_{1}$ such that $1<p_{1}<2\wedge\frac{1}{\alpha}$, and thus $\widehat{Y}\in\mathcal{S}_{\mathbb{F}}^{p_{1}}(0,T;\mathbb{R}^{n})$.

\noindent \textbf{Step 2.} We now show that $\widehat{Z}\in\mathcal{H}_{\mathbb{F}}^{p_{1}}(0,T;\mathbb{R}^{n\times d})$ and there exists a constant $\widehat{C}_{p_{1},A,L,K,T}>0$ depending only on $p_{1},A,L,K,T$ such that for all $\beta>0$,
\begin{equation*}\label{eq 4.38}\tag{4.38}
\begin{aligned}
\mathbb{E}\Big[\Big(\int_{0}^{T}e^{2\beta s}|\widehat{Z}_{s}|^{2}ds\Big)^{\frac{p_{1}}{2}}\Big]\le \widehat{C}_{p_{1},A,L,K,T}\mathbb{E}\big[\sup\limits_{t\in[0,T]}e^{\beta p_{1}t}|\widehat{Y}_{t}|^{p_{1}}\big].
\end{aligned}
\end{equation*}

\noindent Actually, as $\widehat{Y}\in\mathcal{S}_{\mathbb{F}}^{p_{1}}(0,T;\mathbb{R}^{n})$, by using an argument analogous to Step 2 in the proof of Theorem \ref{th 4.9} and applying Lemma \ref{le 4.5} with $\lambda_{1}=A,\ \lambda_{2}=L,\ g_{t}=KW_{1}(\mathbb{P}_{(Y_{t},Z_{t})},\mathbb{P}_{(\widetilde{Y}_{t},\widetilde{Z}_{t})}),\ \theta_{t}\equiv 0$, we deduce that $\widehat{Z}\in\mathcal{H}_{\mathbb{F}}^{p_{1}}(0,T;\mathbb{R}^{n\times d})$. Next, we only need to verify \eqref{eq 4.38}. First, applying It\^{o}'s formula to $e^{2\beta t}|\widehat{Y}_{t}|^{2}$ on $[0,T]$ yields
\begin{equation*}\label{eq 4.39}\tag{4.39}
\begin{aligned}
|\widehat{Y}_{0}|^{2}+2\beta\int_{0}^{T}e^{2\beta s}|\widehat{Y}_{s}|^{2}ds+\int_{0}^{T}e^{2\beta s}|\widehat{Z}_{s}|^{2}ds=2\int_{0}^{T}e^{2\beta s}\langle \widehat{Y}_{s},\Delta \widehat{f}_{s}\rangle ds-2\int_{0}^{T}e^{2\beta s}\langle \widehat{Y}_{s},\widehat{Z}_{s}dW_{s}\rangle,
\end{aligned}
\end{equation*}
\noindent where $\Delta \widehat{f}_{s}:=f(s,Y_{s},Z_{s},\mathbb{P}_{(Y_{s},Z_{s})})-f(s,\widetilde{Y}_{s},\widetilde{Z}_{s},\mathbb{P}_{(\widetilde{Y}_{s},\widetilde{Z}_{s})}),\ s\in[0,T]$. In view of assumptions \textbf{(B1)}--\textbf{(B3)} combined with Young's inequality we get
\begin{equation*}\label{eq 4.40}\tag{4.40}
2\langle \widehat{Y}_{s},\Delta \widehat{f}_{s}\rangle \le (2A+2^{\frac{2+p_{1}}{p_{1}}}K^{2}+2L^{2})|\widehat{Y}_{s}|^{2}+2K|\widehat{Y}_{s}|\mathbb{E}[|\widehat{Y}_{s}|]+2^{-\frac{2+p_{1}}{p_{1}}}\mathbb{E}[|\widehat{Z}_{s}|]^{2}+\frac{1}{2}|\widehat{Z}_{s}|^{2}.
\end{equation*}

\noindent Hence, by setting $\widehat{\mathcal{Y}}_{T}^{\ast}:=\sup\limits_{s\in[0,T]}e^{\beta s}|\widehat{Y}_{s}|$ and substituting \eqref{eq 4.40} into \eqref{eq 4.39}, this implies
\begin{equation*}\label{eq 4.41}\tag{4.41}
\begin{aligned}
\int_{0}^{T}e^{2\beta s}|\widehat{Z}_{s}|^{2}ds&\le(4AT+2^{\frac{2+2p_{1}}{p_{1}}}K^{2}T+4L^{2}T)|\widehat{\mathcal{Y}}_{T}^{\ast}|^{2}+4KT\widehat{\mathcal{Y}}_{T}^{\ast}\mathbb{E}[\widehat{\mathcal{Y}}_{T}^{\ast}]\\
&\hspace{2em}+2^{-\frac{2}{p_{1}}}\int_{0}^{T}e^{2\beta s}\mathbb{E}[|\widehat{Z}_{s}|]^{2}ds+4\sup\limits_{t\in[0,T]}\Big|\int_{0}^{t}e^{2\beta s}\langle\widehat{Y}_{s},\widehat{Z}_{s}dW_{s}\rangle\Big|.
\end{aligned}
\end{equation*}

\noindent From Jensen's inequality we have
\begin{equation*}\label{eq 4.42}\tag{4.42}
\begin{aligned}
\Big(2^{-\frac{2}{p_{1}}}\int_{0}^{T}e^{2\beta s}\mathbb{E}[|\widehat{Z}_{s}|]^{2}ds\Big)^{\frac{p_{1}}{2}}\le \frac{1}{2}\mathbb{E}\Big[\Big(\int_{0}^{T}e^{2\beta s}|\widehat{Z}_{s}|^{2}ds\Big)^{\frac{p_{1}}{2}}\Big].
\end{aligned}
\end{equation*}

\noindent Thus, by raising both sides of \eqref{eq 4.41} to the power $\frac{p_{1}}{2}$ and then taking the expectation, combined with \eqref{eq 4.42} and H\"{o}lder's inequality, we see that there exists a constant $C_{p_{1},A,L,K,T}>0$ depending only on $p_{1},A,L,K,T$ such that
\begin{equation*}\label{eq 4.43}\tag{4.43}
\begin{aligned}
\frac{1}{2}\mathbb{E}\Big[\Big(\int_{0}^{T}e^{2\beta s}|\widehat{Z}_{s}|^{2}ds\Big)^{\frac{p_{1}}{2}}\Big]\le C_{p_{1},A,L,K,T}\bigg(\mathbb{E}[|\widehat{\mathcal{Y}}_{T}^{\ast}|^{p_{1}}]+\mathbb{E}\Big[\sup\limits_{t\in[0,T]}\Big|\int_{0}^{t}e^{2\beta s}\langle \widehat{Y}_{s},\widehat{Z}_{s}dW_{s}\rangle\Big|^{\frac{p_{1}}{2}}\Big]\bigg).
\end{aligned}
\end{equation*}

\noindent Moreover, due to the Burkholder-Davis-Gundy and the Young inequalities, we obtain
\begin{equation*}
\begin{aligned}
C_{p_{1},A,L,K,T}\mathbb{E}\Big[\sup\limits_{t\in[0,T]}\Big|\int_{0}^{t}e^{2\beta s}\langle \widehat{Y}_{s},\widehat{Z}_{s}dW_{s}\rangle\Big|^{\frac{p_{1}}{2}}\Big]\le9C_{p_{1},A,L,K,T}^{2}\mathbb{E}[|\widehat{\mathcal{Y}}_{T}^{\ast}|^{p_{1}}]+\frac{1}{4}\mathbb{E}\Big[\Big(\int_{0}^{T}e^{2\beta s}|\widehat{Z}_{s}|^{2}ds\Big)^{\frac{p_{1}}{2}}\Big].
\end{aligned}
\end{equation*}

\noindent Recall that $\widehat{Z}\in\mathcal{H}_{\mathbb{F}}^{p_{1}}(0,T;\mathbb{R}^{n\times d})$. Consequently, by substituting the latter inequality into \eqref{eq 4.43} and rearranging the terms, it follows that there exists a constant $\widehat{C}_{p_{1},A,L,K,T}>0$ depending only on $p_{1},A,L,K,T$ such that for all $\beta>0$,
\begin{equation*}
\begin{aligned}
\mathbb{E}\Big[\Big(\int_{0}^{T}e^{2\beta s}|\widehat{Z}_{s}|^{2}ds\Big)^{\frac{p_{1}}{2}}\Big]\le \widehat{C}_{p_{1},A,L,K,T}\mathbb{E}\big[\sup\limits_{t\in[0,T]}e^{\beta p_{1}t}|\widehat{Y}_{t}|^{p_{1}}\big].
\end{aligned}
\end{equation*}

\noindent \textbf{Step 3.} The main objective of this step is to show that
\begin{equation*}\label{eq 4.44}\tag{4.44}
\begin{aligned}
\mathbb{E}\big[\sup\limits_{t\in[0,T]}e^{\widehat{\beta} p_{1}t}|\widehat{Y}_{t}|^{p_{1}}\big]\le \frac{1}{2\widehat{C}_{p_{1},A,L,K,T}}\mathbb{E}\Big[\Big(\int_{0}^{T}e^{2\widehat{\beta} s}|\widehat{Z}_{s}|^{2}ds\Big)^{\frac{p_{1}}{2}}\Big],
\end{aligned}
\end{equation*}
for some sufficiently large $\widehat{\beta}>0$. Then, combining this with \eqref{eq 4.38} immediately leads to the uniqueness of the solution to mean-field BSDE \eqref{eq 4.35}.

Let us fix some $\widehat{\beta}>0$ large enough such that
\begin{equation*}\label{eq 4.45}\tag{4.45}
\begin{aligned}
\widehat{\beta}\ge A+2K+\frac{L^{2}}{p_{1}-1}+K^{\frac{p_{1}}{p_{1}-1}}\Big[4\widehat{C}_{p_{1},A,L,K,T}\Big(e^{KT}+\frac{72p_{1}e^{2KT}}{p_{1}-1}\Big)T^{\frac{2-p_{1}}{2}}\Big]^{\frac{1}{p_{1}-1}},
\end{aligned}
\end{equation*}
where $\widehat{C}_{p_{1},A,L,K,T}$ is defined as in \eqref{eq 4.38}. For notational simplicity, we put
\begin{equation*}\label{eq 4.46}\tag{4.46}
\mathcal{Y}_{t}:=e^{\widehat{\beta} t}\widehat{Y}_{t},\ \ \mathcal{Z}_{t}:=e^{\widehat{\beta} t}\widehat{Z}_{t},\ t\in[0,T],\ \ \text{and}\ \ \widehat{J}:=4\widehat{C}_{p_{1},A,L,K,T}\Big(e^{KT}+\frac{72p_{1}e^{2KT}}{p_{1}-1}\Big)T^{\frac{2-p_{1}}{2}}.
\end{equation*}

\noindent Then, applying Corollary 2.3 in \cite{BD03} to the case of $1<p_{1}<2$, we deduce that, for $t\in[0,T]$,
\begin{equation*}\label{eq 4.47}\tag{4.47}
\begin{aligned}
&\hspace{2.6em}|\mathcal{Y}_{t}|^{p_{1}}+\frac{p_{1}(p_{1}-1)}{2}\int_{t}^{T}|\mathcal{Y}_{s}|^{p_{1}-2}\textbf{1}_{\{\mathcal{Y}_{s}\neq0\}}|\mathcal{Z}_{s}|^{2}ds+p_{1}\int_{t}^{T}\widehat{\beta}|\mathcal{Y}_{s}|^{p_{1}}ds\\
&\le p_{1}\int_{t}^{T}|\mathcal{Y}_{s}|^{p_{1}-1}\langle \varrho(\mathcal{Y}_{s}),e^{\widehat{\beta} s}\Delta \widehat{f}_{s}\rangle ds-p_{1}\int_{t}^{T}|\mathcal{Y}_{s}|^{p_{1}-1}\langle \varrho(\mathcal{Y}_{s}),\mathcal{Z}_{s}dW_{s}\rangle.
\end{aligned}
\end{equation*}

\noindent From our assumptions \textbf{(B1)}--\textbf{(B3)}, we have
\begin{equation*}\label{eq 4.48}\tag{4.48}
\begin{aligned}
\langle \varrho(\mathcal{Y}_{s}),e^{\widehat{\beta} s}\Delta \widehat{f}_{s}\rangle \le A|\mathcal{Y}_{s}|+L|\mathcal{Z}_{s}|+K\mathbb{E}[|\mathcal{Y}_{s}|]+K\mathbb{E}[|\mathcal{Z}_{s}|],
\end{aligned}
\end{equation*}

\noindent and we substitute \eqref{eq 4.48} into \eqref{eq 4.47} to obtain that, for $t\in[0,T]$,
\begin{equation*}\label{eq 4.49}\tag{4.49}
\begin{aligned}
&\hspace{2.6em}|\mathcal{Y}_{t}|^{p_{1}}+\frac{p_{1}(p_{1}-1)}{2}\int_{t}^{T}|\mathcal{Y}_{s}|^{p_{1}-2}\textbf{1}_{\{\mathcal{Y}_{s}\neq0\}}|\mathcal{Z}_{s}|^{2}ds+p_{1}\int_{t}^{T}\widehat{\beta}|\mathcal{Y}_{s}|^{p_{1}}ds\\
&\le p_{1}\int_{t}^{T}\Big(A|\mathcal{Y}_{s}|^{p_{1}}+L|\mathcal{Y}_{s}|^{p_{1}-1}|\mathcal{Z}_{s}|+K|\mathcal{Y}_{s}|^{p_{1}-1}\mathbb{E}[|\mathcal{Y}_{s}|]+K|\mathcal{Y}_{s}|^{p_{1}-1}\mathbb{E}[|\mathcal{Z}_{s}|]\Big)ds\\
&\hspace{1.6em}-p_{1}\int_{t}^{T}|\mathcal{Y}_{s}|^{p_{1}-1}\langle \varrho(\mathcal{Y}_{s}),\mathcal{Z}_{s}dW_{s}\rangle.
\end{aligned}
\end{equation*}

\noindent It follows from Young's inequality and H\"{o}lder's inequality that
\begin{equation*}\label{eq 4.50}\tag{4.50}
\begin{aligned}
p_{1}L|\mathcal{Y}_{s}|^{p_{1}-1}|\mathcal{Z}_{s}|\le\frac{p_{1}L^{2}}{p_{1}-1}|\mathcal{Y}_{s}|^{p_{1}}+\frac{p_{1}(p_{1}-1)}{4}|\mathcal{Y}_{s}|^{p_{1}-2}\textbf{1}_{\{\mathcal{Y}_{s}\neq0\}}|\mathcal{Z}_{s}|^{2},
\end{aligned}
\end{equation*}
\begin{equation*}\label{eq 4.51}\tag{4.51}
\begin{aligned}
p_{1}K|\mathcal{Y}_{s}|^{p_{1}-1}\mathbb{E}[|\mathcal{Y}_{s}|]\le (p_{1}-1)K|\mathcal{Y}_{s}|^{p_{1}}+K\mathbb{E}[|\mathcal{Y}_{s}|^{p_{1}}],
\end{aligned}
\end{equation*}
and
\begin{equation*}\label{eq 4.52}\tag{4.52}
\begin{aligned}
p_{1}K|\mathcal{Y}_{s}|^{p_{1}-1}\mathbb{E}[|\mathcal{Z}_{s}|]\le (p_{1}-1)K^{\frac{p_{1}}{p_{1}-1}}\widehat{J}^{\frac{1}{p_{1}-1}}|\mathcal{Y}_{s}|^{p_{1}}+\widehat{J}^{-1}\mathbb{E}[|\mathcal{Z}_{s}|^{p_{1}}].
\end{aligned}
\end{equation*}

\noindent By inserting \eqref{eq 4.50}--\eqref{eq 4.52} into \eqref{eq 4.49}, we get that, for $t\in[0,T]$,
\begin{equation*}
\begin{aligned}
&\hspace{2.6em}|\mathcal{Y}_{t}|^{p_{1}}+\frac{p_{1}(p_{1}-1)}{4}\int_{t}^{T}|\mathcal{Y}_{s}|^{p_{1}-2}\textbf{1}_{\{\mathcal{Y}_{s}\neq0\}}|\mathcal{Z}_{s}|^{2}ds+p_{1}\int_{t}^{T}\widehat{\beta}|\mathcal{Y}_{s}|^{p_{1}}ds\\
&\le \Big[Ap_{1}+\frac{p_{1}L^{2}}{p_{1}-1}+(p_{1}-1)K+(p_{1}-1)K^{\frac{p_{1}}{p_{1}-1}}\widehat{J}^{\frac{1}{p_{1}-1}}\Big]\int_{t}^{T}|\mathcal{Y}_{s}|^{p_{1}}ds\\
&\hspace{1.6em}+K\int_{t}^{T}\mathbb{E}[|\mathcal{Y}_{s}|^{p_{1}}]ds+\widehat{J}^{-1}\int_{t}^{T}\mathbb{E}[|\mathcal{Z}_{s}|^{p_{1}}]ds-p_{1}\int_{t}^{T}|\mathcal{Y}_{s}|^{p_{1}-1}\langle \varrho(\mathcal{Y}_{s}),\mathcal{Z}_{s}dW_{s}\rangle.
\end{aligned}
\end{equation*}

\noindent From H\"{o}lder's inequality, since $1<p_{1}<2$, we have
\begin{equation*}
\begin{aligned}
\int_{t}^{T}\mathbb{E}[|\mathcal{Z}_{s}|^{p_{1}}]ds \le T^{\frac{2-p_{1}}{2}}\mathbb{E}\Big[\Big(\int_{t}^{T}|\mathcal{Z}_{s}|^{2}ds\Big)^{\frac{p_{1}}{2}}\Big].
\end{aligned}
\end{equation*}

\noindent Hence, for $t\in[0,T]$,
\begin{equation*}\label{eq 4.53}\tag{4.53}
\begin{aligned}
&\hspace{2.6em}|\mathcal{Y}_{t}|^{p_{1}}+\frac{p_{1}(p_{1}-1)}{4}\int_{t}^{T}|\mathcal{Y}_{s}|^{p_{1}-2}\textbf{1}_{\{\mathcal{Y}_{s}\neq0\}}|\mathcal{Z}_{s}|^{2}ds+p_{1}\int_{t}^{T}\widehat{\beta}|\mathcal{Y}_{s}|^{p_{1}}ds\\
&\le \Big[Ap_{1}+\frac{p_{1}L^{2}}{p_{1}-1}+(p_{1}-1)K+(p_{1}-1)K^{\frac{p_{1}}{p_{1}-1}}\widehat{J}^{\frac{1}{p_{1}-1}}\Big]\int_{t}^{T}|\mathcal{Y}_{s}|^{p_{1}}ds+K\int_{t}^{T}\mathbb{E}[|\mathcal{Y}_{s}|^{p_{1}}]ds\\
&\hspace{1.6em}+\widehat{J}^{-1} T^{\frac{2-p_{1}}{2}}\mathbb{E}\Big[\Big(\int_{t}^{T}|\mathcal{Z}_{s}|^{2}ds\Big)^{\frac{p_{1}}{2}}\Big]-p_{1}\int_{t}^{T}|\mathcal{Y}_{s}|^{p_{1}-1}\langle \varrho(\mathcal{Y}_{s}),\mathcal{Z}_{s}dW_{s}\rangle.
\end{aligned}
\end{equation*}

\noindent Thanks to Burkholder-Davis-Gundy's inequality and Young's inequality, it is not difficult to check that $\{\int_{0}^{t}|\mathcal{Y}_{s}|^{p_{1}-1}\langle \varrho(\mathcal{Y}_{s}),\mathcal{Z}_{s}dW_{s}\rangle\}_{t\in[0,T]}$ is a uniformly integrable martingale. Thus, by taking the expectation on both sides of \eqref{eq 4.53} and recalling \eqref{eq 4.45}, we obtain that
\begin{equation*}\label{eq 4.54}\tag{4.54}
\begin{aligned}
\frac{p_{1}(p_{1}-1)}{4}\mathbb{E}\Big[\int_{0}^{T}|\mathcal{Y}_{s}|^{p_{1}-2}\textbf{1}_{\{\mathcal{Y}_{s}\neq0\}}|\mathcal{Z}_{s}|^{2}ds\Big]\le \widehat{J}^{-1} T^{\frac{2-p_{1}}{2}}\mathbb{E}\Big[\Big(\int_{0}^{T}|\mathcal{Z}_{s}|^{2}ds\Big)^{\frac{p_{1}}{2}}\Big].
\end{aligned}
\end{equation*}

\noindent Moreover, with using \eqref{eq 4.45} again, by taking first the supremum on both sides of \eqref{eq 4.53} and then the expectation, we get, for $t\in[0,T]$,
\begin{equation*}
\begin{aligned}
\mathbb{E}\big[\sup\limits_{s\in[t,T]}|\mathcal{Y}_{s}|^{p_{1}}\big]&\le K\int_{t}^{T}\mathbb{E}\big[\sup\limits_{r\in[s,T]}|\mathcal{Y}_{r}|^{p_{1}}\big]ds+\widehat{J}^{-1} T^{\frac{2-p_{1}}{2}}\mathbb{E}\Big[\Big(\int_{0}^{T}|\mathcal{Z}_{s}|^{2}ds\Big)^{\frac{p_{1}}{2}}\Big]\\
&\hspace{2em}+p_{1}\mathbb{E}\Big[\sup\limits_{t\in[0,T]}\Big|\int_{t}^{T}|\mathcal{Y}_{s}|^{p_{1}-1}\langle \varrho(\mathcal{Y}_{s}),\mathcal{Z}_{s}dW_{s}\rangle\Big|\Big].
\end{aligned}
\end{equation*}

\noindent Then, Gronwall's inequality yields
\begin{equation*}\label{eq 4.55}\tag{4.55}
\begin{aligned}
\mathbb{E}\big[\sup\limits_{s\in[0,T]}|\mathcal{Y}_{s}|^{p_{1}}\big]&\le e^{KT}\widehat{J}^{-1} T^{\frac{2-p_{1}}{2}}\mathbb{E}\Big[\Big(\int_{0}^{T}|\mathcal{Z}_{s}|^{2}ds\Big)^{\frac{p_{1}}{2}}\Big]\\
&\hspace{2em}+p_{1}e^{KT}\mathbb{E}\Big[\sup\limits_{t\in[0,T]}\Big|\int_{t}^{T}|\mathcal{Y}_{s}|^{p_{1}-1}\langle \varrho(\mathcal{Y}_{s}),\mathcal{Z}_{s}dW_{s}\rangle\Big|\Big].
\end{aligned}
\end{equation*}

\noindent We observe that by the Burkholder-Davis-Gundy and the Young inequalities,
\begin{equation*}\label{eq 4.56}\tag{4.56}
\begin{aligned}
&\hspace{2em}p_{1}e^{KT}\mathbb{E}\Big[\sup\limits_{t\in[0,T]}\Big|\int_{t}^{T}|\mathcal{Y}_{s}|^{p_{1}-1}\langle \varrho(\mathcal{Y}_{s}),\mathcal{Z}_{s}dW_{s}\rangle\Big|\Big]\\
&\le 6p_{1}e^{KT}\mathbb{E}\Big[\sup\limits_{s\in[0,T]}|\mathcal{Y}_{s}|^{\frac{p_{1}}{2}}\Big(\int_{0}^{T}|\mathcal{Y}_{s}|^{p_{1}-2}\textbf{1}_{\{\mathcal{Y}_{s}\neq0\}}|\mathcal{Z}_{s}|^{2}ds\Big)^{\frac{1}{2}}\Big]\\
&\le \frac{1}{2}\mathbb{E}\big[\sup\limits_{s\in[0,T]}|\mathcal{Y}_{s}|^{p_{1}}\big]+18p_{1}^{2}e^{2KT}\mathbb{E}\Big[\int_{0}^{T}|\mathcal{Y}_{s}|^{p_{1}-2}\textbf{1}_{\{\mathcal{Y}_{s}\neq0\}}|\mathcal{Z}_{s}|^{2}ds\Big].
\end{aligned}
\end{equation*}

\noindent Hence, combining \eqref{eq 4.54}--\eqref{eq 4.56} gives
\begin{equation*}
\begin{aligned}
\mathbb{E}\big[\sup\limits_{s\in[0,T]}|\mathcal{Y}_{s}|^{p_{1}}\big]&\le 2\Big(e^{KT}+\frac{72p_{1}e^{2KT}}{p_{1}-1}\Big)T^{\frac{2-p_{1}}{2}}\widehat{J}^{-1}\mathbb{E}\Big[\Big(\int_{0}^{T}|\mathcal{Z}_{s}|^{2}ds\Big)^{\frac{p_{1}}{2}}\Big]\\
&=\frac{1}{2\widehat{C}_{p_{1},A,L,K,T}}\mathbb{E}\Big[\Big(\int_{0}^{T}|\mathcal{Z}_{s}|^{2}ds\Big)^{\frac{p_{1}}{2}}\Big].
\end{aligned}
\end{equation*}
Recall the notations in \eqref{eq 4.46}, \eqref{eq 4.44} is thus proved.

Notice that \eqref{eq 4.38} and \eqref{eq 4.44} imply the uniqueness of solution for mean-field BSDE \eqref{eq 4.35}.

\smallskip

\noindent It remains to prove \underline{the existence}. Given a pair of processes $(U,V)\in\mathcal{S}_{\mathbb{F}}^{1}(0,T;\mathbb{R}^{n})\times \mathcal{H}_{\mathbb{F}}^{1}(0,T;\mathbb{R}^{n\times d})$, we consider the BSDE
\begin{equation*}\label{eq 4.57}\tag{4.57}
	Y_{t}=\xi+\int_{t}^{T}f(s,Y_{s},Z_{s},\mathbb{P}_{(U_{s},V_{s})})ds-\int_{t}^{T}Z_{s}dW_{s},\ t\in[0,T].
\end{equation*}
Thanks to assumptions \textbf{(B2)} and \textbf{(B5)}, we obtain
\begin{equation*}
\begin{aligned}
&\mathbb{E}\bigg[\sup\limits_{t\in[0,T]}\bigg(\mathbb{E}\Big[|\xi|+\int_{0}^{T}|f(s,0,0,\mathbb{P}_{(U_{s},V_{s})})|ds\Big|\mathcal{F}_{t}\Big]\bigg)\bigg]\\
\le\ &\mathbb{E}\bigg[\sup\limits_{t\in[0,T]}\bigg(\mathbb{E}\Big[|\xi|+\int_{0}^{T}|f(s,0,0,\delta_{0})|ds\Big|\mathcal{F}_{t}\Big]\bigg)\bigg]+KT\mathbb{E}\big[\sup\limits_{s\in[0,T]}|U_{s}|\big]\\
&\hspace{1em}+KT^{\frac{1}{2}}\mathbb{E}\Big[\Big(\int_{0}^{T}|V_{s}|^{2}ds\Big)^{\frac{1}{2}}\Big]<+\infty.
\end{aligned}
\end{equation*}
Then, based on our assumptions \textbf{(B1)}--\textbf{(B4)}, it follows from \cite[Theorem 3]{F18} that BSDE \eqref{eq 4.57} has a unique solution $(Y,Z)\in \mathcal{S}_{\mathbb{F}}^{1}(0,T;\mathbb{R}^{n})\times \mathcal{H}_{\mathbb{F}}^{1}(0,T;\mathbb{R}^{n\times d})$. This allows to consider the following Picard iteration scheme: Put $(Y^{0},Z^{0}):=(0,0)$, and then define recursively $(Y^{n},Z^{n})\in \mathcal{S}_{\mathbb{F}}^{1}(0,T;\mathbb{R}^{n})\times \mathcal{H}_{\mathbb{F}}^{1}(0,T;\mathbb{R}^{n\times d}),\ n\ge1$, as the unique solution to the following BSDE:
\begin{equation*}\label{eq 4.58}\tag{4.58}
	Y_{t}^{n}=\xi+\int_{t}^{T}f(s,Y_{s}^{n},Z_{s}^{n},\mathbb{P}_{(Y_{s}^{n-1},Z_{s}^{n-1})})ds-\int_{t}^{T}Z_{s}^{n}dW_{s},\ t\in[0,T].
\end{equation*}

\noindent Next, we set $\widehat{Y}^{n,k}:=Y^{n+k}-Y^{n}$ and $\widehat{Z}^{n,k}:=Z^{n+k}-Z^{n}$, for all $n,k\ge1$. Then, $(\widehat{Y}^{n,k},\widehat{Z}^{n,k})\in \mathcal{S}_{\mathbb{F}}^{1}(0,T;\mathbb{R}^{n})\times \mathcal{H}_{\mathbb{F}}^{1}(0,T;\mathbb{R}^{n\times d})$ solves the following BSDE:
\begin{equation*}\label{eq 4.59}\tag{4.59}
\begin{aligned}
\widehat{Y}_{t}^{n,k}=\int_{t}^{T}\Delta\widehat{f}^{n,k}(s,\widehat{Y}_{s}^{n,k},\widehat{Z}_{s}^{n,k})ds-\int_{t}^{T}\widehat{Z}_{s}^{n,k}dW_{s},\ t\in[0,T],
\end{aligned}
\end{equation*}
where $\Delta\widehat{f}^{n,k}(s,y,z):=f(s,y+Y_{s}^{n},z+Z_{s}^{n},\mathbb{P}_{(Y_{s}^{n+k-1},Z_{s}^{n+k-1})})-f(s,Y_{s}^{n},Z_{s}^{n},\mathbb{P}_{(Y_{s}^{n-1},Z_{s}^{n-1})})$, $(s,y,z)\in[0,T]\times\mathbb{R}^{n}\times\mathbb{R}^{n\times d}$. By employing an argument similar to the proof of \eqref{eq 4.5} for Theorem \ref{th 4.9} and using assumptions \textbf{(B1)}--\textbf{(B3)}, we show that, for $\gamma^{n,k}_{1}:=\Vert\widehat{Y}^{n,k}\Vert_{\mathcal{S}^{1}_{\mathbb{F}}(0,T)}$ and $\gamma^{n,k}_{2}:=\Vert\widehat{Z}^{n,k}\Vert_{\mathcal{H}^{1}_{\mathbb{F}}(0,T)}$,
\begin{equation*}
\begin{aligned}
|\widehat{Y}_{t}^{n,k}|\le \big(\gamma^{n-1,k}_{1}KT+\gamma^{n-1,k}_{2}KT^{\frac{1}{2}}+\widehat{Q}^{n,k}(t)\big)e^{AT},\ t\in[0,T],\ n,k\ge1,
\end{aligned}
\end{equation*}
where $\widehat{Q}^{n,k}(t):=2M\mathbb{E}_{t}\big[\int_{0}^{T}\big(\vartheta_{s}+|Y_{s}^{n}|+\mathbb{E}[|Y_{s}^{n-1}|]+\mathbb{E}[|Z_{s}^{n-1}|]+|Z_{s}^{n}|+|Z^{n+k}_{s}|\big)^{\alpha}ds\big]\in \mathcal{S}_{\mathbb{F}}^{r}(0,T;\mathbb{R})$, given that $\alpha r\le1$ with $r>1$. Then again we get that, for all $n,k\ge1$, $(\widehat{Y}_{t}^{n,k})_{t\in[0,T]}\in \mathcal{S}_{\mathbb{F}}^{r}(0,T;\mathbb{R}^{n})$, for $\alpha r\le1$ for $r>1$. In particular, we fix some $1<r_{0}\le2$ such that $\alpha r_{0}\le1$, and thus $(\widehat{Y}_{t}^{n,k})_{t\in[0,T]}\in \mathcal{S}_{\mathbb{F}}^{r_{0}}(0,T;\mathbb{R}^{n}),\ \text{for all}\ n,k\ge1$.

On the other hand, in view of \eqref{eq 4.59}, by combining assumptions \textbf{(B1)}--\textbf{(B3)}, we have, $d\mathbb{P}\times ds\text{-}a.e.$, for all $s\in[0,T],\ y\in\mathbb{R}^{n}$ and $z\in\mathbb{R}^{n\times d}$,
\begin{equation*}\label{eq 4.60}\tag{4.60}
\begin{aligned}
\langle y,\Delta\widehat{f}^{n,k}(s,y,z)\rangle\le A|y|^{2}+L|y||z|+K|y|\mathbb{E}[|\widehat{Y}_{s}^{n-1,k}|+|\widehat{Z}_{s}^{n-1,k}|].
\end{aligned}
\end{equation*}
Then, by applying Lemma \ref{le new} with $p=r_{0},\ \lambda_{1}=A,\ \lambda_{2}=L,\ g_{t}=K\mathbb{E}[|\widehat{Y}_{t}^{n-1,k}|+|\widehat{Z}_{t}^{n-1,k}|],$ we get  $\widehat{Z}^{n,k}\in \mathcal{H}_{\mathbb{F}}^{r_{0}}(0,T;\mathbb{R}^{n\times d})$. Moreover, there exists two constants $\widetilde{C}_{r_{0}}>0$ depending only on $r_{0}$ and $C_{r_{0},A,L}>0$ depending only on $r_{0},A,L$, such that, for all $n\ge2,\ k\ge1$ and $t\in[0,T]$,
\begin{equation*}\label{eq 4.61}\tag{4.61}
\begin{aligned}
&\hspace{2em}\mathbb{E}\big[\sup\limits_{s\in[t,T]}|\widehat{Y}_{s}^{n,k}|^{r_{0}}\big]+\mathbb{E}\Big[\Big(\int_{t}^{T}|\widehat{Z}_{s}^{n,k}|^{2}ds\Big)^{\frac{r_{0}}{2}}\Big]\\
&\le(2K)^{r_{0}}\widetilde{C}_{r_{0}}e^{C_{r_{0},A,L}(T-t)}\Big((T-t)^{r_{0}}\mathbb{E}\big[\sup\limits_{s\in[t,T]}|\widehat{Y}_{s}^{n-1,k}|^{r_{0}}\big]+(T-t)^{\frac{r_{0}}{2}}\mathbb{E}\Big[\Big(\int_{t}^{T}|\widehat{Z}_{s}^{n-1,k}|^{2}ds\Big)^{\frac{r_{0}}{2}}\Big]\Big).
\end{aligned}
\end{equation*}

\noindent With the notations
\begin{equation*}
\begin{aligned}
\upsilon:=\min\Big\{\frac{\ln 2}{C_{r_{0},A,L}},\Big(\frac{1}{8\cdot(2K)^{r_{0}}\widetilde{C}_{r_{0}}}\Big)^{\frac{1}{r_{0}}},\Big(\frac{1}{8\cdot(2K)^{r_{0}}\widetilde{C}_{r_{0}}}\Big)^{\frac{2}{r_{0}}}\Big\}\ \ \text{and}\ \ \overline{T}_{j}:=(T-j\upsilon)\vee 0,\ j\ge1,
\end{aligned}
\end{equation*}

\noindent we have, for all $t\in[\overline{T}_{1},T]$,
\begin{equation*}\label{eq 4.62}\tag{4.62}
\begin{aligned}
e^{C_{r_{0},A,L}(T-t)}\le 2,\ \ (2K)^{r_{0}}\widetilde{C}_{r_{0}}(T-t)^{r_{0}}\le \frac{1}{8},\ \ (2K)^{r_{0}}\widetilde{C}_{r_{0}}(T-t)^{\frac{r_{0}}{2}}\le\frac{1}{8}.
\end{aligned}
\end{equation*}

\noindent Then we substitute \eqref{eq 4.62} in \eqref{eq 4.61} to obtain that, for all $n\ge2,\ k\ge1$ and $t\in[\overline{T}_{1},T]$,
\begin{equation*}\label{eq 4.63}\tag{4.63}
\begin{aligned}
\hspace{2em}\mathbb{E}\big[\sup\limits_{s\in[t,T]}|\widehat{Y}_{s}^{n,k}|^{r_{0}}\big]+\mathbb{E}\Big[\Big(\int_{t}^{T}|\widehat{Z}_{s}^{n,k}|^{2}ds\Big)^{\frac{r_{0}}{2}}\Big]\le\frac{1}{4}\mathbb{E}\big[\sup\limits_{s\in[t,T]}|\widehat{Y}_{s}^{n-1,k}|^{r_{0}}\big]+\frac{1}{4}\mathbb{E}\Big[\Big(\int_{t}^{T}|\widehat{Z}_{s}^{n-1,k}|^{2}ds\Big)^{\frac{r_{0}}{2}}\Big].
\end{aligned}
\end{equation*}
Therefore, by using the argument developed in Step 3 of Theorem \ref{th 4.10}, we get, for all $t\in[\overline{T}_{1},T]$,
\begin{equation*}
\begin{aligned}
&\hspace{2em}\sup\limits_{n\ge2}\sup\limits_{k\ge1}\Big(\mathbb{E}\big[\sup\limits_{s\in[t,T]}|\widehat{Y}_{s}^{n,k}|^{r_{0}}\big]+\mathbb{E}\Big[\Big(\int_{t}^{T}|\widehat{Z}_{s}^{n,k}|^{2}ds\Big)^{\frac{r_{0}}{2}}\Big]\Big)\\
&\le\mathbb{E}\big[\sup\limits_{s\in[0,T]}|Y_{s}^{2}-Y_{s}^{1}|^{r_{0}}\big]+\mathbb{E}\Big[\Big(\int_{0}^{T}|Z_{s}^{2}-Z_{s}^{1}|^{2}ds\Big)^{\frac{r_{0}}{2}}\Big]<+\infty.
\end{aligned}
\end{equation*}

\noindent This latter estimate allows to conclude from \eqref{eq 4.63} with an argument similar to the proof of \eqref{eq 4.26} in Theorem \ref{th 4.10} that there exists $(Y^{\ast},Z^{\ast})\in\mathcal{S}_{\mathbb{F}}^{r_{0}}(\overline{T}_{1},T;\mathbb{R}^{n})\times \mathcal{H}_{\mathbb{F}}^{r_{0}}(\overline{T}_{1},T;\mathbb{R}^{n\times d})$ such that
\begin{equation*}
\begin{aligned}
\lim\limits_{n\rightarrow+\infty}\mathbb{E}\Big[\sup\limits_{s\in[\overline{T}_{1},T]}|(Y_{s}^{n}-Y_{s}^{1})-Y^{\ast}_{s}|^{r_{0}}+\Big(\int_{\overline{T}_{1}}^{T}|(Z_{s}^{n}-Z_{s}^{1})-Z^{\ast}_{s}|^{2}ds\Big)^{\frac{r_{0}}{2}}\Big]=0.
\end{aligned}
\end{equation*}

\noindent Recall that $(Y^{1},Z^{1})\in\mathcal{S}_{\mathbb{F}}^{1}(0,T;\mathbb{R}^{n})\times \mathcal{H}_{\mathbb{F}}^{1}(0,T;\mathbb{R}^{n\times d})$. Hence, by taking the limit in BSDE \eqref{eq 4.58}, we prove that $(Y_{t},Z_{t})_{t\in[\overline{T}_{1},T]}:=(Y^{\ast}_{t}+Y_{t}^{1},Z^{\ast}_{t}+Z_{t}^{1})_{t\in[\overline{T}_{1},T]}\in\mathcal{S}_{\mathbb{F}}^{1}(\overline{T}_{1},T;\mathbb{R}^{n})\times \mathcal{H}_{\mathbb{F}}^{1}(\overline{T}_{1},T;\mathbb{R}^{n\times d})$ is a solution of the mean-field BSDE \eqref{eq 4.35} on the time interval $[\overline{T}_{1},T]$.

As in Step 3 of Theorem \ref{th 4.10}, a finite iteration extends the solution $(Y,Z)$ to the whole interval $[0,T]$. The proof is now complete.

\end{proof}

\section*{Appendix}

\begin{mylemma}\label{le 4.5}
Let $p>0$. Suppose that the generator $f:\Omega\times[0,T]\times\mathbb{R}^{n}\times\mathbb{R}^{n\times d}\rightarrow\mathbb{R}^{n}$ is $\mathbb{F}$-progressively measurable and there exist constants $\lambda_{1}>0,\ \lambda_{2}>0$, and two nonnegative processes $g=(g_{t})_{t\in[0,T]}\in\mathcal{L}^{p}_{\mathbb{F}}(0,T;\mathbb{R})$, and $\theta=(\theta_{t})_{t\in[0,T]}\in\mathcal{L}^{\frac{p}{2}}_{\mathbb{F}}(0,T;\mathbb{R})$ such that, $d\mathbb{P}\times dt\text{-}a.e.$, for all $t\in[0,T],\ y\in\mathbb{R}^{n}$ and $z\in\mathbb{R}^{n\times d}$,
\begin{equation*}
\begin{aligned}
\langle y,f(t,y,z)\rangle \le \lambda_{1}|y|^{2}+\lambda_{2}|y||z|+|y|g_{t}+\theta_{t}.
\end{aligned}
\end{equation*}

\noindent Let $(Y,Z)$ be a solution of BSDE \eqref{eq 1.1} such that $Y\in\mathcal{S}^{p}_{\mathbb{F}}(0,T;\mathbb{R}^{n})$. Then $Z\in \mathcal{H}^{p}_{\mathbb{F}}(0,T;\mathbb{R}^{n\times d})$ and there exists a constant $C_{p,\lambda_{1},\lambda_{2},T}>0$ depending only on $p,\lambda_{1},\lambda_{2},T$ as well as another constant $C_{p}>0$ depending only on $p$, such that, for all $t\in[0,T]$,
\begin{equation*}
\begin{aligned}
\mathbb{E}\Big[\Big(\int_{t}^{T}|Z_{s}|^{2}ds\Big)^{\frac{p}{2}}\Big]\le C_{p,\lambda_{1},\lambda_{2},T}\mathbb{E}\big[\sup\limits_{s\in[t,T]}|Y_{s}|^{p}\big]+C_{p}\Big(\mathbb{E}\Big[\Big(\int_{t}^{T}g_{s}ds\Big)^{p}\Big]+\mathbb{E}\Big[\Big(\int_{t}^{T}\theta_{s}ds\Big)^{\frac{p}{2}}\Big]\Big).
\end{aligned}
\end{equation*}
\end{mylemma}

\begin{mylemma}\label{le 4.6}
Let $p>1$. Suppose that the generator $f:\Omega\times[0,T]\times\mathbb{R}^{n}\times\mathbb{R}^{n\times d}\rightarrow\mathbb{R}^{n}$ is $\mathbb{F}$-progressively measurable and there exist a constant $\lambda>0$ and a nonnegative process $g=(g_{t})_{t\in[0,T]}\in\mathcal{L}^{p}_{\mathbb{F}}(0,T;\mathbb{R})$, such that, $d\mathbb{P}\times dt\text{-}a.e.$, for all $t\in[0,T],\ y\in\mathbb{R}^{n}$ and $z\in\mathbb{R}^{n\times d}$,
\begin{equation*}
\begin{aligned}
|y|^{p-1}\Big\langle\frac{y}{|y|}\mathbf{1}_{\{|y|\neq0\}},f(t,y,z)\Big\rangle\le \varphi(|y|^{p})+\lambda|y|^{p-1}|z|+|y|^{p-1}g_{t},
\end{aligned}
\end{equation*}
where $\varphi(\cdot):[0,+\infty)\rightarrow[0,+\infty)$ is a nondecreasing and concave function with $\varphi(0)=0$.

Let $(Y,Z)\in\mathcal{S}_{\mathbb{F}}^{p}(0,T;\mathbb{R}^{n})\times \mathcal{H}_{\mathbb{F}}^{p}(0,T;\mathbb{R}^{n\times d})$ be a solution of BSDE \eqref{eq 1.1}. Then there exists a constant $C_{p,\lambda}>0$ depending only on $p,\lambda$ such that, for all $t\in[0,T]$,
\begin{equation*}
\begin{aligned}
\mathbb{E}\big[\sup\limits_{s\in[t,T]}|Y_{s}|^{p}\big]\le e^{C_{p,\lambda}(T-t)}\Big(\mathbb{E}[|\xi|^{p}]+\int_{t}^{T}\varphi(\mathbb{E}[|Y_{s}|^{p}])ds+\mathbb{E}\Big[\Big(\int_{t}^{T}g_{s}ds\Big)^{p}\Big]\Big).
\end{aligned}
\end{equation*}
\end{mylemma}

\begin{mylemma}\label{le 4.7}
Let the generator $f$ satisfy the assumptions of Lemma \ref{le 4.6} with $p=2$. Suppose that $(Y,Z)\in\mathcal{S}_{\mathbb{F}}^{2}(0,T;\mathbb{R}^{n})\times \mathcal{H}_{\mathbb{F}}^{2}(0,T;\mathbb{R}^{n\times d})$ is a solution of BSDE \eqref{eq 1.1}. Then there exists a constant $C_{\lambda}>0$ depending only on $\lambda$ such that, for all $t\in[0,T]$,
\begin{equation*}
\begin{aligned}
\mathbb{E}\big[\sup\limits_{s\in[t,T]}|Y_{s}|^{2}\big]+\mathbb{E}\Big[\int_{t}^{T}|Z_{s}|^{2}ds\Big]\le e^{C_{\lambda}(T-t)}\Big(\mathbb{E}[|\xi|^{2}]+\int_{t}^{T}\varphi(\mathbb{E}[|Y_{s}|^{2}])ds+\mathbb{E}\Big[\Big(\int_{t}^{T}g_{s}ds\Big)^{2}\Big]\Big).
\end{aligned}
\end{equation*}
\end{mylemma}

\begin{mylemma}\label{le 4.8}
Let $p>1$. Suppose that the generator $f:\Omega\times[0,T]\times\mathbb{R}^{n}\times\mathbb{R}^{n\times d}\rightarrow\mathbb{R}^{n}$ is $\mathbb{F}$-progressively measurable and there exists a constant $\lambda>0$ together with a nonnegative process $g=(g_{t})_{t\in[0,T]}\in\mathcal{L}^{p}_{\mathbb{F}}(0,T;\mathbb{R})$, such that, $d\mathbb{P}\times dt\text{-}a.e.$, for all $t\in[0,T],\ y\in\mathbb{R}^{n}$ and $z\in\mathbb{R}^{n\times d}$,
\begin{equation*}
\begin{aligned}
\Big\langle\frac{y}{|y|}\mathbf{1}_{\{|y|\neq0\}},f(t,y,z)\Big\rangle\le \varpi^{\frac{1}{p}}(|y|^{p})+\lambda|z|+g_{t},
\end{aligned}
\end{equation*}
where $\varpi(\cdot):[0,+\infty)\rightarrow[0,+\infty)$ is a nondecreasing and concave function with $\varpi(0)=0$.

Let $(Y,Z)$ be a solution of BSDE \eqref{eq 1.1} such that $Y\in\mathcal{S}^{p}_{\mathbb{F}}(0,T;\mathbb{R}^{n})$. Then $Z\in \mathcal{H}^{p}_{\mathbb{F}}(0,T;\mathbb{R}^{n\times d})$ and there exists a constant $C_{p,\lambda}>0$ depending only on $p,\lambda$ such that, for all $t\in[0,T]$,
\begin{equation*}
\begin{aligned}
\mathbb{E}\big[\sup\limits_{s\in[t,T]}|Y_{s}|^{p}\big]+\mathbb{E}\Big[\Big(\int_{t}^{T}|Z_{s}|^{2}ds\Big)^{\frac{p}{2}}\Big]\le e^{C_{p,\lambda}(T-t)}\Big(\mathbb{E}[|\xi|^{p}]+\int_{t}^{T}\varpi(\mathbb{E}[|Y_{s}|^{p}])ds+\mathbb{E}\Big[\Big(\int_{t}^{T}g_{s}ds\Big)^{p}\Big]\Big).
\end{aligned}
\end{equation*}
\end{mylemma}

\begin{mylemma}\label{le new}
Let $p>1$. Suppose that the generator $f:\Omega\times[0,T]\times\mathbb{R}^{n}\times\mathbb{R}^{n\times d}\rightarrow\mathbb{R}^{n}$ is $\mathbb{F}$-progressively measurable and there exist constants $\lambda_{1}>0,\ \lambda_{2}>0$, and a nonnegative process $g=(g_{t})_{t\in[0,T]}\in\mathcal{L}^{p}_{\mathbb{F}}(0,T;\mathbb{R})$, such that, $d\mathbb{P}\times dt\text{-}a.e.$, for all $t\in[0,T],\ y\in\mathbb{R}^{n}$ and $z\in\mathbb{R}^{n\times d}$,
\begin{equation*}
\begin{aligned}
\langle \frac{y}{|y|}\mathbf{1}_{\{|y|\neq0\}},f(t,y,z)\rangle \le \lambda_{1}|y|+\lambda_{2}|z|+g_{t}.
\end{aligned}
\end{equation*}

\noindent Let $(Y,Z)$ be a solution of BSDE \eqref{eq 1.1} such that $Y\in\mathcal{S}^{p}_{\mathbb{F}}(0,T;\mathbb{R}^{n})$. Then $Z\in \mathcal{H}^{p}_{\mathbb{F}}(0,T;\mathbb{R}^{n\times d})$ and there exists a constant $C_{p}>0$ depending only on $p$, such that, for all $\beta\ge \lambda_{1}+\frac{\lambda_{2}^{2}}{1\wedge(p-1)}$ and $t\in[0,T]$,
\begin{equation*}
\begin{aligned}
\mathbb{E}\Big[\sup\limits_{s\in[t,T]}e^{\beta p(s-t)}|Y_{s}|^{p}+\Big(\int_{t}^{T}e^{2\beta (s-t)}|Z_{s}|^{2}ds\Big)^{\frac{p}{2}}\Big]\le C_{p}\mathbb{E}\Big[e^{\beta p(T-t)}|\xi|^{p}+\Big(\int_{t}^{T}e^{\beta (s-t)}g_{s}ds\Big)^{p}\Big].
\end{aligned}
\end{equation*}
\end{mylemma}

\end{document}